\documentclass[12pt,oneside,reqno,english]{amsart}
\usepackage{lmodern}
\usepackage[T1]{fontenc}
\usepackage[latin9]{inputenc}
\usepackage{geometry}
\usepackage{mathrsfs}
\usepackage{amsbsy}
\usepackage{amstext}
\usepackage{amsthm}
\usepackage{amssymb}
\usepackage{stmaryrd}
\usepackage{graphicx}
\usepackage{setspace}
\usepackage{xcolor}
\usepackage{url}
\usepackage[colorlinks=true,linkcolor=blue]{hyperref}

\makeatletter
\numberwithin{equation}{section}
\numberwithin{figure}{section}
\theoremstyle{plain}
\newtheorem{thm}{\protect\theoremname}[section]
\theoremstyle{remark}
\newtheorem{rem}[thm]{\protect\remarkname}
\theoremstyle{plain}
\newtheorem{prop}[thm]{\protect\propositionname}
\theoremstyle{remark}
\newtheorem{notation}[thm]{\protect\notationname}
\theoremstyle{plain}
\newtheorem{lem}[thm]{\protect\lemmaname}
\theoremstyle{remark}
\newtheorem*{acknowledgement*}{\protect\acknowledgementname}

\makeatother

\usepackage{babel}
\providecommand{\acknowledgementname}{Acknowledgement}
\providecommand{\lemmaname}{Lemma}
\providecommand{\notationname}{Notation}
\providecommand{\propositionname}{Proposition}
\providecommand{\remarkname}{Remark}
\providecommand{\theoremname}{Theorem}

\begin{document}

\title{Scaling Limit of Small Random Perturbation of Dynamical Systems}
\author{Fraydoun Rezakhanlou and Insuk Seo}
\address{Department of Mathematics \\
 University of California \\
 Berkeley, CA 94720-3840}
\email{rezakhan@math.berkeley.edu}
\address{Department of Mathematical Science\\
 Seoul National University\\
 Seoul, South Korea}
\email{insuk.seo@snu.ac.kr}
\begin{abstract}
In this article, we prove that a small random perturbation of dynamical
system with multiple stable equilibria converges to a Markov chain
whose states are neighborhoods of the deepest stable equilibria, under
a suitable time-rescaling, provided that the perturbed dynamics is
reversible in time. Such a result has been anticipated from 1970s,
when the foundation of mathematical treatment for this problem has
been established by Freidlin and Wentzell but the
process level convergence remains open for a long time.  We solve
this problem by reducing the entire analysis to an investigation of
the solution of an associated Poisson equation, and furthermore provide
a method to carry out this analysis by using well-known test functions
in a novel manner.
\end{abstract}

\maketitle

\section{\label{sec1}Introduction}

Dynamical systems that are perturbed by small random noises are known
to exhibit {\em metastable} behavior. There have been numerous
progresses in the last two decades on the rigorous verification of
metastability for a class of models that are collectively known as
{\em Small Random Perturbation of Dynamical System (SRPDS)}. In
this introductory section, we briefly review some of the existing
results on SRPDS, and describe the main contribution of this article.
We refer to a classical monograph \cite{FW1} and a recent monograph
\cite{BdH} for the comprehensive discussion on the metastable behavior
of the SRPDS.

\subsection{Small random perturbation of dynamical systems: historical review}

Consider a dynamical system given by the ordinary differential equation
in $\mathbb{R}^{d}$
\begin{equation}
d\boldsymbol{x}(t)=b(\boldsymbol{x}(t))dt\;,\label{e11}
\end{equation}
where $b:\mathbb{R}^{d}\rightarrow\mathbb{R}^{d}$ is a smooth vector
field. Suppose that this dynamical system owns multiple stable equilibria
as illustrated in Figure \ref{fig0}, and consider the random dynamical
system obtained by perturbing \eqref{e11} with a small Brownian noise.
Such a random dynamical system is defined by a stochastic differential
equation of the form
\begin{equation}
d\boldsymbol{x}_{\epsilon}(t)=b(\boldsymbol{x}_{\epsilon}(t))dt+\sqrt{2\epsilon}\,d\boldsymbol{w}_{t}\;\;;\;t\ge0\;,\label{e12}
\end{equation}
where $(\boldsymbol{w}_{t}:t\ge0)$ is the standard $d$-dimensional
Brownian motion, and $\epsilon>0$ is a small positive parameter representing
the magnitude of the noise. Suppose now that the diffusion process
$\boldsymbol{x}_{\epsilon}(t)$ starts from a neighborhood of a stable
equilibrium of the unperturbed dynamics \eqref{e11}. Then, because
of the small random noise, one can expect that the perturbed dynamics
\eqref{e12} exhibits a rare transition from this starting neighborhood
to another one around different stable equilibrium. This is a typical
metastable or tunneling transition and its quantitative analysis was
originated from Freidlin and Wentzell \cite{FW1,FW2,FW3}. However,
beyond the large-deviation type estimate that was obtained by Freidlin
and Wentzell (explained below), not much is known about the precise
nature of the metastable behavior of the model \eqref{e12}, unless
the drift $b$ is a gradient vector field. For instance, we do not
know of any sharp asymptotic for the expectation of the metastable
transition time.

\begin{figure}
\includegraphics[scale=0.3]{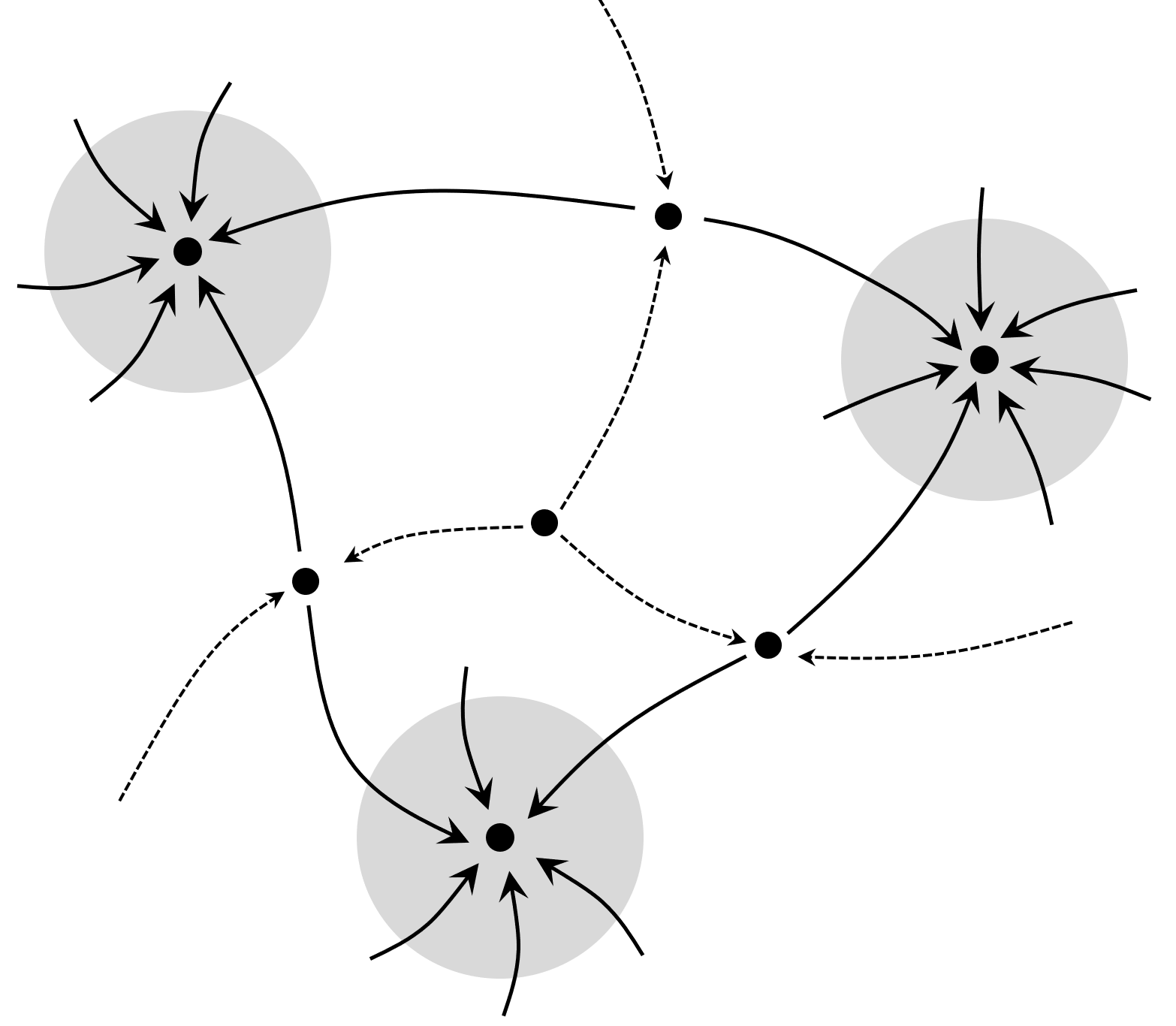}\caption{{The flow chart of the dynamical systems $d\boldsymbol{x}(t)=-b(\boldsymbol{x}(t))dt$
with three stable equilibria. There are four unstable equilibria as
well.} \label{fig0}}
\end{figure}

\subsection{Small random perturbation of dynamical systems: gradient model}

Suppose that the vector field $b$ in \eqref{e12} can be expressed
as $b=-\nabla U$, for a smooth {\em potential} function $U:\mathbb{R}^{d}\rightarrow\mathbb{R}$.
In other words, the stochastic differential equation \eqref{e12}
is of the form
\begin{equation}
d\boldsymbol{x}_{\epsilon}(t)=-\nabla U(\boldsymbol{x}_{\epsilon}(t))dt+\sqrt{2\epsilon}\,d\boldsymbol{w}_{t}\;;\;t\ge0\;.\label{e13}
\end{equation}
In particular, if the function $U(\cdot)$ has several local minima
as illustrated in Figure \ref{fig0}, then the dynamical system associated
with the unperturbed equation $d\boldsymbol{x}(t)=-\nabla U(\boldsymbol{x}(t))dt$,
has multiple stable equilibria, and hence the diffusion process $(\boldsymbol{x}_{\epsilon}(t):t\ge0)$
is destined to exhibit a metastable behavior.

In order to explain some of the classical results obtained in \cite{FW1,FW2}
by Freidlin and Wentzell in its simplest form, let us assume that
$U$ is a double-well potential. That is, the function $U$ has exactly
two local minima $\boldsymbol{m}_{1}$ and $\boldsymbol{m}_{2}$,
and a saddle point $\boldsymbol{\sigma}$ between them, as illustrated
in Figure \ref{fig1}-(left). For such a choice of $U$, the diffusion
$\boldsymbol{x}_{\epsilon}$, wanders mostly in one of the two {\em
potential wells} surrounding $\boldsymbol{m}_{1}$ and $\boldsymbol{m}_{2}$,
and occasionally makes transitions from one well to the other. To
understand the metastable nature of $\boldsymbol{x}_{\epsilon}$ qualitatively,
we analyze the asymptotic behavior of the transition time of $\boldsymbol{x}_{\epsilon}$
between the two potential wells. Writing $\tau_{\epsilon}$ for the
time that it takes for $\boldsymbol{x}_{\epsilon}(t)$ to reach a
small ball around $\boldsymbol{m}_{2}$, we wish to estimate the mean
transition time $\mathbb{E}_{\boldsymbol{m}_{1}}^{\epsilon}[\tau_{\epsilon}]$,
where $\mathbb{E}_{\boldsymbol{m}_{1}}^{\epsilon}$ denotes the expectation
with respect to the law of $\boldsymbol{x}_{\epsilon}(t)$ starting
from $\boldsymbol{m}_{1}$. Freidlin and Wentzell in \cite{FW1,FW2}
establishes a large-deviation type estimate of the form
\begin{equation}
\log\mathbb{E}_{\boldsymbol{m}_{1}}^{\epsilon}[\tau_{\epsilon}]\;\simeq\;\frac{U(\boldsymbol{\sigma})-U(\boldsymbol{m}_{1})}{\epsilon}\;\;\mbox{ as }\epsilon\to0\;.\label{e14}
\end{equation}
For the precise metastable behavior of $\boldsymbol{x}_{\epsilon}$,
we need to go beyond \eqref{e14} and evaluate the low $\epsilon$
limit of
\[
\mathbb{E}_{\boldsymbol{m}_{1}}^{\epsilon}[\tau_{\epsilon}]\ \exp\left\{ -\frac{U(\boldsymbol{\sigma})-U(\boldsymbol{m}_{1})}{\epsilon}\right\} \ .
\]
This was achieved by Bovier et. al. in \cite{BEGK2} by verifying
a classical conjecture of Eyring \cite{Ey} and Kramers \cite{Kra}.
By developing a robust methodology which is now known as {\em the
potential theoretic approach}, Bovier et. al. derive an Eyring-Kramers
type formula in the form
\begin{equation}
\mathbb{E}_{\boldsymbol{m}_{1}}^{\epsilon}[\tau_{\epsilon}]\simeq\frac{2\pi}{\lambda_{\sigma}}\,\sqrt{\frac{-\det(\nabla^{2}U)(\boldsymbol{\sigma})}{\det(\nabla^{2}U)(\boldsymbol{m}_{1})}}\,\exp\left\{ \frac{U(\boldsymbol{\sigma})-U(\boldsymbol{m}_{1})}{\epsilon}\right\} \;\;\mbox{ as }\epsilon\to0\;,\label{e15}
\end{equation}
provided that the Hessians of $U$ at $\boldsymbol{m}_{1},\,\boldsymbol{m}_{2},$
and $\boldsymbol{\sigma}$ are non-degenerate, $(\nabla^{2}U)(\boldsymbol{\sigma})$
has a unique negative eigenvalue $-\lambda_{\boldsymbol{\sigma}}$,
and some additional technical assumptions on $U$ (corresponding to
\eqref{gc} and \eqref{tc} of the current paper) are valid. It is
also verified in the same work that $\tau_{\epsilon}/\mathbb{E}_{\boldsymbol{m}_{1}}^{\epsilon}[\tau_{\epsilon}]$
converges to the mean-one exponential random variable. Similar formulas
can be derived when $U$ has multiple local minima as in Figure \ref{fig1}
(right).

\begin{figure}
\includegraphics[scale=0.19]{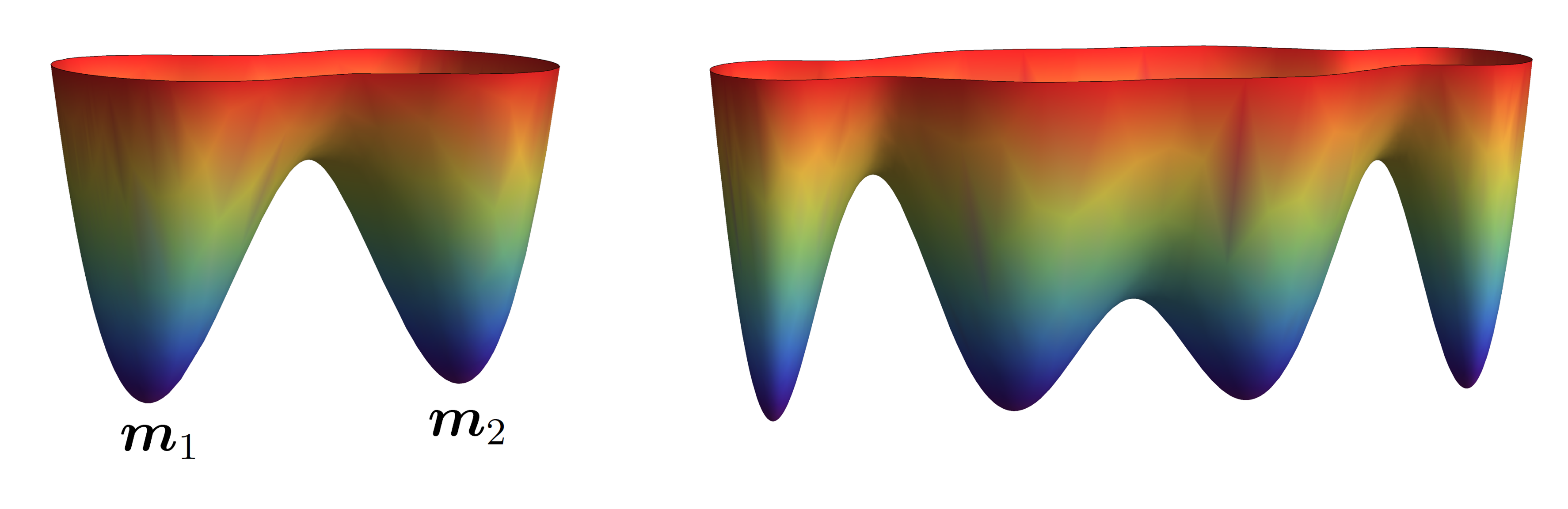}\caption{Potential $U$ with two global minima $\boldsymbol{m}_{1}$ and $\boldsymbol{m}_{2}$
(left) and multiple global minima (right). \label{fig1}}
\end{figure}

\subsection{Main result}

We starts with an informal explanation of our main result when $U$
is a double-well potential with $U(\boldsymbol{m}_{1})=U(\boldsymbol{m}_{2})$.
Heuristically speaking, the process starting from a neighborhood of
$\boldsymbol{m}_{1}$ makes a transition to that of $\boldsymbol{m}_{2}$
after an exponentially long time, as suggested by \eqref{e14}. After
spending another exponentially long time, the process makes a transition
back to the neighborhood of $\boldsymbol{m}_{1}$. These tunneling-type
transitions take place repeatedly and may be explained in terms of
a Markov chain among two valleys around $\boldsymbol{m}_{1}$ and
$\boldsymbol{m}_{2}$. More generally, if $U$ has several global
minima as in Figure \ref{fig1} (right), then the successive inter-valley
dynamics seems to be approximated by a Markov chain whose states are
the deepest valleys of $U$. In spite of the appeal of the above heuristic
description, and its consistency with \eqref{e14}, its rigorous verification
for our process \eqref{e13} was not known before. In the main result
of the current paper (Theorem \ref{main}), we show that after a rescaling
of time, a finite state Markov chain governs the inner-valley dynamics
of $\boldsymbol{x}_{\epsilon}$.

\subsection{Methodology}

The most natural way to describe the inter-valley dynamics of metastable
random processes is the reduction of the model to a continuous time
Markov process (cf. \cite{BL1,BL2,new DLLN review}). Namely, we try
to demonstrate that a suitable scaling limit of the metastable random
processes are governed by finite state Markov chains whose jump rates
are evaluated with the aid of Eyring-Kramers type formulas.

Recently, there have been numerous active researches toward this direction,
especially when the underlying metastable process lives in a discrete
space. Beltran and Landim in \cite{BL1,BL2} provide a general framework,
known as the \textit{martingale approach} to obtain the scaling limit
of metastable Markov chains. This method is quite robust and has been
applied to a wide scope of metastable processes including the condensing
zero-range processes \cite{AGL,BL3,Lan2,S}, the condensing simple
inclusion processes \cite{BDG,GRV2,Kim-Seo2}, the random walks in
potential fields \cite{LMT,LS1}, and the Potts models \cite{Kim-Seo1,LS2,NZ}.

The method of Beltran and Landim relies on a careful analysis of the
so-called trace process. A trace process is obtained from the original
process by turning off the clock when the process is not in a suitable
neighborhood of a stable equilibrium. However, as Landim pointed out
in \cite{Lan3}, it is not clear how to apply this methodology when
the underlying metastable process is a diffusion. In this paper, instead
of modifying the approach outlined in \cite{BL1,BL2}, we appeal to
an entirely new method that is a refinement of a scheme that was utilized
in \cite{ET,ST}.

We establish the metastable behavior of our diffusion $\boldsymbol{x}_{\epsilon}$
by analyzing the solutions of certain classes of Poisson equations
related to its infinitesimal generator. Theorem \ref{t51} is the
main step of our approach and will play an essential role in the proof
of our main result Theorem \ref{main}. The proof of Theorem \ref{t51}
is to some extent model-dependent, though the deduction of the main
result from this Theorem is robust and applicable to many other examples.
{For instance, the method originally developed in this
article is successfully applied to various models, e.g., the special
case of general dynamics defined by \eqref{e12} in \cite{newLS2}
and the critical zero-range processes in \cite{newLMS1}. }

{We remark that both the martingale approach and the
approach developed in the current article concern on the }\textit{{process
level convergence}}{{} in the model reduction via Markov
process. For the other sense of convergence, we refer to \cite{new DLLN review,new DLLN2}
for the approach based on the quasi-stationary distribution, and for
\cite{Su1} for the convergence in the sense of finite dimensional
marginals. 

\subsection{Non-gradient model}

As we mentioned earlier, except for the exponential estimate similar
to \eqref{e14}, the analog of \eqref{e15} is not known for the general
case \eqref{e12}. Even for \eqref{e14}, the term $U(\boldsymbol{\sigma})-U(\boldsymbol{m}_{1})$
on the right-hand side is replaced with the so-called {\em quasi-potential}
$V(\boldsymbol{\sigma};\boldsymbol{m}_{1})$. For the sake of comparison,
let us describe three simplifying features of the diffusion \eqref{e13}
that play essential roles in our work:
\begin{itemize}
\item The quasi-potential function governing the rare behaviors of the process
\eqref{e13} is given by $U$. In general, the quasi-potential $V$
is given by a variational principle in a suitable function space.
For the metastability questions, we need to study the regularity of
this quasi-potential that in general is a very delicate issue.
\item The diffusion $\boldsymbol{x}_{\epsilon}$ of the equation \eqref{e13}
admits an invariant measure with a density of the form $Z_{\epsilon}^{-1}\exp\left\{ -U/\epsilon\right\} $.
For the general case, no explicit formula for the invariant measure
is expected. The invariant measure density is specified as the unique
solution of an elliptic PDE associated with the adjoint of the generator
of \eqref{e12}.
\item The diffusion $\boldsymbol{x}_{\epsilon}$ of the equation \eqref{e13}
is reversible with respect to its invariant measure. This is no longer
the case for non-gradient models.
\end{itemize}
The main tool for proving the Eyring-Kramers formula for the gradient
model \eqref{e12} in \cite{BEGK2} is the potential theory associated
with reversible processes. Of course the special form of the invariant
measure is also critically used, and hence its extension to general
case requires non-trivial additional work. Recently, in \cite{LMS}
a potential theory for non-reversible processes is obtained, and accordingly
the Eyring-Kramers formula is extended to a class of non-reversible
diffusions with Gibbsian invariant measures.
{Moreover,
in \cite{newLS1}, the Eyring-Kramers formula for the non-gradient
model \eqref{e12} when $\boldsymbol{b}$ can be written as $\boldsymbol{b}=\nabla U+\boldsymbol{\ell}$
where $\boldsymbol{\ell}:\mathbb{R}^{d}\rightarrow\mathbb{R}^{d}$
is a smooth vector field orthogonal to $\nabla U$ and divergence-free,
i.e., $\nabla\cdot\boldsymbol{\ell}\equiv0$. These results offers
a meaningful advance to the general case.}

The current work can be regarded as an entirely new alternative approach
to the general case. Comparing to previous approaches, the main difference
of ours is the fact that we do not rely on potential theory, especially
the estimation of the capacity. Hence our approach does not rely on
the reversibility of the process $\boldsymbol{x}_{\epsilon}$. Keeping
in mind that one of main challenge of the non-reversible case is the
estimation of the capacity between valleys, the methodology adopted
in the current paper appears to be well-suited for treating non-reversible
models. This possibility is partially verified in \cite{LS3} by Landim
and an author of the current paper. In this work, the scaling limit
for the diffusion $\boldsymbol{x}_{\epsilon}$ of the equation \eqref{e12}
on a circle is obtained. It is worth mentioning that in the case of
a circle, many simplifications and explicit computations are available.
Nonetheless, the results of \cite{LS3} demonstrates that the Eyring-Kramers'
formula as well as the limiting Markov chain are very different from
the reversible case, and many peculiar features are observed.

\subsection{Related works}
{
We end the introduction with an overview of related works. As was
already explored in Bovier et al. \cite{newBGK}, the the mean transition
time $\mathbb{E}^{\epsilon}\tau_{\epsilon}$ starting from a local
minima is related to the exponentially small eigenvalues of the infinitesimal
generator $\mathcal{L}_{\epsilon}$ of the diffusion $\boldsymbol{x}_{\epsilon}$.
In particular the reciprocal of right-hand of \eqref{e15} should
serve as an asymptotic representation of the spectral gap of the operator
$\mathcal{L}_{\epsilon}$. This suggests a strategy for verifying
Eyring-Kramer formula via a Poincar\'{e} inequality for the invariant
measure of $\mathcal{L}_{\epsilon}$. This has been successfully employed
by Menz and Schlichting in \cite{newMS}. More importantly, the corresponding
\textit{logarithmic Sobolev inequality} is also valid as has been
shown in the same paper \cite{newMS}.
Indeed, the connection between the small eigenvalues of $\mathcal{L}_{\epsilon}$
to those of the corresponding \textit{Witten Laplacian} has been explored
to derive various refinements of Eyring-Kramer formula. The first
important step in this connection was taken by Helffer, Klein and
Nier \cite{newHKN1,newHKN2,newHKN3} who deduced the Eyring-Kramer
formula and WKB type asymptotic with the aid of semiclassical analysis
(see also \cite{newBerg} for an overview). Furthermore, the associated
eigenfunction can be used to build a local quasi stationary measure
as have been extensively studied by De Gesu et al. in \cite{new DLLN2}.
Most notably, a precise asymptotic analysis of the eigenfunction in
\cite{new DLLN2} leads to an exact asymptotic for the law of $\boldsymbol{x}_{\epsilon}(\tau_{\epsilon})$
(See also \cite{new DLLN review} for an overview). We also refer
to Berglund, Di Ges\`{u}, and Weber \cite{newBDW} where an Eyring-Kramers
type formula has been derived for the stochastic Allen-Cahn equation.
}
\section{\label{sec2}Model and Main result}

Our main interest in this paper is the metastable behavior of the
diffusion process \eqref{e13} when the potential function $U$ has
multiple global minima. In Section \ref{s21}, we explain basic assumptions
on $U$ and the geometric structure of its graph related to the metastable
valleys and saddle points between them. In Section \ref{s22} some
elementary results about the invariant measure of the process \eqref{e13}
is recalled. Finally, in Section \ref{s23} we describe the main result
of the paper, which is a convergence theorem for the metastable process
\eqref{e13}. We remark that the presentation and the result in the
current section are similar to a discrete counterpart model considered
in \cite{LMT}, though our proof of the main result is entirely different
from the one that is presented therein.

\subsection{\label{s21}Potential function and its landscape}

We shall consider the potential function $U:\mathbb{R}^{d}\rightarrow\mathbb{R}$
that belongs to $C^{2}(\mathbb{R}^{d})$, satisfying the growth condition
\begin{equation}
\lim_{|\boldsymbol{x}|\rightarrow\infty}\frac{U(\boldsymbol{x})}{|\boldsymbol{x}|}=\infty\;,\label{gc}
\end{equation}
and the tightness condition
\begin{equation}
\int_{\{\boldsymbol{x}:U(\boldsymbol{x})\ge a\}}e^{-U(\boldsymbol{x})/\epsilon}dx\le C_{a}e^{-a/\epsilon}\;\;\text{for all }a\in\mathbb{R\;\text{and }}\epsilon\in(0,\,1]\;,\label{tc}
\end{equation}
where $C_{a}$, $a\in\mathbb{R}$, is a constant that depends on $a$,
but not on $\epsilon$. These two conditions are required to confine
the process $\boldsymbol{x}_{\epsilon}(t)$ in a compact region with
high probability.

The metastable behavior of our model critically depends on the graphical
structures of the level sets of the potential function $U$. To guarantee
the occurrence of a metastable behavior of the type we have described
in Section 1, we need to make some standard assumptions on $U$. We
refer to Figure \ref{fig2} for the visualization of some the notations
that appear in the rest of the current section.

\begin{figure}
\includegraphics[scale=0.22]{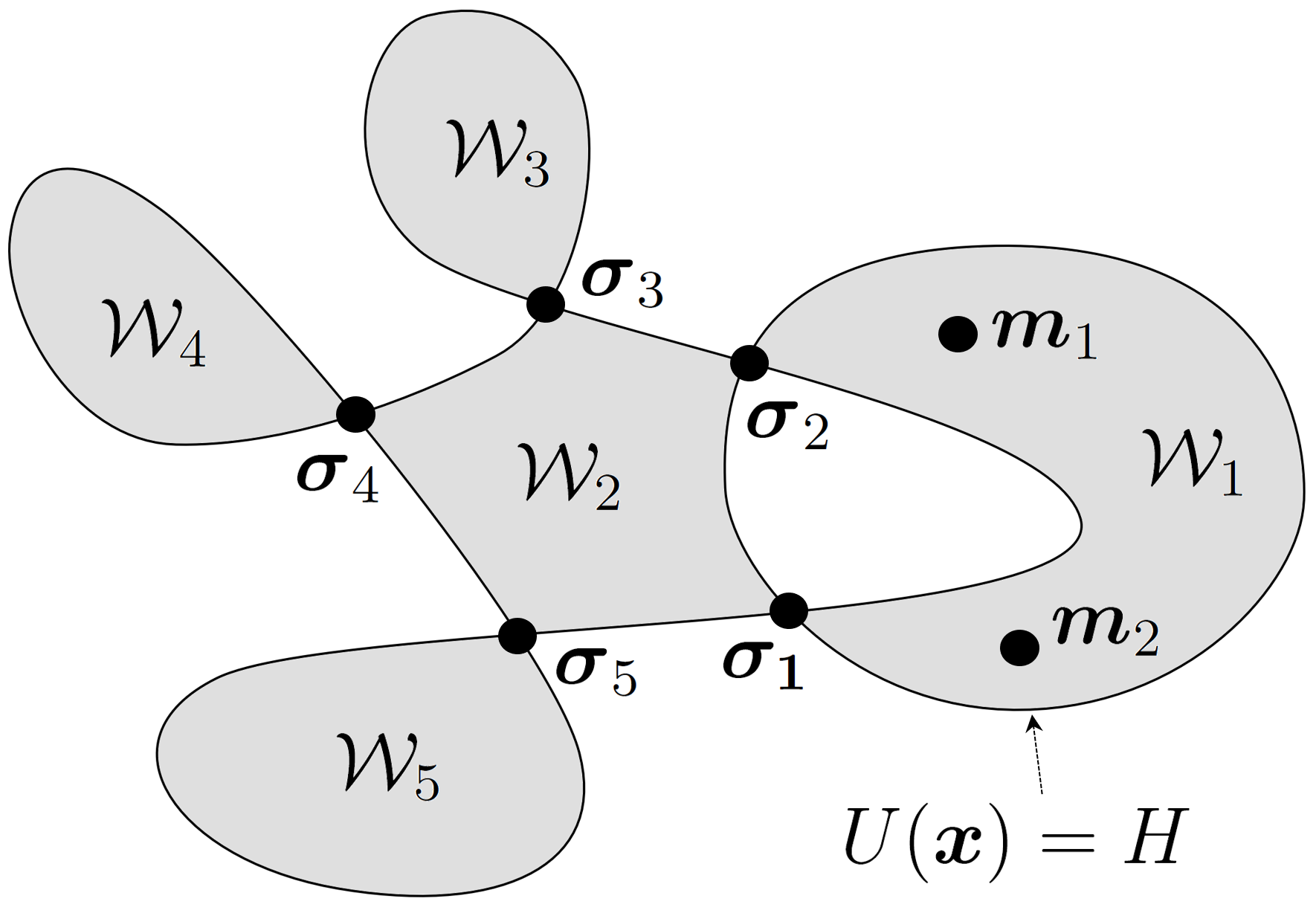}\caption{\label{fig2}Shadow area represents $\Omega$. For this case $S=\{1,\,2,\,3,\,4,\,5\}$,
$\mathcal{S}=\{\boldsymbol{\sigma}_{1},\,\boldsymbol{\sigma}_{2},\,\boldsymbol{\sigma}_{3},\,\boldsymbol{\sigma}_{4},\,\boldsymbol{\sigma}_{5}\}$,
and $\mathcal{M}_{1}=\{\boldsymbol{m}_{1},\,\boldsymbol{m}_{2}\}$. }
\end{figure}

\subsubsection{Structure of the metastable wells}

Fix $H\in\mathbb{R}$ and let $\mathcal{S}=\{\boldsymbol{\sigma}_{1},\,\boldsymbol{\sigma}_{2},\,\cdots,\,\boldsymbol{\sigma}_{L}\}$
be the set of saddle points of $U$ with height $H$, i.e.,
\[
U(\boldsymbol{\sigma}_{1})=U(\boldsymbol{\sigma}_{2})=\cdots=U(\boldsymbol{\sigma}_{L})=H\;.
\]
Denote by $\mathcal{W}_{1},\,\cdots,\,\mathcal{W}_{K}$ the connected
components of the set
\begin{equation}
\Omega=\{\boldsymbol{x}:U(\boldsymbol{x})<H\}\;.\label{omega}
\end{equation}
Let us write $S=\{1,\,2,\,\cdots,\,K\}$. By the growth condition
\eqref{gc}, all the sets $\mathcal{W}_{i}$, $i\in S$, are bounded.
We assume that $\overline{\Omega}=\cup_{i\in S}\overline{\mathcal{W}}_{i}$
is a connected set, where $\overline{\mathcal{A}}$ represents the
topological closure of the set $\mathcal{A}\subset\mathbb{R}^{d}$.

Let $h_{i}$, $i\in S$, be the minimum of the function $U$ in the
well $\mathcal{W}_{i}$. We regard $H-h_{i}$ as the depth of the
well $\mathcal{W}_{i}$. Define
\begin{equation}
h=\min_{i\in S}h_{i}\label{eh}
\end{equation}
and let
\begin{equation}
S_{\star}=\{i\in S:h_{i}=h\}\subset S\;.\label{ess}
\end{equation}
Note that the collection $\{\mathcal{W}_{i}:\,i\in S_{\star}\}$ represents
the set of \textit{deepest} wells. The purpose of the current article
is to describe the metastable behavior of the diffusion process $\boldsymbol{x}_{\epsilon}(t)$
among these deepest wells. For a non-trivial result, we assume that
$|S_{\star}|\ge2$.
\begin{rem}
When the set $\overline{\Omega}$ is not connected, we can still apply
our result to each connected component to get the metastability among
the neighborhood of this component. In order to deduce the global
result instead, one must find a larger $H$ to unify the connected
components. Because of this, our assumptions are quite general. For
the details for such a multi-scale analysis, we refer to \cite{LMT,newLS2}.
\end{rem}

\begin{rem}
{If the set $\overline{\Omega}$ is not connected and
if we selected one of them, then $h$ may not be the global minimum
of $U$ and the sets $\{\mathcal{W}_{i}:\,i\in S_{\star}\}$ may not
be the deepest wells in the landscape of $U$. Hence, the method presented
in the current article can be applied to the inter-valley dynamics
between shallow wells as well. We refer to \cite{newLS2} for more
detail. }
\end{rem}

\subsubsection{Assumptions on the critical points of $U$}

For $i\in S$, define
\[
\mathcal{M}_{i}=\{\boldsymbol{m}\in\mathcal{W}_{i}:U(\boldsymbol{m})=h_{i}\}
\]
which represents the set of minima of $U$ in the set $\mathcal{W}_{i}$.
We assume that $\mathcal{M}_{i}$ is a finite set for all $i\in S$.
Define
\begin{equation}
\mathcal{M}=\bigcup_{i\in S}\mathcal{M}_{i}\;\;\text{and\;\;}\mathcal{M}_{\star}=\bigcup_{i\in S_{\star}}\mathcal{M}_{i}\;,\label{mstar}
\end{equation}
so that the set $\mathcal{M}_{\star}$ denotes the set of global minima
of $U$. We assume that those critical points of $U$ that belong
to $\mathcal{M}_{\star}\cup\mathcal{S}$ are non-degenerate, i.e.,
the Hessian of $U$ is invertible at each point of $\mathcal{M}_{\star}\cup\mathcal{S}$.
Furthermore, we assume that the Hessian $(\nabla^{2}U)(\boldsymbol{\boldsymbol{\sigma}})$
has one negative eigenvalue and $(d-1)$ positive eigenvalues for
all $\boldsymbol{\sigma}\in\mathcal{S}$. These assumptions are standard
in the study of metastability (cf. \cite{BEGK2,LMS,LMT,LS1}). In
particular, they are satisfied if the function $U$ is a Morse function
.

\subsubsection{\label{s213}Metastable valleys}

Fix a small constant $a>0$ such that there is no critical point $\boldsymbol{c}$
of $U$ satisfying $U(\boldsymbol{c})\in[H-a,\,H)$. For $i\in S$,
denote by $\mathcal{W}_{i}^{o}$ the unique connected component of
the level set $\{\boldsymbol{x}:U(\boldsymbol{x})<H-a\}$ which is
a subset of $\mathcal{W}_{i}$. We write $\mathcal{B}(\boldsymbol{x},\,r)$
for the ball of radius $r>0$ centered at $\boldsymbol{x}\in\mathbb{R}^{d}$,
i.e.,
\begin{equation}
\mathcal{B}(\boldsymbol{x},\,r)=\{\boldsymbol{y}\in\mathbb{R}^{d}:|\boldsymbol{x}-\boldsymbol{y}|<r\}\;.\label{ball}
\end{equation}
Pick $r_{0}'$ and $r_{0}$ with $0<r_{0}<r'_{0}$. Assume that $r_{0}'$
is small enough so that the ball $\mathcal{B}(\boldsymbol{m},\,r'_{0})$
does not contain any critical points of $U$ other than $\boldsymbol{m}$,
and $\mathcal{B}(\boldsymbol{m},\,r'_{0})\subset\bigcup_{i\in S}\mathcal{W}_{i}^{o}$
for all $\boldsymbol{m}\in\mathcal{M}$. For $i\in S$, the metastable
valley corresponding to the well $\mathcal{W}_{i}$ is defined by
\begin{equation}
\mathcal{V}_{i}=\bigcup_{\boldsymbol{m}\in\mathcal{M}_{i}}\mathcal{B}(\boldsymbol{m},\,r_{0})\;.\label{ev}
\end{equation}
For our purposes, we need to consider a larger valley
\begin{equation}
\mathcal{V}'_{i}=\bigcup_{\boldsymbol{m}\in\mathcal{M}_{i}}\mathcal{B}(\boldsymbol{m},\,r'_{0})\;.\label{ev'}
\end{equation}
Finally, we write
\begin{equation}
\mathcal{V}_{\star}=\bigcup_{i\in S_{\star}}\mathcal{V}_{i}\;,\;\;\text{and\;\;}\Delta=\mathbb{R}^{d}\setminus\mathcal{V}_{\star}\;.\label{twoset}
\end{equation}

\subsection{\label{s22}Invariant measure}

The generator corresponding to the diffusion process $\boldsymbol{x}_{\epsilon}(t)$
of the equation \eqref{e13}, can be written as
\[
\mathscr{L}_{\epsilon}=\epsilon\,\Delta-\nabla U\cdot\nabla=\epsilon\,e^{U(\boldsymbol{x})/\epsilon}\,\nabla\cdot\left[e^{-U(\boldsymbol{x})/\epsilon}\nabla\right]\;.
\]
From this, it is not hard to show that the invariant measure for the
process $\boldsymbol{x}_{\epsilon}(\cdot)$ is given by
\begin{equation}
\mu_{\epsilon}(d\boldsymbol{x})=Z_{\epsilon}^{-1}e^{-U(\boldsymbol{x})/\epsilon}d\boldsymbol{x}:=\hat{\mu}_{\epsilon}(\boldsymbol{x})\ d\boldsymbol{x}\label{mu}
\end{equation}
where $Z_{\epsilon}$ is the partition function defined by
\[
Z_{\epsilon}=\int_{\mathbb{R}^{d}}e^{-U(\boldsymbol{x})/\epsilon}d\boldsymbol{x}<\infty\;.
\]
Notice that $Z_{\epsilon}$ is finite because of \eqref{tc}. Define
\begin{equation}
\nu_{i}=\sum_{\boldsymbol{m}\in\mathcal{M}_{i}}\frac{1}{\sqrt{\det(\nabla^{2}U)(\boldsymbol{m})}}\;\text{\;for}\;i\in S_{\star}\;\;\text{and\;\;}\nu_{\star}=\sum_{j\in S_{\star}}\nu_{j}\;.\label{nu}
\end{equation}
We state some asymptotic results for the partition function $Z_{\epsilon}$
and the invariant measure $\mu_{\epsilon}(\cdot)$. We write $o_{\epsilon}(1)$
for a term that vanishes as $\epsilon\rightarrow0$.
\begin{prop}
\label{p21}It holds that
\begin{align}
 & Z_{\epsilon}=(1+o_{\epsilon}(1))\,(2\pi\epsilon)^{d/2}\,e^{-h/\epsilon}\,\nu_{\star}\;,\label{ep3}\\
 & \mu_{\epsilon}(\mathcal{V}_{i})=(1+o_{\epsilon}(1))\,\frac{\nu_{i}}{\nu_{\star}}\;\;\text{for }i\in S_{\star}\;,\label{ep4}\\
 & \mu_{\epsilon}(\mathcal{V}'_{i})=(1+o_{\epsilon}(1))\,\frac{\nu_{i}}{\nu_{\star}}\;\;\text{for }i\in S_{\star}\;,\label{ep4'}\\
 & \mu_{\epsilon}(\Delta)=o_{\epsilon}(1)\;.\label{ep4''}
\end{align}
\end{prop}

\begin{proof}
By Laplace's method, we can deduce that, for $i\in S_{\star}$,
\begin{align}
\mu_{\epsilon}(\mathcal{V}_{i}) & =Z_{\epsilon}^{-1}\,(1+o_{\epsilon}(1))\,(2\pi\epsilon)^{d/2}\,e^{-h/\epsilon}\,\nu_{i}\;.\label{ep1}\\
\mu_{\epsilon}(\mathcal{V}'_{i}) & =Z_{\epsilon}^{-1}\,(1+o_{\epsilon}(1))\,(2\pi\epsilon)^{d/2}\,e^{-h/\epsilon}\,\nu_{i}\;.\label{ep1'}
\end{align}
On the other hand, by \eqref{tc}, we have
\begin{equation}
\mu_{\epsilon}(\Delta)=Z_{\epsilon}^{-1}\,o_{\epsilon}(1)\,\epsilon^{d/2}\,e^{-h/\epsilon}\;.\label{ep2}
\end{equation}
Now, \eqref{ep3} follows from \eqref{ep1} and \eqref{ep2} because
\[
1=\mu_{\epsilon}(\Delta)+\sum_{i\in S_{\star}}\mu_{\epsilon}(\mathcal{V}_{i})\;.
\]
Moreoever, \eqref{ep4}, \eqref{ep4'} and \eqref{ep4''} are obtained
by inserting \eqref{ep3} into \eqref{ep1}, \eqref{ep1'}, and \eqref{ep2},
respectively.
\end{proof}

\subsection{\label{s23}Main result}

The metastable behavior of the process $\boldsymbol{x}_{\epsilon}(t)$
is a consequence of its convergence to a Markov chain $\mathbf{y}(t)$
on $S_{\star}$ in a proper sense, as is explained in Section \ref{s233}
below. The Markov chain $\mathbf{y}(t)$ is defined in Section \ref{s232},
based on an auxiliary Markov chain $\mathbf{x}(t)$ on $S$ that is
introduced below.

\subsubsection{\label{s231}Markov chain $\mathbf{x}(t)$ on $S$}

For a saddle point $\boldsymbol{\sigma}\in\mathcal{S}$, we write
$-\lambda_{\boldsymbol{\sigma}}$ for the unique negative eigenvalue
of the Hessian $(\nabla^{2}U)(\boldsymbol{\sigma})$, and define
\[
\omega_{\boldsymbol{\sigma}}=\frac{\lambda_{\boldsymbol{\sigma}}}{2\pi\sqrt{-\det(\nabla^{2}U)(\boldsymbol{\sigma})}}\;.
\]
For distinct $i,\,j\in S$, let $\mathcal{S}_{i,\,j}$ be the set
of saddle points between wells $\mathcal{W}_{i}$ and $\mathcal{W}_{j}$
in the sense that
\[
\mathcal{S}_{i,\,j}=\overline{\mathcal{W}}_{i}\cap\overline{\mathcal{W}}_{j}\subset\mathcal{S}\;.
\]
Define
\[
\omega_{i,\,j}=\sum_{\boldsymbol{\sigma}\in\mathcal{S}_{i,\,j}}\omega_{\boldsymbol{\sigma}}\;.
\]
For convenience, we set $\omega_{i,\,i}=0$ for all $i\in S$. For
$i\in S$, we define
\[
\omega_{i}=\sum_{j\in S}\omega_{i,\,j}\;\;\text{and\;\;}\mu(i)=\omega_{i}/(\sum_{j\in S}\omega_{j})\;.
\]
We have $\omega_{i}>0$ since the set $\overline{\Omega}$ is connected
by our assumption. Denote by $\{\mathbf{x}(t):t\ge0\}$ the continuous
time Markov chain on $S$ whose jump rate from $i\in S$ to $j\in S$
is given by $\omega_{i,\,j}/\mu(i)$. For $i\in S$, denote by $\mathbf{P}_{i}$
the law of the Markov chain $\mathbf{x}(t)$ starting from $i$. Notice
that this Markov chain is reversible with respect to the probability
measure $\mu(\cdot)$. The generator $L_{\mathbf{x}}$ corresponding
to the chain $\mathbf{x}(t)$ can be written as,
\[
(L_{\mathbf{x}}\mathbf{f})(i)=\sum_{j\in S}\frac{\omega_{i,\,j}}{\mu(i)}\left[\mathbf{f}(j)-\mathbf{f}(i)\right]\;\;;\;i\in S\;,
\]
for $\mathbf{f}\in\mathbb{R}^{S}$. Define, for $\mathbf{f}$, $\mathbf{g}\in\mathbb{R}^{S}$,
\begin{equation}
D_{\mathbf{x}}(\mathbf{f},\,\mathbf{g})=\sum_{i\in S}\mu(i)\,\mathbf{f}(i)(-L_{\mathbf{x}}\mathbf{g})(i)=\frac{1}{2}\sum_{i,\,j\in S}\omega_{i,\,j}\left[\mathbf{f}(j)-\mathbf{f}(i)\right]\left[\mathbf{g}(j)-\mathbf{g}(i)\right]\;.\label{dxfg}
\end{equation}
Then, $D_{\mathbf{x}}(\mathbf{f},\,\mathbf{f})$ represents the Dirichlet
form associated with the chain $\mathbf{x}(t)$.

Now we define the equilibrium potential and the capacity corresponding
to the chain $\mathbf{x}(t)$. For $A\subseteq S$, denote by $H_{A}$
the hitting time of the set $A$, i.e., $H_{A}=\inf\{t\ge0:\mathbf{x}(t)\in A\}$.
For two non-empty disjoint subsets $A$ and $B$ of $S$, define a
function $\mathbf{h}_{A,\,B}:S\rightarrow[0,\,1]$ by
\begin{equation}
\mathbf{h}_{A,\,B}(i)=\mathbf{P}_{i}(H_{A}<H_{B})\;.\label{epx}
\end{equation}
The function $\mathbf{h}_{A,\,B}$ is called the equilibrium potential
between two sets $A$ and $B$ with respect to the Markov chain $\mathbf{x}(t)$.
One of the notable fact about the equilibrium potential is that, $\mathbf{h}_{A,\,B}$
can be characterized as the unique solution of the following equation:
\begin{equation}
\begin{cases}
(L_{\mathbf{x}}\mathbf{h}_{A,\,B})(i)=0 & \text{ for all }i\in(A\cup B)^{c}\;,\\
\mathbf{h}_{A,\,B}(a)=1 & \text{ for all }a\in A\;,\\
\mathbf{h}_{A,\,B}(b)=0 & \text{ for all }b\in B\;.
\end{cases}\label{epx1}
\end{equation}
The capacity between these two sets $A$ and $B$ is now defined as
\[
\textup{cap}_{\mathbf{x}}(A,\,B)=D_{\mathbf{x}}(\mathbf{h}_{A,\,B},\,\mathbf{h}_{A,\,B})\;.
\]

\subsubsection{\label{s232}Markov chain $\mathbf{y}(t)$ on $S_{\star}$}

For distinct $i,\,j\in S_{\star}$, define
\begin{equation}
\beta_{i,\,j}=\frac{1}{2}\left[\textup{cap}_{\mathbf{x}}(\{i\},\,S_{\star}\setminus\{i\})+\textup{cap}_{\mathbf{x}}(\{j\},\,S_{\star}\setminus\{j\})-\textup{cap}_{\mathbf{x}}(\{i,\,j\},\,S_{\star}\setminus\{i,\,j\})\right]\label{beta}
\end{equation}
and set $\beta_{i,\,i}=0$ for all $i\in S_{\star}$. Note that $\beta_{i,\,j}=\beta_{j,\,i}$
for all $i,\,j\in S$. Recall $\nu_{i}$ from \eqref{nu} and let
$\{\mathbf{y}(t):t\ge0\}$ be a continuous time Markov chain on $S_{\star}$
whose jump rate from $i\in S_{\star}$ to $j\in S_{\star}$ is given
by $\beta_{i,\,j}/\nu_{i}$. Denote by $\mathbf{Q}_{i}$, $i\in S_{\star}$,
the law of Markov chain $\mathbf{y}(t)$ starting from $i$. Notice
that the probability measure $\mu_{\star}$ on $S_{\star}$, defined
by
\begin{equation}
\mu_{\star}(i)=\frac{\nu_{i}}{\nu_{\star}}\;\text{\;\;for}\;i\in S_{\star}\label{inv}
\end{equation}
is the invariant measure for the Markov chain $\mathbf{y}(t)$. For
$\mathbf{f}\in\mathbb{R}^{S_{\star}},$ the generator $L_{\mathbf{y}}$
corresponding to the Markov chain $\mathbf{y}(t)$ is given by
\[
(L_{\mathbf{y}}\mathbf{f})(i)=\sum_{j\in S_{\star}}\frac{\beta_{i,\,j}}{\nu_{i}}\left[\mathbf{f}(j)-\mathbf{f}(i)\right]\;\;;\;i\in S_{\star}\;.
\]
Similar to \eqref{dxfg}, we define, for $\mathbf{f}$, $\mathbf{g}\in\mathbb{R}^{S_{\star}}$,

\begin{align*}
D_{\mathbf{y}}(\mathbf{f},\,\mathbf{g}) & =\sum_{i\in S}\frac{\nu_{i}}{\nu_{\star}}\mathbf{f}(i)(-L_{\mathbf{y}}\mathbf{g})(i)=\frac{1}{2\nu_{\star}}\sum_{i,\,j\in S_{\star}}\beta_{i,\,j}\left[\mathbf{f}(j)-\mathbf{f}(i)\right]\left[\mathbf{g}(j)-\mathbf{g}(i)\right]\;.
\end{align*}
We acknowledge here that a similar construction has been carried out
in \cite{Su1} at which a sharp asymptotics of the low-lying spectra
of the metastable diffusions on $\sigma$-compact Riemannian manifold
has been carried out for special form of the potential function $U$.

\subsubsection{\label{s233}Main result}

It is anticipated from \eqref{e15} that the time scale corresponding
to the metastable transition is given by
\begin{equation}
\theta_{\epsilon}=e^{(H-h)/\epsilon}\;.\label{theta}
\end{equation}
Define the rescaled process $\{\widehat{\boldsymbol{x}}_{\epsilon}(t):t\ge0\}$
as of $\boldsymbol{x}_{\epsilon}(t)$
\[
\widehat{\boldsymbol{x}}_{\epsilon}(t)=\boldsymbol{x}_{\epsilon}(\theta_{\epsilon}t)\;.
\]
We now define the trace process $\boldsymbol{y}^{\epsilon}(t)$ of
$\widehat{\boldsymbol{x}}^{\epsilon}(t)$ inside $\mathcal{V}_{\star}$.
To this end, define the total time spent by $\left(\widehat{\boldsymbol{x}}_{\epsilon}(s):\ s\in[0,t]\right)$
in the valley $\mathcal{V}_{\star}$ as
\[
T^{\epsilon}(t)=\int_{0}^{t}\chi_{\mathcal{V}_{\star}}(\widehat{\boldsymbol{x}}_{\epsilon}(s))ds\;\;;\;t\ge0\;,
\]
where the function $\chi_{\mathcal{A}}:\mathbb{R}^{d}\rightarrow\{0,\,1\}$
represents the characteristic function of $\mathcal{A}\subseteq\mathbb{R}^{d}$.
Then, define
\begin{equation}
S^{\epsilon}(t)=\sup\{s\ge0:T^{\epsilon}(s)\le t\}\;\;;\;t\ge0\;,\label{Se}
\end{equation}
which is the generalized inverse of the increasing function $T^{\epsilon}(\cdot)$.
Finally, the trace process of $\widehat{\boldsymbol{x}}_{\epsilon}(t)$
in the set $\mathcal{V}_{\star}$ is defined by
\begin{equation}
\boldsymbol{y}_{\epsilon}(t)=\widehat{\boldsymbol{x}}_{\epsilon}(S^{\epsilon}(t))\;\;;\;t\ge0\;.\label{e241}
\end{equation}
One can readily verify that $\boldsymbol{y}_{\epsilon}(t)\in\mathcal{V}_{\star}$
for all $t\ge0$. Define a projection function $\Psi:\mathcal{V}_{\star}\rightarrow S_{\star}$
by
\begin{equation}
\Psi(\boldsymbol{x})=\sum_{i\in S_{\star}}i\,\chi_{\mathcal{V}_{i}}(\boldsymbol{x)}\;.\label{proj}
\end{equation}
Since $\boldsymbol{y}_{\epsilon}(t)$ is always in the set $\mathcal{V}_{\star}$,
the following process is well-defined:
\begin{equation}
\mathbf{y}_{\epsilon}(t)=\Psi(\boldsymbol{y}_{\epsilon}(t))\;\;;\;t\ge0\;.\label{e242}
\end{equation}
The process $\mathbf{y}_{\epsilon}(t)$ represents the index of the
valley in which the process $\boldsymbol{y}_{\epsilon}(t)$ is residing.
Denote by $\mathbb{P}_{\boldsymbol{x}}^{\epsilon}$ and $\mathbb{\widehat{P}}_{\boldsymbol{x}}^{\epsilon}$
the law of processes $\boldsymbol{x}_{\epsilon}(\cdot)$ and $\widehat{\boldsymbol{x}}_{\epsilon}(\cdot)$
starting from $\boldsymbol{x}\in\mathbb{R}^{d}$, respectively, and
denote by $\mathbb{E}_{\boldsymbol{x}}^{\epsilon}$ and $\mathbb{\widehat{E}}_{\boldsymbol{x}}^{\epsilon}$
the corresponding expectations. For $\boldsymbol{x}\in\mathcal{V}_{\star}$,
denote by $\mathbf{Q}_{\boldsymbol{x}}^{\epsilon}$ the law of process
$\mathbf{y}^{\epsilon}(\cdot)$ when the underlying diffusion process
$\boldsymbol{x}_{\epsilon}(t)$ follows $\mathbb{P}_{\boldsymbol{x}}^{\epsilon}$,
i.e.,
\[
\mathbf{Q}_{\boldsymbol{x}}^{\epsilon}=\mathbb{\widehat{P}}_{\boldsymbol{x}}^{\epsilon}\circ\Psi^{-1}\;.
\]
For any Borel probability measure $\pi$ on $\mathcal{V}_{\star}$,
we denote by $\mathbb{P}_{\pi}^{\epsilon}$ the law of process $\boldsymbol{x}_{\epsilon}(\cdot)$
with initial distribution $\pi$. Then, define $\mathbb{\widehat{P}}_{\pi}^{\epsilon}$,
$\mathbb{E}_{\pi}^{\epsilon}$, $\mathbb{\widehat{E}}_{\pi}^{\epsilon}$,
and $\mathbf{Q}_{\pi}^{\epsilon}$ similarly as above. We are now
ready to state the main result of this article:
\begin{thm}
\label{main}For all $i\in S_{\star}$ and for any sequence of Borel
probability measures $(\pi_{\epsilon})_{\epsilon>0}$ concentrated
on $\mathcal{V}_{i}$, the sequence of probability laws $(\mathbf{Q}_{\pi_{\epsilon}}^{\epsilon})_{\epsilon>0}$
converges to $\mathbf{Q}_{i}$, the law of the Markov process $(\mathbf{y}(t))_{t\ge0}$
starting from $i$, as $\epsilon$ tends to $0$.
\end{thm}

We finish this section by explaining the organization of the rest
of the paper. In Section \ref{sec3}, we construct a class of test
functions which are useful in some of the computations we carry out
in Section \ref{sec5}. In Section \ref{sec5}, we analyze a Poisson
equation that will play a crucial role in the proof of both the \textit{tightness}
in Section \ref{sec4}, and the \textit{uniqueness of the limit point}
in Section \ref{sec6}. These two ingredients complete the proof of
the convergence result stated in Theorem \ref{main}, as we will demonstrate
in Section \ref{sec6}.

\section{\label{sec3}Test functions}

The purpose of the current section is to construct some test functions.
We acknowledge that these functions are not new; similar functions
have already been used in \cite{BEGK1} and \cite{LMS} in order to
obtain sharp estimates on the capacity associated with pairs of valleys.
Hence we refer to those papers for some proofs. We also remark here
that the way we utilize these test functions will be entirely different
from how they are used in \cite{BEGK1} and \cite{LMS}. We use these
functions to estimate the value of a solution of our {\em Poisson
Problem} in each valley (see Theorem \ref{t51}).

\subsection{\label{s31}Neighborhoods of saddle points}

We now introduce some subsets of $\mathbb{R}^{d}$ related to the
inter-valley structure of $U$. For each saddle point $\boldsymbol{\sigma}\in\mathcal{S}$,
denote by $-\lambda_{1}^{\boldsymbol{\sigma}}$ the unique negative
eigenvalue of $(\nabla^{2}U)(\boldsymbol{\boldsymbol{\sigma}})$,
and by $\lambda_{2}^{\boldsymbol{\sigma}},\,\cdots\,\,\lambda_{d}^{\boldsymbol{\sigma}}$
the positive eigenvalues of $(\nabla^{2}U)(\boldsymbol{\boldsymbol{\sigma}})$.
We choose unit eigenvectors $\boldsymbol{v}_{1}^{\boldsymbol{\sigma}},\dots\boldsymbol{v}_{d}^{\boldsymbol{\sigma}}$
of $(\nabla^{2}U)(\boldsymbol{\boldsymbol{\sigma}})$ corresponding
to the eigenvalues $-\lambda_{1}^{\boldsymbol{\sigma}},\lambda_{2}^{\boldsymbol{\sigma}},\dots,\lambda_{d}^{\boldsymbol{\sigma}}$.
\begin{rem}
\label{remd}Some care is needed as we select the direction of $\boldsymbol{v}_{1}^{\boldsymbol{\sigma}}$.
If $\boldsymbol{\sigma}\in\mathcal{S}_{i,\,j}$ for some $i<j$, we
choose $\boldsymbol{v}_{1}^{\boldsymbol{\sigma}}$ to be directed
toward the valley $\mathcal{W}_{j}$. Formally stating, we assume
that $\boldsymbol{\sigma}+\alpha\boldsymbol{v}_{1}^{\boldsymbol{\sigma}}\in\mathcal{W}_{j}$
for all sufficiently small $\alpha>0$.
\end{rem}

We define
\begin{equation}
\delta=\delta(\epsilon)=\sqrt{\epsilon\log(1/\epsilon)}\;.\label{edelta}
\end{equation}
A closed box $\ensuremath{\mathcal{C}_{\boldsymbol{\sigma}}^{\epsilon}}$
around the saddle point $\boldsymbol{\sigma}$ is defined by
\[
\ensuremath{\mathcal{C}_{\boldsymbol{\sigma}}^{\epsilon}}=\left\{ \boldsymbol{\sigma}+\sum_{i=1}^{d}\alpha_{i}\boldsymbol{v}_{i}^{\boldsymbol{\sigma}}:\alpha_{1}\in\left[-\frac{J\delta}{\sqrt{\lambda_{1}^{\boldsymbol{\sigma}}}},\,\frac{J\delta}{\sqrt{\lambda_{1}^{\boldsymbol{\sigma}}}}\right]\text{ and }\alpha_{i}\in\left[-\frac{2J\delta}{\sqrt{\lambda_{i}^{\boldsymbol{\sigma}}}},\,\frac{2J\delta}{\sqrt{\lambda_{i}^{\boldsymbol{\sigma}}}}\right]\text{ for }2\le i\le d\,\right\} \;,
\]
where $J$ is a constant which is larger than $2^{1/2}$\textbf{ }(cf.
\eqref{e4101}). We refer to Figure \ref{fig3} for the illustration
of the sets defined in this subsection.

\begin{figure}
\includegraphics[scale=0.21]{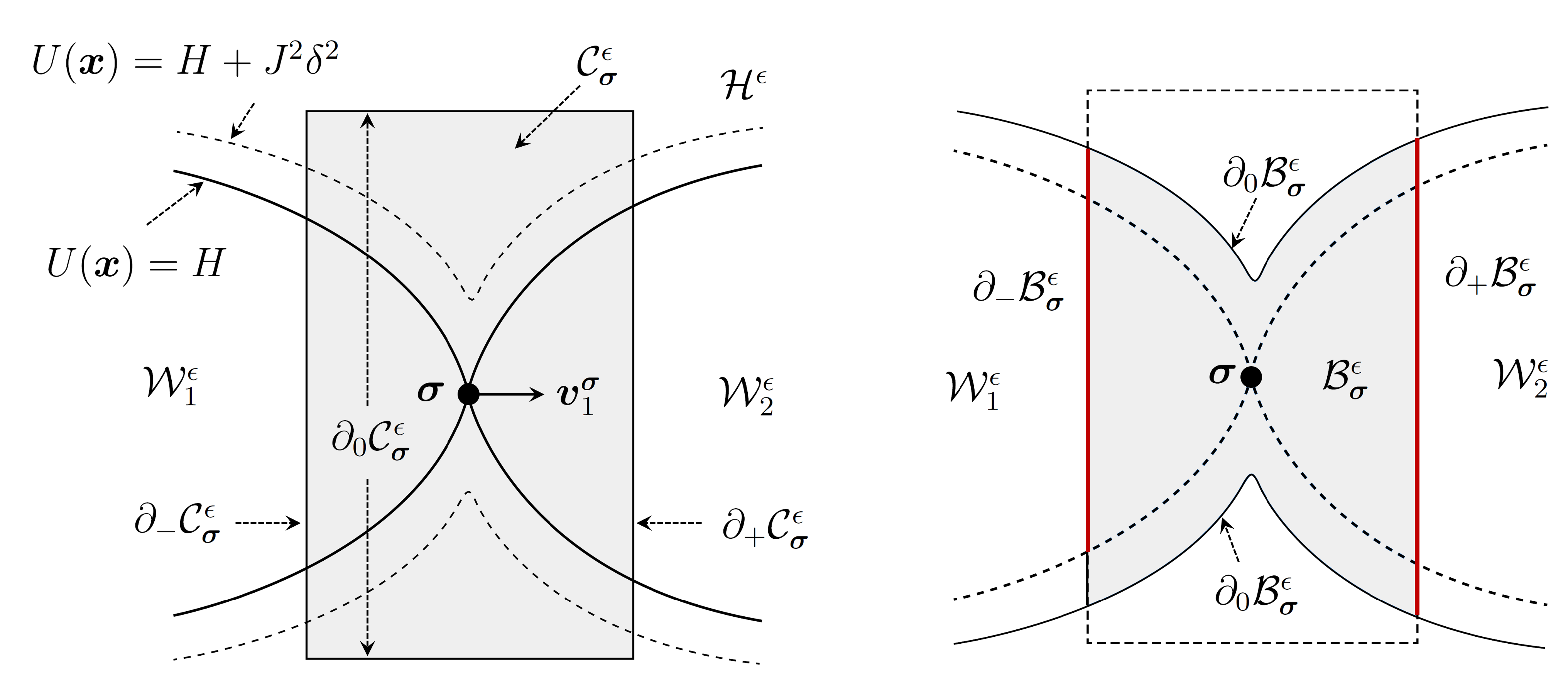} \caption{Visualization of a neighborhood of $\boldsymbol{\sigma}\in\mathcal{S}_{1,\,2}$.\label{fig3} }
\end{figure}

\begin{notation}
We summarize the notations used in the remaining of the paper. We
regard $J$ as a constant so that the terms like $o_{\epsilon}(1)$,
$O(\delta^{2})$ may depend on $J$ as well. All the constants without
subscript or superscript $\epsilon$ are independent of $\epsilon$
(and hence of $\delta$) but may depend on $J$ or the function $U$.
Constants are usually denoted by $c$ or $C$ and different appearances
may take different values.
\end{notation}

Decompose the boundary $\partial\ensuremath{\mathcal{C}_{\boldsymbol{\sigma}}^{\epsilon}}$
into
\begin{align*}
 & \partial_{+}\mathcal{C}_{\boldsymbol{\sigma}}^{\epsilon}=\left\{ \boldsymbol{\sigma}+\sum_{i=1}^{d}\alpha_{i}\boldsymbol{v}_{i}^{\boldsymbol{\sigma}}\in\ensuremath{\mathcal{C}_{\boldsymbol{\sigma}}^{\epsilon}}:\alpha_{1}=\frac{J\delta}{\sqrt{\lambda_{1}^{\boldsymbol{\sigma}}}}\,\right\} \;,\\
 & \partial_{-}\mathcal{C}_{\boldsymbol{\sigma}}^{\epsilon}=\left\{ \boldsymbol{\sigma}+\sum_{i=1}^{d}\alpha_{i}\boldsymbol{v}_{i}^{\boldsymbol{\sigma}}\in\ensuremath{\mathcal{C}_{\boldsymbol{\sigma}}^{\epsilon}}:\alpha_{1}=-\frac{J\delta}{\sqrt{\lambda_{1}^{\boldsymbol{\sigma}}}}\,\right\} \;,\;\;\text{and}\;\;\partial_{0}\mathcal{C}_{\boldsymbol{\sigma}}^{\epsilon}=\partial\ensuremath{\mathcal{C}_{\epsilon}^{\boldsymbol{\sigma}}}\setminus(\partial_{+}\ensuremath{\mathcal{C}_{\boldsymbol{\sigma}}^{\epsilon}}\cup\partial_{-}\mathcal{C}_{\boldsymbol{\sigma}}^{\epsilon})\;.
\end{align*}
The following is a direct consequence of a Taylor expansion of $U$
around $\boldsymbol{\sigma}$, since $U(\boldsymbol{\sigma})=H$.
\begin{lem}
\label{lem31}For all $\boldsymbol{x}\in\partial_{0}\mathcal{C}_{\boldsymbol{\sigma}}^{\epsilon}$,
we have that
\[
U(\boldsymbol{x})\ge H+(1+o_{\epsilon}(1))\,\frac{3J^{2}\delta^{2}}{2}\;.
\]
\end{lem}

\begin{proof}
This follows from the Taylor expansion of $U$ at $\sigma$ (see \cite[Lemma 6.1]{LMS}).
\end{proof}
Now we define
\[
\mathcal{H}^{\epsilon}=\left\{ \boldsymbol{x}\in\mathbb{R}^{d}:U(\boldsymbol{x})\le H+J^{2}\delta^{2}\right\} \;,
\]
and let $\mathcal{B}_{\boldsymbol{\sigma}}^{\epsilon}=\mathcal{C}_{\boldsymbol{\sigma}}^{\epsilon}\cap\mathcal{H}^{\epsilon}$
for $\boldsymbol{\sigma}\in\mathcal{S}$. Decompose the boundary $\partial\mathcal{B}_{\boldsymbol{\sigma}}^{\epsilon}$
as
\begin{equation}
\partial_{+}\mathcal{B}_{\boldsymbol{\sigma}}^{\epsilon}=\partial\mathcal{B}_{\boldsymbol{\sigma}}^{\epsilon}\cap\partial_{+}\mathcal{C}_{\boldsymbol{\sigma}}^{\epsilon}\;,\;\;\partial_{-}\mathcal{B}_{\boldsymbol{\sigma}}^{\epsilon}=\partial\mathcal{B}_{\boldsymbol{\sigma}}^{\epsilon}\cap\partial_{-}\mathcal{C}_{\boldsymbol{\sigma}}^{\epsilon}\;,\;\text{and}\;\;\partial_{0}\mathcal{B}_{\boldsymbol{\sigma}}^{\epsilon}=\partial\mathcal{B}_{\boldsymbol{\sigma}}^{\epsilon}\setminus(\partial_{+}\mathcal{B}_{\boldsymbol{\sigma}}^{\epsilon}\cup\partial_{-}\mathcal{B}_{\boldsymbol{\sigma}}^{\epsilon})\;.\label{eqb}
\end{equation}
Then, by Lemma \ref{lem31}, for small enough $\epsilon$, we have
\begin{equation}
U(\boldsymbol{x})=H+J^{2}\delta^{2}\text{ for all }\boldsymbol{x}\in\partial_{0}\mathcal{B}_{\epsilon}^{\boldsymbol{\sigma}}\;.\label{b0}
\end{equation}
Thus, the set $\mathcal{H}^{\epsilon}\setminus\bigcup_{\boldsymbol{\sigma}\in\mathcal{S}}\mathcal{B}_{\boldsymbol{\sigma}}^{\epsilon}$
consists of $K$ connected components $\mathcal{W}_{1}^{\epsilon}$,
$\cdots,\,\mathcal{W}_{K}^{\epsilon}$ such that $\mathcal{V}'_{i}\subset\mathcal{W}_{i}^{\epsilon}$
for all $i\in S$. Furthermore, if $\boldsymbol{\sigma}\in\mathcal{S}_{i,\,j}$
with $i<j$, then by Remark \ref{remd} we have that
\begin{equation}
\partial_{-}\mathcal{B}_{\boldsymbol{\sigma}}^{\epsilon}\subset\partial\mathcal{\mathcal{W}}_{i}^{\epsilon}\text{ \;and\; }\partial_{+}\mathcal{B}_{\boldsymbol{\sigma}}^{\epsilon}\subset\partial\mathcal{\mathcal{W}}_{j}^{\epsilon}\;.\label{bb1}
\end{equation}
We shall assume from now on that $\epsilon>0$ is small enough so
that the construction above is in force.

\subsection{\label{s32}Test function and basic estimates}

For $\boldsymbol{\sigma}\in\mathcal{S}$, define a normalizing constant
$c_{\epsilon}^{\boldsymbol{\sigma}}$ by
\begin{equation}
c_{\epsilon}^{\boldsymbol{\sigma}}=\int_{-J\delta/\sqrt{\lambda_{1}^{\boldsymbol{\sigma}}}}^{J\delta/\sqrt{\lambda_{1}^{\boldsymbol{\sigma}}}}\ \sqrt{\frac{\lambda_{1}^{\boldsymbol{\sigma}}}{2\pi\epsilon}}\exp\left\{ -\frac{\lambda_{1}^{\boldsymbol{\sigma}}}{2\epsilon}t^{2}\right\} \ dt=1+o_{\epsilon}(1)\;,\label{ce}
\end{equation}
and define a function $f_{\boldsymbol{\sigma}}^{\epsilon}(\cdot)$
on $\mathcal{B}_{\boldsymbol{\sigma}}^{\epsilon}$ by,
\begin{equation}
f_{\epsilon}^{\boldsymbol{\sigma}}(\boldsymbol{x})=(c_{\epsilon}^{\boldsymbol{\sigma}})^{-1}\int_{-J\delta/\sqrt{\lambda_{1}^{\boldsymbol{\sigma}}}}^{(\boldsymbol{x}-\boldsymbol{\sigma})\cdot\boldsymbol{v}_{1}^{\boldsymbol{\sigma}}}\sqrt{\frac{\lambda_{1}^{\boldsymbol{\sigma}}}{2\pi\epsilon}}\exp\left\{ -\frac{\lambda_{1}^{\boldsymbol{\sigma}}}{2\epsilon}t^{2}\right\} dt\;\;;\;\boldsymbol{x}\in\mathcal{B}_{\boldsymbol{\sigma}}^{\epsilon}\;.\label{cee}
\end{equation}
By \eqref{ce} we have
\begin{equation}
f_{\epsilon}^{\boldsymbol{\sigma}}(\boldsymbol{x})=\begin{cases}
0 & \mbox{if }\boldsymbol{x}\in\partial_{-}\mathcal{B}_{\boldsymbol{\sigma}}^{\epsilon}\\
1 & \mbox{if }\boldsymbol{x}\in\partial_{+}\mathcal{B}_{\boldsymbol{\sigma}}^{\epsilon}
\end{cases}\;.\label{ece}
\end{equation}
We next investigate two basic properties of $f_{\epsilon}^{\boldsymbol{\sigma}}$
in Lemmas 3.4 and 3.5 below. The statement and the proof of the first
lemma is similar to those of \cite[Lemma 8.7]{LMS} (in terms of the
notations of \cite{LMS}, our model corresponding to the special case
$\mathbb{M}=\mathbb{I}$, where $\mathbb{I}$ denotes the identity
matrix). Since the proof is much simpler for our specific case, and
some of the computations carried out below will be useful later, we
give the full proof of this lemma.
\begin{lem}
\label{lem35}For all $\boldsymbol{\sigma}\in\mathcal{S}$, we have
that
\begin{equation}
\theta_{\epsilon}\int_{\mathcal{C}_{\boldsymbol{\sigma}}^{\epsilon}}\left|(\mathscr{L}_{\epsilon}f_{\epsilon}^{\boldsymbol{\sigma}})(\boldsymbol{x})\right|\hat{\mu}_{\epsilon}(\boldsymbol{x})\,d\boldsymbol{x}=o_{\epsilon}(1)\;.\label{ee35}
\end{equation}
\end{lem}

\begin{proof}
To ease the notation, we may assume that $\boldsymbol{\sigma}=0$.
For $\boldsymbol{x}\in\mathcal{C}_{\boldsymbol{\sigma}}^{\epsilon}$,
write $\alpha_{i}:=\alpha_{i}(\boldsymbol{x})=\boldsymbol{x}\cdot\boldsymbol{v}_{i}^{\boldsymbol{\sigma}}$
so that $\boldsymbol{x}=\sum_{i=1}^{d}\alpha_{i}\boldsymbol{v}_{i}^{\boldsymbol{\sigma}}$.
By elementary computations, we can write
\begin{equation}
(\mathscr{L}_{\epsilon}f_{\epsilon}^{\boldsymbol{\sigma}})(\boldsymbol{x})=-\frac{1}{c_{\epsilon}^{\boldsymbol{\sigma}}}\sqrt{\frac{\lambda_{1}^{\boldsymbol{\sigma}}}{2\pi\epsilon}}e^{-\frac{\lambda_{1}^{\boldsymbol{\sigma}}}{2\epsilon}\alpha_{1}^{2}}\left[(\nabla U(\boldsymbol{x})+\lambda_{1}^{\boldsymbol{\sigma}}\boldsymbol{x})\cdot\boldsymbol{v}_{1}^{\boldsymbol{\sigma}}\right]\;.\label{c01}
\end{equation}
By the Taylor expansion of $\nabla U$ around $\boldsymbol{\sigma}$,
we have
\begin{align}
\nabla U(\boldsymbol{x})+\lambda_{1}^{\boldsymbol{\sigma}}\boldsymbol{x} & =(\nabla^{2}U)(\boldsymbol{\boldsymbol{\sigma}})\,\boldsymbol{x}+O(\delta^{2})+\lambda_{1}^{\boldsymbol{\sigma}}\boldsymbol{x}=\sum_{i=2}^{d}(\alpha_{i}\lambda_{i}+\alpha_{i}\lambda_{1})^{\boldsymbol{\sigma}}\boldsymbol{v}_{i}^{\boldsymbol{\sigma}}+O(\delta^{2})\;.\label{c02}
\end{align}
Since $\boldsymbol{v}_{1}^{\boldsymbol{\sigma}}\cdot\boldsymbol{v}_{i}^{\boldsymbol{\sigma}}=0$
for $2\le i\le d$, we conclude from \eqref{c01} and \eqref{c02}
that
\begin{equation}
(\mathscr{L}_{\epsilon}f_{\epsilon}^{\boldsymbol{\sigma}})(\boldsymbol{x})=O(\delta^{2})\,\epsilon^{-\frac{1}{2}}\,\exp\left\{ -\frac{\lambda_{1}^{\boldsymbol{\sigma}}}{2\epsilon}\alpha_{1}^{2}\right\} \;.\label{c03}
\end{equation}
Therefore, the left-hand side of \eqref{ee35} is bounded above by
\begin{align}
O(\delta^{2})\,\theta_{\epsilon}\, & \epsilon^{-\frac{1}{2}}Z_{\epsilon}^{-1}\int_{\mathcal{C}_{\boldsymbol{\sigma}}^{\epsilon}}\exp\left\{ -\frac{U(\boldsymbol{x})+(1/2)\lambda_{1}^{\boldsymbol{\sigma}}\alpha_{1}^{2}}{\epsilon}\right\} d\boldsymbol{x}\nonumber \\
 & =O(\delta^{2})\,\theta_{\epsilon}\,\epsilon^{-\frac{1}{2}}Z_{\epsilon}^{-1}e^{-\frac{H}{\epsilon}}\int_{\mathcal{C}_{\boldsymbol{\sigma}}^{\epsilon}}\exp\left\{ -\frac{1}{2\epsilon}\sum_{i=2}^{d}\lambda_{i}^{\boldsymbol{\sigma}}\alpha_{i}^{2}\right\} d\boldsymbol{x}\;,\label{c010}
\end{align}
where the identity follows from the second-order Taylor expansion
of $U$ around $\boldsymbol{\sigma}$ and the fact that $O(\delta^{3}/\epsilon)=o_{\epsilon}(1)$.
By the change of variables, the last integral can be bounded as
\begin{align*}
\frac{2J\delta}{\sqrt{\lambda_{1}^{\boldsymbol{\sigma}}}} & \int_{-2J\delta/\sqrt{\lambda_{2}^{\boldsymbol{\sigma}}}}^{2J\delta/\sqrt{\lambda_{2}^{\boldsymbol{\sigma}}}}\cdots\int_{-2J\delta/\sqrt{\lambda_{d}^{\boldsymbol{\sigma}}}}^{2J\delta/\sqrt{\lambda_{d}^{\boldsymbol{\sigma}}}}\exp\left\{ -\frac{1}{2\epsilon}\sum_{i=2}^{d}\lambda_{i}^{\boldsymbol{\sigma}}\alpha_{i}^{2}\right\} d\alpha_{2}\cdots d\alpha_{d}\\
 & \le\epsilon^{\frac{d-1}{2}}\frac{2J\delta}{\sqrt{\lambda_{1}^{\boldsymbol{\sigma}}}}\int_{-\infty}^{\infty}\cdots\int_{-\infty}^{\infty}\exp\left\{ -\frac{1}{2}\sum_{i=2}^{d}\lambda_{i}^{\boldsymbol{\sigma}}y_{i}^{2}\right\} dy_{2}\cdots dy_{d}=C\epsilon^{\frac{d-1}{2}}\delta\;.
\end{align*}
Inserting this into \eqref{c010} finishes the proof.
\end{proof}
\begin{lem}
\label{lem33}For all $\boldsymbol{\sigma}\in\mathcal{S}$, we have
\[
\theta_{\epsilon}\,\epsilon\int_{\mathcal{B}_{\boldsymbol{\sigma}}^{\epsilon}}|\nabla f_{\epsilon}^{\boldsymbol{\sigma}}(\boldsymbol{x})|^{2}\hat{\mu}_{\epsilon}(\boldsymbol{x})d\boldsymbol{x}=(1+o_{\epsilon}(1))\,\nu_{\star}^{-1}\,\omega_{\boldsymbol{\sigma}}\;.
\]
\end{lem}

\begin{proof}
See \cite[Lemma 8.4]{LMS}.
\end{proof}
For $\mathbf{q}=(\mathbf{q}(i):i\in S)\in\mathbb{R}^{S}$, we now
define a test function $F_{\epsilon}^{\mathbf{q}}:\mathbb{R}^{d}\rightarrow\mathbb{R}$.
This test function is used in Sections \ref{sec5} and \ref{sec4}.
In particular, in Section \ref{sec5}, the vector $\mathbf{q}$ may
depend on $\epsilon$. For this reason, we will keep track of the
dependence of the constants on $\mathbf{q}$ in the inequalities that
appear in this section.

We start by defining a real-valued function $\widehat{F}_{\epsilon}^{\mathbf{q}}$
on $\mathcal{H}^{\epsilon}$. This function is defined by
\begin{equation}
\widehat{F}_{\epsilon}^{\mathbf{q}}(\boldsymbol{x})=\begin{cases}
\mathbf{q}(i) & \text{if}\;\boldsymbol{x}\in\mathcal{W}_{i}^{\epsilon}\;,\;i\in S\;,\\
\mathbf{q}(i)+(\mathbf{q}(j)-\mathbf{q}(i))f_{\epsilon}^{\boldsymbol{\sigma}}(\boldsymbol{x}) & \text{if }\boldsymbol{x}\in\mathcal{B}_{\boldsymbol{\sigma}}^{\epsilon}\;,\;\boldsymbol{\sigma}\in\mathcal{S}_{i,\,j}\text{ with }i<j\;.
\end{cases}\label{fq}
\end{equation}
By \eqref{ece}, the function $\widehat{F}_{\epsilon}^{\mathbf{q}}$
is continuous on $\mathcal{H}^{\epsilon}$. Evidently,
\begin{equation}
\Vert\widehat{F}_{\epsilon}^{\mathbf{q}}\Vert_{L^{\infty}(\mathcal{H}^{\epsilon})}\le\Vert\mathbf{q}\Vert_{\infty}:=\max\{|\mathbf{q}(i)|:i\in S\}\;.\label{bfq1}
\end{equation}
Furthermore, since $\Vert\nabla f_{\epsilon}^{\boldsymbol{\sigma}}\Vert\le C\,\epsilon^{-1/2}$,
we deduce that the function $\widehat{F}_{\epsilon}^{\mathbf{q}}$
satisfies
\begin{equation}
\Vert\nabla\widehat{F}_{\epsilon}^{\mathbf{q}}\Vert_{L^{\infty}(\mathcal{H}^{\epsilon})}\le C\,\epsilon^{-1/2}\max\{|\mathbf{q}(i)-\mathbf{q}(j)|:i,\,j\in S\}\le C\,\epsilon^{-1/2}[D_{\mathbf{x}}(\mathbf{q},\,\mathbf{q})]^{1/2}\;.\label{bfq2}
\end{equation}
Here we stress that the constant $C$ is independent of $\mathbf{q}$.

Let $\mathcal{K}$ be a compact set containing $\mathcal{H}^{\epsilon}$
for all $\epsilon\in(0,\,2]$. For instance, one can select $\mathcal{K}=\mathcal{H}^{a}$
for any $a>2$. Then, for $\epsilon\in(0,\,1]$, by \eqref{bfq1}
and \eqref{bfq2}, there exists a continuous extension $F_{\epsilon}^{\mathbf{q}}:\mathbb{R}^{d}\rightarrow\mathbb{R}$
of $\widehat{F}_{\epsilon}^{\mathbf{q}}$ satisfying
\begin{equation}
\textup{supp }F_{\epsilon}^{\mathbf{q}}\subset\mathcal{K}\;,\;\;\left\Vert F_{\epsilon}^{\mathbf{q}}\right\Vert _{L^{\infty}(\mathbb{R}^{d})}\le\Vert\mathbf{q}\Vert_{\infty}\;,\text{and\;}\left\Vert \nabla F_{\epsilon}^{\mathbf{q}}\right\Vert _{L^{\infty}(\mathbb{R}^{d})}\le C\,\epsilon^{-1/2}[D_{\mathbf{x}}(\mathbf{q},\,\mathbf{q})]^{1/2}\;.\label{eka1}
\end{equation}
Suppose from now on that $\epsilon$ is not larger than $1$ so that
we can define $F_{\epsilon}^{\mathbf{q}}$ satisfying \eqref{eka1}.

Note that the Dirichlet form $\mathscr{D}_{\epsilon}(\cdot)$ corresponding
to the process $\boldsymbol{x}_{\epsilon}(t)$ is given by
\begin{equation}
\mathscr{D}_{\epsilon}(f)=\epsilon\,\int_{\mathbb{R}^{d}}|\nabla f(\boldsymbol{x})|^{2}\hat{\mu}_{\epsilon}(\boldsymbol{x})d\boldsymbol{x}\;\;;\;f\in H_{\textrm{loc}}^{1}(\mathbb{R}^{d})\;.\label{df}
\end{equation}

\begin{lem}
\label{lem34}For all $\mathbf{q}=(\mathbf{q}(i):i\in S)\in\mathbb{R}^{K}$,
we have that
\[
\theta_{\epsilon}\,\mathscr{D}_{\epsilon}(F_{\epsilon}^{\mathbf{q}})=(1+o_{\epsilon}(1))\,\nu_{\star}^{-1}D_{\mathbf{x}}(\mathbf{q},\,\mathbf{q})\;.
\]
\end{lem}

\begin{proof}
It is immediate from Lemma \ref{lem33} that
\[
\theta_{\epsilon}\,\epsilon\int_{\mathcal{H}^{\epsilon}}|\nabla F_{\epsilon}^{\mathbf{q}}(\boldsymbol{x})|^{2}\hat{\mu}_{\epsilon}(\boldsymbol{x})d\boldsymbol{x}=(1+o_{\epsilon}(1))\,\nu_{\star}^{-1}\,D_{\mathbf{x}}(\mathbf{q},\,\mathbf{q})\;.
\]
Thus, it suffices to show that
\begin{equation}
\theta_{\epsilon}\,\epsilon\int_{(\mathcal{H}^{\epsilon})^{c}}|\nabla F_{\epsilon}^{\mathbf{q}}(\boldsymbol{x})|^{2}\hat{\mu}_{\epsilon}(\boldsymbol{x})d\boldsymbol{x}=o_{\epsilon}(1)\,D_{\mathbf{x}}(\mathbf{q},\,\mathbf{q})\;.\label{eka3}
\end{equation}
Since $\nabla F_{\epsilon}^{\mathbf{q}}\equiv0$ on $\mathcal{K}^{c}$
we can replace the domain of integration in \eqref{eka3} with $\mathcal{K}\setminus\mathcal{H}_{\epsilon}$.
Then, by \eqref{eka1}, \eqref{ep3}, and by the fact that $U(\boldsymbol{x})\ge H+J^{2}\delta^{2}$
for $\boldsymbol{x}\notin\mathcal{H}_{\epsilon}$,
\[
\theta_{\epsilon}\,\epsilon\int_{\mathcal{K}\setminus\mathcal{H}^{\epsilon}}|\nabla F_{\epsilon}^{\mathbf{q}}(\boldsymbol{x})|^{2}\hat{\mu}_{\epsilon}(\boldsymbol{x})d\boldsymbol{x}\le\,m_{d}(\mathcal{K})\,D_{\mathbf{x}}(\mathbf{q,\,\mathbf{q}})\,\theta_{\epsilon}\,Z_{\epsilon}^{-1}\,e^{-H/\epsilon}\epsilon^{J^{2}}\;\le C\,D_{\mathbf{x}}(\mathbf{q,\,\mathbf{q}})\,\epsilon^{J^{2}-(d/2)}\;,
\]
where $m_{d}(\cdot)$ is the Lebesgue measure on $\mathbb{R}^{d}$.
This completes the proof since $J>\sqrt{12d}$.
\end{proof}

\section{\label{sec5}A Poisson equation }

For each $i\in S_{\star}$, we pick a smooth function $\zeta^{i}:\mathbb{R}^{d}\to\mathbb{R}$
such that $0\le\zeta^{i}\le1,$ $\zeta^{i}=1$ on $\mathcal{V}_{i}$,
and $\zeta^{i}=0$ on the complement of $\mathcal{V}'_{i}$. We also
set
\[
\bar{\zeta}^{i}=\int\zeta^{i}(\boldsymbol{x})\,\hat{\mu}_{\epsilon}(\boldsymbol{x})\,d\boldsymbol{x}\;.
\]
Define $\mathbf{a}_{\epsilon}=(\mathbf{a}_{\epsilon}(i):i\in S_{\star})\in\mathbb{R}^{S_{\star}}$
by
\[
\mathbf{a}_{\epsilon}(i)=\frac{Z_{\epsilon}^{-1}\,(2\pi\epsilon)^{d/2}\,e^{-h/\epsilon}\,\nu_{i}}{\bar{\zeta}^{i}}\;\;;\;i\in S_{\star}\;.
\]
From $\mu_{\epsilon}(\mathcal{V}_{i})\le\bar{\zeta}^{i}\le\mu_{\epsilon}(\mathcal{V}'_{i})$,
and \eqref{ep1}, we learn
\begin{equation}
\mathbf{a}_{\epsilon}(i)=1+o_{\epsilon}(1)\text{ for all }i\in S_{\star}\;.\label{aei}
\end{equation}
The main result of the current section is stated in the following
theorem.
\begin{thm}
\label{t51}For all $\mathbf{f}:S_{\star}\rightarrow\mathbb{R}$,
there exists a bounded function $\phi_{\epsilon}=\phi_{\epsilon}^{\mathbf{f}}:\mathbb{R}^{d}\rightarrow\mathbb{R}$
satisfying all the following properties:
\begin{enumerate}
\item $\phi_{\epsilon}\in C^{2}(\mathbb{R}^{d})$.
\item $\phi_{\epsilon}$ satisfies the equation
\begin{equation}
\theta_{\epsilon}\mathcal{\mathscr{L}_{\epsilon}}\phi_{\epsilon}=\sum_{i\in S_{\star}}\mathbf{a}_{\epsilon}(i)\,(L_{\mathbf{y}}\mathbf{f})(i)\,\zeta^{i}\;.\label{p321}
\end{equation}
\item For all $i\in S_{\star}$, it holds that
\begin{equation}
\lim_{\epsilon\rightarrow\infty}\sup_{\boldsymbol{x}\in\mathcal{V}_{i}}\left|\phi_{\epsilon}(\boldsymbol{x})-\mathbf{f}(i)\right|=0\;.\label{p322}
\end{equation}
\end{enumerate}
\end{thm}

\begin{rem}
In \cite[Theorem 5.3]{ST}, a similar analysis has been carried out
for a slightly different situation. In \cite{ST}, we treat a Poisson
equation of the form \eqref{p321} for a different right-hand side.
The form of the right-hand we have chosen in \eqref{p321} enables
us to use the Poincar\'{e}'s inequality (see subsection 4.3 below).
Furthermore, the proof therein relies on the capacity estimates between
metastable valleys. Our proof though does not use any capacity estimates
and has a chance to be applicable to the non-reversible variant of
our model. This fact deserves to be highlighted here once more.
\end{rem}

Note that the function $L_{\mathbf{y}}\mathbf{f}:S_{\star}\mathbb{\rightarrow R}$
satisfies
\begin{equation}
\sum_{i\in S_{\star}}(L_{\mathbf{y}}\mathbf{f})(i)\,\mu_{\star}(i)=0\label{eg2}
\end{equation}
since $\mu_{\star}(\cdot)$ defined in \eqref{inv} is the invariant
measure for the Markov chain $\mathbf{y}(t)$. Let $\mathbf{e}_{i}\in\mathbb{R}^{S_{\star}}$,
$i\in S_{\star}$, be the $i$th unit vector defined by
\begin{equation}
\mathbf{e}_{i}(j)=\mathbf{1}\{i=j\}\;\;;\;j\in S_{\star}\;.\label{estar}
\end{equation}
For $i,\,j\in S_{\star}$, let $\mathbf{S}_{i,j}$ be the collection
of $\mathbf{f}\in\mathbb{R}^{S_{\star}}$ satisfying
\[
L_{\mathbf{y}}\mathbf{f}=\frac{1}{\mu_{\star}(i)}\mathbf{e}_{i}-\frac{1}{\mu_{\star}(j)}\mathbf{e}_{j}=\frac{\nu_{\star}}{\nu_{i}}\mathbf{e}_{i}-\frac{\nu_{\star}}{\nu_{j}}\mathbf{e}_{j}\;.
\]
Remark that the selection $\mathbf{S}_{i,j}$ is consistent with the
condition $\eqref{eg2}$ for $L_{\mathbf{y}}\mathbf{f}$. It is immediate
from the irreducibility of the Markov chain $\mathbf{y}(t)$ that
$\bigcup_{i,\,j\in S_{\star}}\mathbf{S}_{i,j}$ spans whole space
$\mathbb{R}^{S_{\star}}.$ Note that for $\mathbf{f}\equiv0$, it
suffices to select $\phi_{\epsilon}\equiv0$ and thus it suffices
to consider non-zero $\mathbf{f}$. Therefore, by the linearity of
the statement of Theorem \ref{t51} with respect to $\mathbf{f}$,
it suffices to prove the theorem for $\mathbf{f}\in\mathbf{S}_{i,j}$
only. \textit{To simplify notations, let us assume that $1,\,2\in S_{\star}$,
and assume that $\mathbf{f}\in\mathbf{S}_{1,2}$, i.e,
\begin{equation}
L_{\mathbf{y}}\mathbf{f}=\frac{\nu_{\star}}{\nu_{1}}\mathbf{e}_{1}-\frac{\nu_{\star}}{\nu_{2}}\mathbf{e}_{2}\;.\label{eg3}
\end{equation}
Now we fix such $\mathbf{f}$ throughout the remaining part of the
current section. }We note that $(L_{\mathbf{y}}\mathbf{f})(i)=0$
for all $i\neq1,\,2$.

Our plan is to select the test function $\phi_{\epsilon}$ that appeared
in Theorem \ref{t51} as a minimizer of a functional $\mathscr{I}_{\epsilon}(\cdot)$
that will be defined in Section \ref{s51}. More precisely, we first
take a minimizer $\psi_{\epsilon}$ of that functional satisfies a
certain symmetry condition (see \eqref{hg1} below) and analyze its
property thoroughly in Sections \ref{sec52}-\ref{s55}. Then, we
shall prove that a translation of $\psi_{\epsilon}$, which is also
a minimizer of $\mathscr{I}_{\epsilon}(\cdot)$, satisfies all the
requirements of Theorem \ref{t51} in Section \ref{s56}.

\subsection{\label{s51}A Variational principle}

Recall from \eqref{df} the functional $\mathscr{D}_{\epsilon}(\cdot)$
and define a functional $\mathscr{I}_{\epsilon}(\cdot)$ on $H^{1}(\mathbb{R}^{d})$
as
\begin{equation}
\mathscr{I}_{\epsilon}(\phi)=\frac{1}{2}\theta_{\epsilon}\,\mathscr{D}_{\epsilon}(\phi)+\sum_{i=1,\,2}\mathbf{a}_{\epsilon}(i)\,(L_{\mathbf{y}}\mathbf{f})(i)\,\int\zeta^{i}(\boldsymbol{x})\phi(\boldsymbol{x})\,\hat{\mu}_{\epsilon}(\boldsymbol{x})d\boldsymbol{x}\;.\label{hg1}
\end{equation}
Denote by $\psi_{\epsilon}$ a minimizer of $\mathscr{I}_{\epsilon}(\cdot)$.
Then, it is well-known that $\psi_{\epsilon}$ classically solves
\eqref{p321}, i.e.,
\begin{equation}
\theta_{\epsilon}\mathcal{\mathscr{L}_{\epsilon}}\psi_{\epsilon}=\sum_{i\in S_{\star}}\mathbf{a}_{\epsilon}(i)\,(L_{\mathbf{y}}\mathbf{f})(i)\,\chi_{\mathcal{V}_{i}}\;.\label{psie}
\end{equation}
Our purpose in the remaining part is to find a constant $c_{\epsilon}$
such that $\phi_{\epsilon}=\psi_{\epsilon}+c_{\epsilon}$ satisfies
\eqref{p322}. Note that this $\phi_{\epsilon}$ also satisfies \eqref{p321}
and hence, this finishes the proof.

Write
\begin{equation}
\mathbf{p}_{\epsilon}(i)=\mathbf{a}_{\epsilon}(i)\,(L_{\mathbf{y}}\mathbf{f})(i)\,\int\zeta^{i}(\boldsymbol{x})\,\psi_{\epsilon}(\boldsymbol{x})\,\hat{\mu}_{\epsilon}(\boldsymbol{x})d\boldsymbol{x}\;\;;\;i\in S_{\star}\;,\label{eg4}
\end{equation}
so that $\mathbf{p}_{\epsilon}(i)=0$ for all $i\neq1,\,2$ because
of \eqref{eg3}. Note that if we add a constant $a$ to $\psi_{\epsilon}$,
then the value of $\mathbf{p}_{\epsilon}(i)$ for $i=1,2$ changes
to $\mathbf{p}'_{\epsilon}(i)$, with
\[
\mathbf{p}'_{\epsilon}(1)=\mathbf{p}_{\epsilon}(1)+ab\;,\;\;\;\mathbf{p}'_{\epsilon}(2)=\mathbf{p}_{\epsilon}(2)-ab\;,\ \ \ \mathbf{p}_{\epsilon}(i)=0\;,
\]
for $i\neq1,2$, where
\[
b=Z_{\epsilon}^{-1}\,(2\pi\epsilon)^{d/2}\,e^{-h/\epsilon}\,\nu_{\star}\;.
\]
Hence, by adding a constant to $\psi_{\epsilon}$ if necessary, we
can assume without loss of generality that $\mathbf{p}_{\epsilon}(1)=\mathbf{p}_{\epsilon}(2)$.
Set
\begin{equation}
\lambda_{\epsilon}:=-\mathbf{p}_{\epsilon}(1)=-\mathbf{p}_{\epsilon}(2)\;.\label{lame}
\end{equation}
We now multiply both sides of the equation \eqref{psie} by $-\psi_{\epsilon}$
and integrate with respect to the invariant measure $\mu_{\epsilon}$
to deduce
\begin{equation}
\theta_{\epsilon}\,\mathscr{D}_{\epsilon}(\psi_{\epsilon})=2\lambda_{\epsilon}\;.\label{hg2}
\end{equation}
Consequently, $\lambda_{\epsilon}>0$ and furthermore, by \eqref{hg1},
\eqref{lame}, and \eqref{hg2} we obtain
\begin{equation}
\mathscr{I}_{\epsilon}(\psi_{\epsilon})=-\lambda_{\epsilon}\;.\label{hg3}
\end{equation}

\subsection{\label{sec52}Lower bound on $\lambda_{\epsilon}$}

In this subsection, we prove a rough lower bound for $\lambda_{\epsilon}$
in Proposition \ref{p53}.

We start by providing some relations between Dirichlet forms $D_{\mathbf{x}}(\cdot,\,\cdot)$
and $D_{\mathbf{y}}(\cdot,\,\cdot)$. For $\mathbf{u}:S_{\star}\rightarrow\mathbb{R}$
and $\mathbf{u}':S\rightarrow\mathbb{R}$, we say that $\mathbf{u}'$
is an extension of $\mathbf{u}$ if $\mathbf{u}'(i)=\mathbf{u}(i)$
for all $i\in S_{\star}$. For $\mathbf{u}:S_{\star}\rightarrow\mathbb{R}$,
we define the harmonic extension $\widetilde{\mathbf{u}}:S\rightarrow\mathbb{R}$
of $\mathbf{u}$ as the extension of $\mathbf{u}$ satisfying
\begin{equation}
(L_{\mathbf{x}}\widetilde{\mathbf{u}})(i)=0\;\;\text{for all }i\in S\setminus S_{\star}\;.\label{ecc2}
\end{equation}
The following lemma will be used in several instances in the remaining
part of the article.
\begin{lem}
\label{lem54}For all $\mathbf{u},\,\mathbf{v}:S_{\star}\rightarrow\mathbb{R}$,
the following properties hold.
\begin{enumerate}
\item For harmonic extension $\widetilde{\mathbf{u}}$ and $\widetilde{\mathbf{v}}$
of $\mathbf{u}$ and $\mathbf{v}$, respectively, we have
\begin{equation}
D_{\mathbf{x}}(\widetilde{\mathbf{u}},\,\widetilde{\mathbf{v}})=\nu_{\star}D_{\mathbf{y}}(\mathbf{u},\,\mathbf{v})\;.\label{euv}
\end{equation}
\item For any extensions $\mathbf{v}_{1},\,\mathbf{v}_{2}:S\rightarrow\mathbb{R}$
of $\mathbf{v}$, we have
\begin{equation}
D_{\mathbf{x}}(\widetilde{\mathbf{u}},\,\mathbf{v}_{1})=D_{\mathbf{x}}(\widetilde{\mathbf{u}},\,\mathbf{v}_{2})\;.\label{euv2}
\end{equation}
\end{enumerate}
\end{lem}

\begin{proof}
For part (1), recall the function $\mathbf{e}_{i}$, $i\in S_{\star}$,
that was defined in \eqref{estar}. Since both $D_{\mathbf{x}}(\cdot,\,\cdot)$
and $D_{\mathbf{y}}(\cdot,\,\cdot)$ are bi-linear forms, it suffices
to check \eqref{euv} for $(\mathbf{u},\,\mathbf{v})=(\mathbf{e}_{i},\,\mathbf{e}_{j})$
for $i\in S_{\star}$ and $j\in S_{\star}$. By \eqref{epx1}, the
harmonic extension of $\mathbf{e}_{i}$, namely $\mathbf{\widetilde{e}}_{i}:S\rightarrow\mathbb{R}$,
is the equilibrium potential between $\{i\}$ and $S_{\star}\setminus\{i\}$,
with respect to the process $\mathbf{x}(\cdot)$, and hence we have
\begin{equation}
D_{\mathbf{x}}(\mathbf{\widetilde{e}}_{i},\,\mathbf{\widetilde{e}}_{i})=\textup{cap}_{\mathbf{x}}(\{i\},\,S_{\star}\setminus\{i\})\;.\label{euv3}
\end{equation}
Similarly, for $i\neq j\in S_{\star},$ the function $\mathbf{\widetilde{e}}_{i}+\widetilde{\mathbf{e}}_{j}:S\rightarrow\mathbb{R}$
is the equilibrium potential between $\{i,\,j\}$ and $S_{\star}\setminus\{i,\,j\}$,
with respect to the process $\mathbf{x}(\cdot)$, and therefore it
holds
\begin{equation}
D_{\mathbf{x}}(\mathbf{\widetilde{e}}_{i}+\widetilde{\mathbf{e}}_{j},\,\mathbf{\widetilde{e}}_{i}+\widetilde{\mathbf{e}}_{j})=\textup{cap}_{\mathbf{x}}(\{i,j\},\,S_{\star}\setminus(\{i\}\cap\{j\}))\;.\label{euv4}
\end{equation}
By \eqref{euv3}, \eqref{euv4} and the bi-linearity of $D_{\mathbf{x}}$,
we have
\begin{equation}
D_{\mathbf{x}}(\mathbf{\widetilde{e}}_{i},\,\widetilde{\mathbf{e}}_{j})=-\beta_{i,\,j}\;\;;\;i\neq j\in S_{\star}\;.\label{euv5}
\end{equation}
It also follows from the definition
\begin{equation}
D_{\mathbf{y}}(\mathbf{{e}}_{i},\,{\mathbf{e}}_{j})=\frac{1}{2\nu_{\star}}\beta_{i,\,j}(0-1)(1-0)+\frac{1}{2\nu_{\star}}\beta_{j,\,i}(1-0)(0-1)=-\frac{\beta_{i,\,j}}{\nu_{\star}}\;.\label{euv6}
\end{equation}
From \eqref{euv5} and \eqref{euv6}, we deduce \eqref{euv} for $(\mathbf{u},\,\mathbf{v})=(\mathbf{e}_{i},\,\mathbf{e}_{j})$
with $i\neq j$.

Now, we turn to the case $(\mathbf{u},\,\mathbf{v})=(\mathbf{e}_{i},\,\mathbf{e}_{i})$
for some $i\in S_{\star}$. For this case, since $\sum_{j\in S_{\star}}\mathbf{e}_{j}=1$
on $S_{\star}$, it is immediate that $\sum_{j\in S_{\star}}\widetilde{\mathbf{e}}_{j}=1$
on $S$. Therefore,
\[
D_{\mathbf{x}}\bigg(\mathbf{\widetilde{e}}_{i},\,\sum_{j\in S_{\star}}\widetilde{\mathbf{e}}_{j}\bigg)=0\;.
\]
By this equation, \eqref{euv5}, and the bi-linearity of $D_{\mathbf{x}}$,
we obtain
\[
D_{\mathbf{x}}(\mathbf{\widetilde{e}}_{i},\,\mathbf{\widetilde{e}}_{i})=-\sum_{j\in S_{\star}:j\neq i}D_{\mathbf{x}}(\mathbf{\widetilde{e}}_{i},\,\mathbf{\widetilde{e}}_{j})=\sum_{j\in S_{\star}:j\neq i}\beta_{i,\,j}\;.
\]
This finishes the proof for part (1) since by the direct computation
we can verify that $D_{\mathbf{y}}(\mathbf{{e}}_{i},\,{\mathbf{e}}_{i})=\nu_{\star}^{-1}\sum_{j\in S_{\star}:j\neq i}\beta_{i,\,j}$.

For part (2), by the definition \eqref{dxfg} of $D_{\mathbf{x}}$,
we can write
\[
D_{\mathbf{x}}(\widetilde{\mathbf{u}},\,\mathbf{v}_{1}-\mathbf{v}_{2})=\sum_{i\in S}\mu(i)\,(-L_{\mathbf{x}}\mathbf{\widetilde{u}})(i)\,(\mathbf{v}_{1}(i)-\mathbf{v}_{2}(i))\;.
\]
The last summation is $0$ since $(L_{\mathbf{x}}\mathbf{\widetilde{u}})(i)=0$
for $i\in S\setminus S_{\star}$ and $\mathbf{v}_{1}(i)-\mathbf{v}_{2}(i)=0$
for $i\in S_{\star}$. This completes the proof.
\end{proof}
Now we are ready to establish an a priori lower bound on $\lambda_{\epsilon}$.
We remark that a sharp asymptotic of $\lambda_{\epsilon}$ will be
given in Section \ref{s56}. Recall that we have fixed $\text{\ensuremath{\mathbf{f}}}$
as in \eqref{eg3}.
\begin{prop}
\label{p53}We have
\[
\lambda_{\epsilon}\ge(1/2)D_{\mathbf{y}}(\mathbf{f},\mathbf{f})+o_{\epsilon}(1)\;.
\]
\end{prop}

\begin{proof}
Recall from Section \ref{s32} the test function $F_{\epsilon}^{\mathbf{\widetilde{\mathbf{f}}}}:\mathbb{R}^{d}\rightarrow\mathbb{R}$
where $\widetilde{\mathbf{f}}$ is the harmonic extension of $\mathbf{f}$
defined above. By \eqref{ep4}, \eqref{ep4'}, \eqref{inv} and Lemma
\ref{lem34},
\begin{align*}
\mathscr{I}_{\epsilon}(F_{\epsilon}^{\mathbf{\widetilde{\mathbf{f}}}}) & =\frac{1}{2}\nu_{\star}^{-1}D_{\mathbf{x}}(\widetilde{\mathbf{f}},\,\widetilde{\mathbf{f}})+\sum_{i\in S_{\star}}\text{\ensuremath{\mu_{\star}(i)\,}\ensuremath{(L_{\mathbf{y}}\mathbf{f})(i)}\,\ensuremath{\mathbf{f}(i)}}+o_{\epsilon}(1)\;.
\end{align*}
By Lemma \ref{lem54}, we can conclude that the right-hand side of
the previous display is equal to
\[
\frac{1}{2}D_{\mathbf{y}}(\mathbf{f},\,\mathbf{f})-\sum_{i\in S_{\star}}\mu_{\star}(i)\,(-\text{\ensuremath{L_{\mathbf{y}}\mathbf{f})(i)}\,\ensuremath{\mathbf{f}(i)}}+o_{\epsilon}(1)=-\frac{1}{2}D_{\mathbf{y}}(\mathbf{f},\,\mathbf{f})+o_{\epsilon}(1)\;.
\]
The proof is completed by recalling that $\mathscr{I}_{\epsilon}(F_{\epsilon}^{\mathbf{\widetilde{\mathbf{f}}}})\ge\mathscr{I}_{\epsilon}(\psi_{\epsilon})=-\lambda_{\epsilon}$.
\end{proof}

\subsection{\label{s53}$L^{2}$-estimates based on Poincar\'{e}'s inequality }

Recall from Section \ref{s213} the small constant $a>0$ such that
there is no critical point $\boldsymbol{c}$ of $U$ satisfying $U(\boldsymbol{c})\in[H-a,\,H)$.
For $i\in S$, denote by\footnote{In fact, the set $\mathcal{V}_{i}^{(1)}$ is the same set with $\mathcal{W}_{i}^{o}$
defined in Section \ref{s213}; we use alternative notation here for
the notational convenience. } $\mathcal{V}_{i}^{(1)}$ and $\mathcal{V}_{i}^{(2)}$ the unique
connected component of $\{\boldsymbol{x}:U(\boldsymbol{x})\le H-a\}$
and of $\{\boldsymbol{x}:U(\boldsymbol{x})\le H-a/2\}$ contained
in $\mathcal{W}_{i}$, respectively. Thus, we have
\[
\mathcal{V}_{i}\subset\mathcal{V}_{i}^{(1)}\subset\mathcal{V}_{i}^{(2)}\subset\mathcal{W}_{i}\;.
\]

For $i\in S$, define
\[
\mathbf{q}_{\epsilon}(i)=\frac{1}{m_{d}(\mathcal{V}_{i}^{(1)})}\int_{\mathcal{V}_{i}^{(1)}}\psi_{\epsilon}(\boldsymbol{x})d\boldsymbol{x}\;\;\;\text{and\;}\;\;\widehat{\mathbf{q}}_{\epsilon}(i)=\frac{1}{m_{d}(\mathcal{V}_{i}^{(2)})}\int_{\mathcal{V}_{i}^{(2)}}\psi_{\epsilon}(\boldsymbol{x})d\boldsymbol{x}\;,
\]
where $m_{d}$ denotes the Lebesgue measure of $\mathbb{R}^{d}$.
\begin{prop}
\label{p55}There exists a constant $C>0$ such that the following
estimate holds for all $i\in S$:
\[
\left\Vert \psi_{\epsilon}-\mathbf{q}_{\epsilon}(i)\right\Vert _{L^{2}(\mathcal{V}_{i}^{(2)})}\le C\,e^{-a/(3\epsilon)}\,\lambda_{\epsilon}\;.
\]
\end{prop}

\begin{rem}
Here and elsewhere in this paper, $L^{p}$ norms are computed with
respect to the Lebesgue measure of $\mathbb{R}^{d}$.
\end{rem}

\begin{proof}
By Poincaré's inequality, the definition of $\mathcal{V}_{i}^{(2)}$,
\eqref{ep3}, and \eqref{hg2},
\begin{align*}
\int_{\mathcal{V}_{i}^{(2)}}\left|\psi_{\epsilon}(\boldsymbol{x})-\widehat{\mathbf{q}}_{\epsilon}(i)\right|^{2}d\boldsymbol{x} & \le C\int_{\mathcal{V}_{i}^{(2)}}|\nabla\psi_{\epsilon}(\boldsymbol{x})|^{2}d\boldsymbol{x}\le Ce^{(H-a/2)/\epsilon}\int_{\mathcal{V}_{i}^{(2)}}|\nabla\psi_{\epsilon}(\boldsymbol{x})|^{2}e^{-U(\boldsymbol{x})/\epsilon}d\boldsymbol{x}\\
 & \le Ce^{(H-a/2)/\epsilon}Z_{\epsilon}\epsilon^{-1}\mathcal{\mathscr{D}}_{\epsilon}(\psi_{\epsilon})\le Ce^{-a/(2\epsilon)}\epsilon^{d/2-1}\lambda_{\epsilon}\le Ce^{-a/(3\epsilon)}\lambda_{\epsilon}\;,
\end{align*}
From this and Cauchy-Schwarz's inequality we deduce,
\[
\left|\mathbf{q}_{\epsilon}(i)-\mathbf{\widehat{\mathbf{q}}}_{\epsilon}(i)\right|\le\frac{1}{m_{d}(\mathcal{V}_{i}^{(1)})}\int_{\mathcal{V}_{i}^{(1)}}\left|\psi_{\epsilon}(\boldsymbol{x})-\widehat{\mathbf{q}}_{\epsilon}(i)\right|d\boldsymbol{x}\le C\int_{\mathcal{V}_{i}^{(2)}}\left|\psi_{\epsilon}(\boldsymbol{x})-\widehat{\mathbf{q}}_{\epsilon}(i)\right|d\boldsymbol{x}\le Ce^{-a/(6\epsilon)}\lambda_{\epsilon}^{1/2}\;.
\]
Combining the above two bounds complete the proof.
\end{proof}

\subsection{\label{s54}$L^{\infty}$-estimates on valleys}

In this subsection, we use the interior elliptic regularity techniques
and a suitable bootstrapping argument to reinforce the $L^{2}$-estimate
in $\mathcal{V}_{i}^{(2)}$ that was obtained in Proposition \ref{p55}
to $L^{\infty}$-estimate in the smaller set $\mathcal{V}_{i}^{(1)}$.
This type of argument has been introduced originally in \cite{ET},
and is suitably modified to yield a desired $L^{\infty}$-estimate.

We henceforth write
\[
\frac{Z_{\epsilon}}{(2\pi\epsilon)^{d/2}\,e^{-h/\epsilon}\,\nu_{\star}}=1+\eta_{\epsilon}
\]
where $\eta_{\epsilon}=o_{\epsilon}(1)$ by \eqref{ep4}.
\begin{lem}
\label{lem57}We have
\begin{align*}
 & \left|(1+\eta_{\epsilon})\mathbf{p}_{\epsilon}(1)-\mathbf{q}_{\epsilon}(1)\right|\le\left\Vert \psi_{\epsilon}-\mathbf{q}_{\epsilon}(1)\right\Vert _{L^{\infty}(\mathcal{V}_{1}^{(1)})}\;\text{and}\\
 & \left|(1+\eta_{\epsilon})\mathbf{p}_{\epsilon}(2)+\mathbf{q}_{\epsilon}(2)\right|\le\left\Vert \psi_{\epsilon}-\mathbf{q}_{\epsilon}(2)\right\Vert _{L^{\infty}(\mathcal{V}_{2}^{(1)})}\;.
\end{align*}
\end{lem}

\begin{proof}
By \eqref{aei}, \eqref{eg3}, and \eqref{eg4}, we can write
\begin{align*}
\mathbf{p}_{\epsilon}(1) & =\frac{Z_{\epsilon}^{-1}\,(2\pi\epsilon)^{d/2}\,e^{-h/\epsilon}\,\nu_{1}}{\bar{\zeta}^{1}}\,\frac{\nu_{\star}}{\nu_{1}}\,\int\psi_{\epsilon}(\boldsymbol{x})\,\zeta^{1}(\boldsymbol{x})\,\hat{\mu}_{\epsilon}(\boldsymbol{x})d\boldsymbol{x}\\
 & =\frac{1}{1+\eta_{\epsilon}}\frac{1}{\bar{\zeta}^{1}}\int\psi_{\epsilon}(\boldsymbol{x})\,\zeta^{1}(\boldsymbol{x})\,\hat{\mu}_{\epsilon}(\boldsymbol{x})d\boldsymbol{x}
\end{align*}
Therefore, by the definition of $\bar{\zeta}^{1}$, we can write
\begin{align*}
\left|(1+\eta_{\epsilon})\mathbf{p}_{\epsilon}(1)-\mathbf{q}_{\epsilon}(1)\right| & =\left|\frac{1}{\bar{\zeta}^{1}}\int\,\left(\psi_{\epsilon}(\boldsymbol{x})-\mathbf{q}_{\epsilon}(1)\right)\,\zeta^{1}(\boldsymbol{x})\,\hat{\mu}_{\epsilon}(\boldsymbol{x})d\boldsymbol{x}\right|\\
 & \le\left\Vert \psi_{\epsilon}-\mathbf{q}_{\epsilon}(1)\right\Vert _{L^{\infty}(\mathcal{V}_{1}^{(1)})}
\end{align*}
where the last equality holds since the support of $\zeta^{1}$ is
a subset of $\mathcal{V}_{1}^{(1)}$. Thus, the estimate for $\mathbf{p}_{\epsilon}(1)$
follows. The proof for $\mathbf{p}_{\epsilon}(2)$ is identical.
\end{proof}
\begin{prop}
\label{p58}For all $i\in S$, we have
\[
\left\Vert \psi_{\epsilon}-\mathbf{q}_{\epsilon}(i)\right\Vert _{L^{\infty}(\mathcal{V}_{i}^{(1)})}=o_{\epsilon}(1)\lambda_{\epsilon}\;.
\]
\end{prop}

\begin{proof}
Fix $i\in S$. On $\mathcal{V}_{i}^{(2)}$, the function $\psi_{\epsilon}$
satisfies the equation
\[
\mathcal{\mathscr{L}_{\epsilon}}\psi_{\epsilon}=\theta_{\epsilon}^{-1}\mathbf{a}_{\epsilon}(i)\,\mathbf{g}(i)\,\zeta^{i}\;.
\]
where $\mathbf{g}=L_{\mathbf{y}}\mathbf{f}$. We can rewrite the equation
as
\[
\epsilon\Delta(\psi_{\epsilon}-\mathbf{q}_{\epsilon}(i))=\nabla\cdot\left[(\psi_{\epsilon}-\mathbf{q}_{\epsilon}(i))\nabla U\right]-(\psi_{\epsilon}-\mathbf{q}_{\epsilon}(i))\Delta U+\frac{C}{\theta_{\epsilon}}\zeta^{i}\ ,
\]
for some constant $C\ge0$. Then, by the local interior elliptic estimate
\cite[Theorem 8.17]{GT} with
\[
R:=\frac{1}{3}\min_{j\in S}\textup{dist}(\partial\mathcal{V}_{j}^{(1)},\,\partial\mathcal{V}_{j}^{(2)})\;,
\]
we obtain that, for any $p>d$ and for some constant $C_{p}>0$,
\[
\left\Vert \psi_{\epsilon}-\mathbf{q}_{\epsilon}(i)\right\Vert _{L^{\infty}(\mathcal{V}_{i}^{(1)})}\le C_{p}\left\Vert \psi_{\epsilon}-\mathbf{q}_{\epsilon}(i)\right\Vert _{L^{2}(\mathcal{V}_{i}^{(2)})}+\frac{C_{p}}{\epsilon}\left\Vert \psi_{\epsilon}-\mathbf{q}_{\epsilon}(i)\right\Vert _{L^{p}(\mathcal{V}_{i}^{(2)})}+o_{\epsilon}(1)\;.
\]
Let us select $p=2d$ for the sake of definiteness and let us write
$\left\Vert \psi_{\epsilon}\right\Vert _{\infty}:=\left\Vert \psi_{\epsilon}\right\Vert _{L^{\infty}(\mathbb{R}^{d})}$
for the simplicity of notation. Then, by Propositions \ref{p53},
\ref{p55}, H\"older's inequality, and the trivial fact that $|\mathbf{q}_{\epsilon}(i)|\le\left\Vert \psi_{\epsilon}\right\Vert _{\infty}$,
we obtain
\begin{equation}
\begin{aligned}\left\Vert \psi_{\epsilon}-\mathbf{q}_{\epsilon}(i)\right\Vert _{L^{\infty}(\mathcal{V}_{i}^{(1)})} & \le o_{\epsilon}(1)\lambda_{\epsilon}+\frac{C}{\epsilon}\left\Vert \psi_{\epsilon}-\mathbf{q}_{\epsilon}(i)\right\Vert _{L^{2}(\mathcal{V}_{i}^{(2)})}^{1/d}\left\Vert \psi_{\epsilon}-\mathbf{q}_{\epsilon}(i)\right\Vert _{L^{\infty}(\mathcal{V}_{i}^{(2)})}^{1-(1/d)}\\
 & =o_{\epsilon}(1)\left[\lambda_{\epsilon}+\lambda_{\epsilon}^{1/d}\left\Vert \psi_{\epsilon}-\mathbf{q}_{\epsilon}(i)\right\Vert _{L^{\infty}(\mathcal{V}_{i}^{(2)})}^{1-(1/d)}\right]\;,\\
 & \le o_{\epsilon}(1)\left[\lambda_{\epsilon}+\lambda_{\epsilon}^{1/d}\left\Vert \psi_{\epsilon}\right\Vert _{\infty}^{1-(1/d)}\right]\;.\\
 & \le o_{\epsilon}(1)\left[\lambda_{\epsilon}+\left\Vert \psi_{\epsilon}\right\Vert _{\infty}\right]\;.
\end{aligned}
\label{se1}
\end{equation}
Now we present a bootstrapping argument. Write
\[
\mathbf{m}_{\epsilon}(i)=\left\Vert \psi_{\epsilon}\right\Vert _{L^{\infty}(\mathcal{V}_{i}^{(1)})}\mbox{ for }i\in S\;\;\;\mbox{and}\;\;\;\xi=\xi_{\epsilon}=\max\{\mathbf{m}_{\epsilon}(1),\,\mathbf{m}_{\epsilon}(2)\}
\]
Then, it holds that $\Vert\psi_{\epsilon}||_{\infty}=\xi$, since
otherwise $\mathscr{I}(u_{\epsilon}\circ\psi_{\epsilon})<\mathscr{I}(\psi_{\epsilon})$
where
\[
u_{\epsilon}(t)=\begin{cases}
\Vert\psi_{\epsilon}||_{\infty} & \text{if }t\ge\xi\;,\\
t & \text{if }|t|<\Vert\psi_{\epsilon}||_{\infty\;,}\\
-\Vert\psi_{\epsilon}||_{\infty} & \text{if }t\le-\xi\;.
\end{cases}
\]
Thus we can write $\Vert\psi_{\epsilon}||_{\infty}=\mathbf{m}_{\epsilon}(k)$
where $k$ is either $1$ or $2$. Then,
\begin{equation}
\Vert\psi_{\epsilon}||_{\infty}=\mathbf{m}_{\epsilon}(k)=\left\Vert \psi_{\epsilon}\right\Vert _{L^{\infty}(\mathcal{V}_{k}^{(1)})}\le\left\Vert \psi_{\epsilon}-\mathbf{q}_{\epsilon}(k)\right\Vert _{L^{\infty}(\mathcal{V}_{k}^{(1)})}+|\mathbf{q}_{\epsilon}(k)|\;.\label{se2}
\end{equation}
By Lemma \ref{lem57} and \eqref{lame}, we have that
\begin{equation}
|\mathbf{q}_{\epsilon}(k)|\le(1+o_{\epsilon}(1))\lambda_{\epsilon}+\left\Vert \psi_{\epsilon}-\mathbf{q}_{\epsilon}(k)\right\Vert _{L^{\infty}(\mathcal{V}_{k}^{(1)})}\;.\label{se3}
\end{equation}
By combining \eqref{se2} and \eqref{se3}, we obtain
\begin{equation}
\Vert\psi_{\epsilon}||_{\infty}\le(1+o_{\epsilon}(1))\lambda_{\epsilon}+2\left\Vert \psi_{\epsilon}-\mathbf{q}_{\epsilon}(k)\right\Vert _{L^{\infty}(\mathcal{V}_{k}^{(1)})}\;.\label{se4}
\end{equation}
Inserting \eqref{se4} into \eqref{se1} with $i=k$ yields
\begin{equation}
\left\Vert \psi_{\epsilon}-\mathbf{q}_{\epsilon}(k)\right\Vert _{L^{\infty}(\mathcal{V}_{k}^{(1)})}\le o_{\epsilon}(1)\lambda_{\epsilon}\;.\label{se5}
\end{equation}
By \eqref{se4} and \eqref{se5}, we have
\begin{equation}
\Vert\psi_{\epsilon}||_{\infty}\le(1+o_{\epsilon}(1))\lambda_{\epsilon}\;.\label{se6}
\end{equation}
Finally, inserting this into \eqref{se1} finishes the proof.
\end{proof}

\subsection{\label{subsec:s45}Extension of flatness of $\psi_{\epsilon}$}

For $i\in S$, we denote by $\mathcal{V}_{i}^{(3)}$ the unique connected
component of the set
\[
\Omega_{\epsilon}=\,\Big\{\,\boldsymbol{x}:U(\boldsymbol{x})\le H-\frac{1}{4}J^{2}\delta^{2}\,\Big\}\;.
\]
contained in the set $\mathcal{W}_{i}$. Note that we now have
\[
\mathcal{V}_{i}\subset\mathcal{V}_{i}^{(1)}\subset\mathcal{V}_{i}^{(2)}\subset\mathcal{V}_{i}^{(3)}\subset\mathcal{W}_{i}\;.
\]
Note that the sets $\mathcal{V}_{i}^{(1)}$ and $\mathcal{V}_{i}^{(2)}$
are independent of $\epsilon$, while the set $\mathcal{V}_{i}^{(3)}$
depends on $\epsilon$. In the following proposition, we extend the
flatness result obtained in Proposition \ref{p58} for $\psi_{\epsilon}$
on $\mathcal{V}_{i}^{(1)}$ to $\mathcal{V}_{i}^{(3)}$.
\begin{prop}
\label{p49}For all $i\in S$, we have
\[
\left\Vert \psi_{\epsilon}-\mathbf{q}_{\epsilon}(i)\right\Vert _{L^{\infty}(\mathcal{V}_{i}^{(3)})}=o_{\epsilon}(1)\lambda_{\epsilon}\;.
\]
\end{prop}

We divide the proof of this proposition into several lemmas. We write
\[
\mathcal{V}^{(k)}=\bigcup_{j\in S}\mathcal{V}_{j}^{(k)}\;\;\;\;\;;\;k\in\{1,\,2,\,3\}\;.
\]
For $i\in S$, we denote by $\phi_{\epsilon}^{i}:\mathbb{R}^{d}\rightarrow\mathbb{R}$
the unique solution of the boundary problem
\[
\begin{cases}
(\mathcal{L}_{\epsilon}\phi_{\epsilon}^{i})(\boldsymbol{x})=0 & \text{ if }\boldsymbol{x}\in\mathbb{R}^{d}\setminus\mathcal{V}^{(1)}\;,\\
\phi_{\epsilon}^{i}(\boldsymbol{x})=\mathbf{1}\{\boldsymbol{x}\in\mathcal{V}_{i}^{(1)}\} & \text{ if }\boldsymbol{x}\in\mathcal{V}^{(1)}\;.
\end{cases}
\]
The function $\phi_{\epsilon}^{i}$ is called the equilibrium potential
between $\mathcal{V}^{(1)}$ and $\mathcal{V}^{(1)}\setminus\mathcal{V}_{i}^{(1)}$.
\begin{lem}
\label{lem410}For $i\in S$, it holds that
\[
\phi_{\epsilon}^{i}(\boldsymbol{x})=\begin{cases}
1-o_{\epsilon}(1) & \text{if }\boldsymbol{x}\in\mathcal{V}_{i}^{(3)}\;,\\
o_{\epsilon}(1) & \text{if }\boldsymbol{x}\in\mathcal{V}^{(3)}\setminus\mathcal{V}_{i}^{(3)}\;.
\end{cases}
\]
\end{lem}

\begin{proof}
By \cite[Corollary 4.8]{BEGK2} with $A=\mathcal{V}^{(1)}\setminus\mathcal{V}_{i}^{(1)}$
and $D=\mathcal{V}_{i}^{(1)}$ we can deduce that, for all $\boldsymbol{x}\in\mathcal{V}^{(3)}\setminus\mathcal{V}_{i}^{(3)}$
\begin{equation}
\phi_{\epsilon}^{i}(\boldsymbol{x})\le C\epsilon^{-1/2}e^{-[H-(H-\frac{1}{4}J^{2}\delta^{2})]/\epsilon}=C\epsilon^{J^{2}/4-1/2}=o_{\epsilon}(1)\label{e4101}
\end{equation}
since we take $J$ large enough. On the other hand, with the selection
$A=\mathcal{V}_{i}^{(1)}$ and $D=\mathcal{V}^{(1)}\setminus\mathcal{V}_{i}^{(1)}$,
we obtain from \cite[Corollary 4.8]{BEGK2} that\footnote{We implicitly use $h_{A,\,B}=1-h_{B,\,A}$ where $h_{A,\,B}$ is the
equilibrium potential introduced in \cite{BEGK2}.}, for all $\boldsymbol{x}\in\mathcal{V}_{i}^{(3)}$,
\begin{equation}
1-\phi_{\epsilon}^{i}(\boldsymbol{x})\le C\epsilon^{-1/2}e^{-[H-(H-\frac{1}{4}J^{2}\delta^{2})]/\epsilon}=C\epsilon^{J^{2}/4-1/2}=o_{\epsilon}(1)\;.\label{e4102}
\end{equation}
The proof is completed by \eqref{e4101} and \eqref{e4102}.
\end{proof}
Now we define $\widetilde{\psi}_{\epsilon}:\mathbb{R}^{d}\rightarrow\mathbb{R}$
as
\[
\widetilde{\psi}_{\epsilon}(\boldsymbol{x})=\sum_{i\in S}\mathbf{q}_{\epsilon}(i)\,\phi_{\epsilon}^{i}(\boldsymbol{x})\;\;\;\;\;;\;\boldsymbol{x}\in\mathbb{R}^{d}\;.
\]
Then, we can readily deduce the following estimate from the previous
lemma.
\begin{lem}
\label{lem411}For all $i\in S$, we have
\[
\left\Vert \widetilde{\psi}_{\epsilon}-\mathbf{q}_{\epsilon}(i)\right\Vert _{L^{\infty}(\mathcal{V}_{i}^{(3)})}=o_{\epsilon}(1)\lambda_{\epsilon}\;.
\]
\end{lem}

\begin{proof}
Fix $i\in S$. For $\boldsymbol{x}\in\mathcal{V}_{i}^{(3)}$, we can
write
\[
\widetilde{\psi}_{\epsilon}(\boldsymbol{x})-\mathbf{q}_{\epsilon}(i)=-\mathbf{q}_{\epsilon}(i)(1-\phi_{\epsilon}^{i}(\boldsymbol{x}))+\sum_{j\in S\setminus\{i\}}\mathbf{q}_{\epsilon}(j)\,\phi_{\epsilon}^{j}(\boldsymbol{x})
\]
Since $1-\phi_{\epsilon}^{i}(\boldsymbol{x})=o_{\epsilon}(1)$ and
$\phi_{\epsilon}^{j}(\boldsymbol{x})=o_{\epsilon}(1)$ for all $j\in S\setminus\{i\}$
by Lemma \ref{lem410}, we obtain
\[
|\widetilde{\psi}_{\epsilon}(\boldsymbol{x})-\mathbf{q}_{\epsilon}(i)|=o_{\epsilon}(1)\max_{i\in S}|\mathbf{q}_{\epsilon}(i)|\le o_{\epsilon}(1)\Vert\psi_{\epsilon}\Vert_{\infty}\;.
\]
It suffices to recall \eqref{se6} to complete the proof.
\end{proof}
We finally claim that $\widetilde{\psi}_{\epsilon}$ approximates
$\psi_{\epsilon}$.
\begin{lem}
\label{lem412}For all $i\in S$, we have
\[
\left\Vert \widetilde{\psi}_{\epsilon}-\psi_{\epsilon}\right\Vert _{\infty}=o_{\epsilon}(1)\lambda_{\epsilon}\;.
\]
\end{lem}

\begin{proof}
Write $\widehat{\psi}_{\epsilon}=\widetilde{\psi}_{\epsilon}-\psi_{\epsilon}$.
First, by Proposition \ref{p58} and Lemmas \ref{lem411}, we have
that
\begin{equation}
\Vert\widehat{\psi}_{\epsilon}\Vert_{L^{\infty}(\mathcal{V}^{(1)})}=o_{\epsilon}(1)\lambda_{\epsilon}\;.\label{e4121}
\end{equation}
Since we have $\mathcal{L}_{\epsilon}\widehat{\psi}_{\epsilon}\equiv0$
on $\mathbb{R}^{d}\setminus\mathcal{V}^{(1)}$, we can write
\[
\widehat{\psi}_{\epsilon}(\boldsymbol{x})=\mathbb{E}_{\boldsymbol{x}}^{\epsilon}[\widehat{\psi}_{\epsilon}(\boldsymbol{x}_{\epsilon}(\tau_{\mathcal{V}^{(1)}}))]\;\;\;\text{for all }\boldsymbol{x}\in\mathbb{R}^{d}\setminus\mathcal{V}^{(1)}\;,
\]
where $\tau_{\mathcal{V}^{(1)}}$ denotes the hitting time of the
set $\mathcal{V}^{(1)}$. Hence, by \eqref{e4121}, we also have
\begin{equation}
\Vert\widehat{\psi}_{\epsilon}\Vert_{L^{\infty}(\mathbb{R}^{d}\setminus\mathcal{V}^{(1)})}\le\Vert\widehat{\psi}_{\epsilon}\Vert_{L^{\infty}(\mathcal{V}^{(1)})}=o_{\epsilon}(1)\lambda_{\epsilon}\;.\label{e4122}
\end{equation}
The proof is completed by \eqref{e4121} and \eqref{e4122}.
\end{proof}
Now we are ready to conclude the proof of Proposition \ref{p49}
\begin{proof}[Proof of Proposition \ref{p49}]
The proof directly follows from Lemmas \ref{lem411} and \ref{lem412}.
\end{proof}

\subsection{\label{s55}Characterization of $\mathbf{q}_{\epsilon}$ on deepest
valleys}

In the previous subsection, we proved that if the constant $\lambda_{\epsilon}$
is bounded above, then for every $i\in S$, the function $\psi_{\epsilon}(x)-\mathbf{q}_{\epsilon}(i)$
is almost $0$ in each valley $\mathcal{V}_{i}^{(1)}$. This boundedness
of $\lambda_{\epsilon}$ will be established later in \eqref{bdl}.
In this subsection, we shall prove that, for each $i\in S_{\star}$,
the value $\mathbf{q}_{\epsilon}(i)$ is close to $\mathbf{f}(i)$
up to a constant $c_{\epsilon}$ that does not depend on $i$. The
following is a formulation of this result.
\begin{prop}
\label{p59}For all small enough $\epsilon>0$, there exists a constant
$c_{\epsilon}$ such that, for all $i\in S_{\star}$,
\[
\left|\mathbf{q}_{\epsilon}(i)-\mathbf{f}(i)-c_{\epsilon}\right|=o_{\epsilon}(1)\,\lambda_{\epsilon}\;.
\]
\end{prop}

Indeed, this characterization of $\mathbf{q}_{\epsilon}$ is the main
innovation of the current work. We shall use the test function constructed
in Section \ref{s32} in a novel manner to establish Proposition~4.9.
For each $\epsilon>0$, we consider a function $\mathbf{h}_{\epsilon}:S_{\star}\rightarrow\mathbb{R}$
and write $\widetilde{\mathbf{h}}_{\epsilon}:S\rightarrow\mathbb{R}$
for its harmonic extension as was introduced in Section \ref{sec52}.
Our selection for $\mathbf{h}_{\epsilon}$ will be revealed at the
last stage of the proof (cf. \eqref{ecc6}). To simplify the notation,
we write
\begin{equation}
F_{\epsilon}:=F_{\epsilon}^{\widetilde{\mathbf{h}}_{\epsilon}}\;,\label{nos}
\end{equation}
where the notation $F_{\epsilon}^{\widetilde{\mathbf{h}}_{\epsilon}}$
was introduced in Section \ref{s32}. We denote by $\Vert\mathbf{h}_{\epsilon}\Vert_{\infty}$
and

$\Vert\widetilde{\mathbf{h}}_{\epsilon}\Vert_{\infty}$ the maximum
of $|\mathbf{h}_{\epsilon}|$ and $|\widetilde{\mathbf{h}}_{\epsilon}|$
on $S_{\star}$ and $S$, respectively. Using a discrete Maximum Principle,
one can readily verify that $\Vert\mathbf{h}_{\epsilon}\Vert_{\infty}=\Vert\widetilde{\mathbf{h}}_{\epsilon}\Vert_{\infty}$.

Since $\psi_{\epsilon}$ satisfies the equation \eqref{p321} and
since $F_{\epsilon}\equiv\widetilde{\mathbf{h}}_{\epsilon}(i)=\mathbf{h}_{\epsilon}(i)$
on $\mathcal{V}'_{i}$, $i\in S_{\star}$, we have the identity
\begin{equation}
\theta_{\epsilon}\int_{\mathbb{R}^{d}}F_{\epsilon}(\boldsymbol{x})\,(\mathscr{L}_{\epsilon}\psi_{\epsilon})(\boldsymbol{x})\,\mu_{\epsilon}(d\boldsymbol{x})=\sum_{i\in S_{\star}}\mathbf{h}_{\epsilon}(i)\,(L_{\mathbf{y}}\mathbf{f})(i)\,\mathbf{a}_{\epsilon}(i)\,\bar{\zeta}^{i}\;.\label{emm}
\end{equation}
In order to prove Proposition \ref{p59}, we compute two sides of
\eqref{emm} separately. From the comparison of these computations,
we obtain the characterization described in Proposition \ref{p59}.

The right-hand side of \eqref{emm} is relatively easy to compute.
By Proposition \ref{p21} and \eqref{aei}, we have
\[
\mathbf{a}_{\epsilon}(i)\,\bar{\zeta}^{i}=(1+o_{\epsilon}(1))(\nu_{i}/\nu_{\star})
\]
and thus we can rewrite the right-hand side of \eqref{emm} as
\begin{equation}
\sum_{i\in S_{\star}}\mathbf{h}_{\epsilon}(i)\,(L_{\mathbf{y}}\mathbf{f})(i)\,\mathbf{a}_{\epsilon}(i)\,\bar{\zeta}^{i}=-D_{\mathbf{y}}(\mathbf{h}_{\epsilon},\,\mathbf{f})+o_{\epsilon}(1)\Vert\mathbf{h}_{\epsilon}\Vert_{\infty}\;.\label{e59}
\end{equation}
The main difficulty of the proof lies on the computation of the left-hand
side of \eqref{emm}. We carry out this computation in several lemmas
below.
\begin{lem}
\label{l510}With the notations above, it holds that
\begin{align}
 & \theta_{\epsilon}\int_{\mathbb{R}^{d}}F_{\epsilon}(\boldsymbol{x})\,(\mathscr{L}_{\epsilon}\psi_{\epsilon})(\boldsymbol{x})\,\mu_{\epsilon}(d\boldsymbol{x})\nonumber \\
 & \qquad\qquad=-\theta_{\epsilon}\,\epsilon\,\sum_{\boldsymbol{\sigma}\in\mathcal{S}}\int_{\mathcal{B}_{\boldsymbol{\sigma}}^{\epsilon}}(\nabla F_{\epsilon}\cdot\nabla\psi_{\epsilon})(\boldsymbol{x})\,\mu_{\epsilon}(d\boldsymbol{x})+o_{\epsilon}(1)\,\lambda_{\epsilon}^{1/2}\,\Vert\mathbf{h}_{\epsilon}\Vert_{\infty}\;.\label{e510}
\end{align}
\end{lem}

\begin{proof}
By the divergence theorem, the left-hand side of \eqref{e510} is
equal to
\begin{equation}
-\theta_{\epsilon}\,\epsilon\,\int_{\mathbb{R}^{d}}(\nabla F_{\epsilon}\cdot\nabla\psi_{\epsilon})(\boldsymbol{x})\,\mu_{\epsilon}(d\boldsymbol{x})\;.\label{e5101}
\end{equation}
By the definition of $F_{\epsilon}=F_{\epsilon}^{\widetilde{\mathbf{h}}_{\epsilon}}$,
we have that
\begin{equation}
\nabla F_{\epsilon}\equiv0\;\;\text{in }\mathcal{W}_{i}^{\epsilon}\text{ for all }i\in S\;.\label{e511}
\end{equation}
Since
\[
\mathcal{H}^{\epsilon}\setminus\left(\bigcup_{i\in S}\mathcal{W}_{i}^{\epsilon}\right)=\bigcup_{\boldsymbol{\sigma}\in\mathcal{S}}\mathcal{B}_{\boldsymbol{\sigma}}^{\epsilon}\;,
\]
it suffices to show that
\begin{equation}
-\theta_{\epsilon}\,\epsilon\,\int_{(\mathcal{H}^{\epsilon})^{c}}(\nabla F_{\epsilon}\cdot\nabla\psi_{\epsilon})(\boldsymbol{x})\,\mu_{\epsilon}(d\boldsymbol{x})=o_{\epsilon}(1)\,\lambda_{\epsilon}^{1/2}\,\Vert\mathbf{h}_{\epsilon}\Vert_{\infty}\;.\label{e512}
\end{equation}
By the Cauchy-Schwarz inequality, the square of the left-hand side
of \eqref{e512} is bounded above by
\[
\theta_{\epsilon}\epsilon\left(\int_{(\mathcal{H}^{\epsilon})^{c}}|\nabla F_{\epsilon}(\boldsymbol{x})|^{2}\mu_{\epsilon}(d\boldsymbol{x})\right)^{\frac{1}{2}}\,\left(\int_{(\mathcal{H}^{\epsilon})^{c}}|\nabla\psi_{\epsilon}(\boldsymbol{x})|^{2}\mu_{\epsilon}(d\boldsymbol{x})\right)^{\frac{1}{2}}\;.
\]
By \eqref{eka3} and \eqref{hg2}, the last expression is $o_{\epsilon}(1)\,\lambda_{\epsilon}^{1/2}\,\Vert\mathbf{h}_{\epsilon}\Vert_{\infty}$.
Thus, \eqref{e512} follows.
\end{proof}
Recall the function $f_{\epsilon}^{\boldsymbol{\sigma}}$ from \eqref{cee}.
The estimate below corresponds to that of each summand on the right-hand
side of \eqref{e510}.
\begin{lem}
\label{l511}For $i,\,j\in S$ with $i<j$ and for $\boldsymbol{\sigma}\in\mathcal{W}_{i,\,j}$,
it holds that
\begin{equation}
\theta_{\epsilon}\,\epsilon\,\int_{\mathcal{B}_{\boldsymbol{\sigma}}^{\epsilon}}(\nabla f_{\epsilon}^{\boldsymbol{\sigma}}\cdot\nabla\psi_{\epsilon})(\boldsymbol{x})\,\mu_{\epsilon}(d\boldsymbol{x})=\frac{\omega_{\boldsymbol{\sigma}}}{\nu_{\star}}\left[\mathbf{q}_{\epsilon}(j)-\mathbf{q}_{\epsilon}(i)\right]+o_{\epsilon}(1)\,\lambda_{\epsilon}\;.\label{e52}
\end{equation}
\end{lem}

\begin{proof}
Recall the decomposition of boundary of $\mathcal{B}_{\boldsymbol{\sigma}}^{\epsilon}$
from \eqref{eqb}. By applying the divergence theorem to the left-hand
side of \eqref{e52}, we can write
\begin{equation}
\theta_{\epsilon}\,\epsilon\,\int_{\mathcal{B}_{\boldsymbol{\sigma}}^{\epsilon}}(\nabla f_{\epsilon}^{\boldsymbol{\sigma}}\cdot\nabla\psi_{\epsilon})(\boldsymbol{x})\,\mu_{\epsilon}(d\boldsymbol{x})=A_{1}+A_{2}+A_{3}+A_{4}\;,\label{e52-1}
\end{equation}
where
\begin{align*}
A_{1} & =-\theta_{\epsilon}\,\int_{\mathcal{B}_{\boldsymbol{\sigma}}^{\epsilon}}(\mathscr{L}_{\epsilon}f_{\epsilon}^{\boldsymbol{\sigma}})(\boldsymbol{x})\,\psi_{\epsilon}(\boldsymbol{x})\,\mu_{\epsilon}(d\boldsymbol{x})\;,\\
A_{2} & =\theta_{\epsilon}\,\epsilon\,\int_{\partial_{0}\mathcal{B}_{\boldsymbol{\sigma}}^{\epsilon}}\big[(\nabla f_{\epsilon}^{\boldsymbol{\sigma}})(\boldsymbol{x})\cdot\boldsymbol{n}_{\mathcal{B}_{\boldsymbol{\sigma}}^{\epsilon}}\big]\,\psi_{\epsilon}(\boldsymbol{x})\,\hat{\mu}_{\epsilon}(\boldsymbol{x})\,\sigma(d\boldsymbol{x})\;,\\
A_{3} & =\theta_{\epsilon}\,\epsilon\,\int_{\partial_{+}\mathcal{B}_{\boldsymbol{\sigma}}^{\epsilon}}\big[(\nabla f_{\epsilon}^{\boldsymbol{\sigma}})(\boldsymbol{x})\cdot\boldsymbol{n}_{\mathcal{B}_{\boldsymbol{\sigma}}^{\epsilon}}\big]\,\psi_{\epsilon}(\boldsymbol{x})\,\hat{\mu}_{\epsilon}(\boldsymbol{x})\,\sigma(d\boldsymbol{x})\;,\\
A_{4} & =\theta_{\epsilon}\,\epsilon\,\int_{\partial_{-}\mathcal{B}_{\boldsymbol{\sigma}}^{\epsilon}}\big[(\nabla f_{\epsilon}^{\boldsymbol{\sigma}})(\boldsymbol{x})\cdot\boldsymbol{n}_{\mathcal{B}_{\boldsymbol{\sigma}}^{\epsilon}}\big]\,\psi_{\epsilon}(\boldsymbol{x})\,\hat{\mu}_{\epsilon}(\boldsymbol{x})\,\sigma(d\boldsymbol{x})\;,
\end{align*}
where the vector $\boldsymbol{n}_{\mathcal{B}_{\epsilon}^{\boldsymbol{\sigma}}}$
denotes the outward unit normal vector to the domain $\mathcal{B}_{\boldsymbol{\sigma}}^{\epsilon}$,
and $\sigma(d\boldsymbol{x})$ represents the surface integral. We
now compute these four expressions.

Without loss of generality, we may assume that $\boldsymbol{\sigma}=0$.
First, we claim that $A_{1}$ and $A_{2}$ are negligible in the sense
that
\begin{equation}
A_{1}=o_{\epsilon}(1)\,\lambda_{\epsilon}\;\;\;\text{and\;\;\;}A_{2}=o_{\epsilon}(1)\,\lambda_{\epsilon}\;.\label{a1}
\end{equation}
The estimate for $A_{1}$ is immediate from Lemma \ref{lem35} and
\eqref{se6}. For $A_{2}$, notice first that by the definition \eqref{cee}
of $f_{\epsilon}^{\boldsymbol{\sigma}}$, we can write
\begin{equation}
(\nabla f_{\epsilon}^{\boldsymbol{\sigma}})(\boldsymbol{x})=\frac{1}{c_{\epsilon}^{\,\boldsymbol{\sigma}}}\sqrt{\frac{\lambda_{1}^{\boldsymbol{\sigma}}}{2\pi\epsilon}}e^{-\frac{\lambda_{1}^{\boldsymbol{\sigma}}}{2\epsilon}(\boldsymbol{x}\cdot\boldsymbol{v}_{1}^{\boldsymbol{\sigma}})^{2}}\boldsymbol{v}_{1}^{\boldsymbol{\sigma}}\;.\label{gradf}
\end{equation}
By inserting this into $A_{2}$, and applying \eqref{b0}, \eqref{ce},
and \eqref{se6}, we are able to deduce
\begin{equation}
|A_{2}|\le C\,\theta_{\epsilon}\,\epsilon^{1/2}\,\lambda_{\epsilon}\,Z_{\epsilon}^{-1}\,e^{-(H+J^{2}\delta^{2})/\epsilon}\,\delta^{d-1}\;=o_{\epsilon}(1)\,\lambda_{\epsilon}\;.\label{a2}
\end{equation}
Here we have used trivial facts such as $|\boldsymbol{v}_{1}^{\boldsymbol{\sigma}}\cdot\boldsymbol{n}_{\mathcal{B}_{\boldsymbol{\sigma}}^{\epsilon}}|\le1$,
$e^{-\frac{\lambda_{1}^{\boldsymbol{\sigma}}}{2\epsilon}(\boldsymbol{x}\cdot\boldsymbol{v}_{1}^{\boldsymbol{\sigma}})^{2}}\le1$,
and that the $\sigma$-measure of $\partial_{0}\mathcal{B}_{\boldsymbol{\sigma}}^{\epsilon}$
is of order $\delta^{d-1}$.

Next, we shall prove that
\begin{equation}
A_{3}=\frac{\omega_{\boldsymbol{\sigma}}}{\nu_{\star}}\,\,\mathbf{q}_{\epsilon}(j)+o_{\epsilon}(1)\,\lambda_{\epsilon}\;\;\;\text{and\;\;\;}A_{4}=-\frac{\omega_{\boldsymbol{\sigma}}}{\nu_{\star}}\,\mathbf{q}_{\epsilon}(i)+o_{\epsilon}(1)\,\lambda_{\epsilon}\;.\label{a3}
\end{equation}
Since the proofs for these two estimates are identical, we only focus
on the former. Note that the surface $\partial_{+}\mathcal{B}_{\boldsymbol{\sigma}}^{\epsilon}$
is flat, and hence the outward normal vector $\boldsymbol{n}_{\mathcal{B}_{\boldsymbol{\sigma}}^{\epsilon}}$
is merely equal to $\boldsymbol{v}_{1}^{\boldsymbol{\sigma}}$. Hence,
by \eqref{ep3}, \eqref{ce} and \eqref{gradf} we can rewrite $A_{3}$
as
\[
A_{3}=(1+o_{\epsilon}(1))\theta_{\epsilon}\,\epsilon\,\sqrt{\frac{\lambda_{1}^{\boldsymbol{\sigma}}}{2\pi\epsilon}}\,\frac{1}{(2\pi\epsilon)^{d/2}e^{-h/\epsilon}\nu_{\star}}\,\int_{\partial_{+}\mathcal{B}_{\boldsymbol{\sigma}}^{\epsilon}}e^{-\frac{\lambda_{1}^{\boldsymbol{\sigma}}}{2\epsilon}(\boldsymbol{x}\cdot\boldsymbol{v}_{1}^{\boldsymbol{\sigma}})^{2}-\frac{U(\boldsymbol{x})}{\epsilon}}\psi_{\epsilon}(\boldsymbol{x})\,\sigma(d\boldsymbol{x})\;.
\]
By the Taylor expansion, we have
\[
U(\boldsymbol{x})=H+\frac{1}{2}\left(-\lambda_{1}^{\boldsymbol{\sigma}}(\boldsymbol{x}\cdot\boldsymbol{v}_{1}^{\boldsymbol{\sigma}})^{2}+\sum_{i=2}^{d}\lambda_{i}^{\boldsymbol{\sigma}}(\boldsymbol{x}\cdot\boldsymbol{v}_{i}^{\boldsymbol{\sigma}})^{2}\right)+o(\delta^{2})\;.
\]
Inserting this into the penultimate display, we can reorganize the
right-hand side so that
\begin{equation}
A_{3}=(1+o_{\epsilon}(1))\,\frac{\sqrt{\lambda_{1}^{\boldsymbol{\sigma}}}}{2\pi\nu_{\star}}\,\int_{\partial_{+}\mathcal{B}_{\boldsymbol{\sigma}}^{\epsilon}}\frac{1}{(2\pi\epsilon)^{(d-1)/2}}e^{-\frac{1}{2\epsilon}\sum_{i=2}^{d}\lambda_{i}^{\boldsymbol{\sigma}}(\boldsymbol{x}\cdot\boldsymbol{v}_{i}^{\boldsymbol{\sigma}})^{2}}\psi_{\epsilon}(\boldsymbol{x})\,\sigma(d\boldsymbol{x})\;.\label{a31}
\end{equation}

Now we introduce a change of variable to estimate the last integral.
Define a map $g_{\epsilon}^{\boldsymbol{\sigma}}:\mathbb{R}^{d-1}\rightarrow\mathbb{R}^{d}$
as, for $\boldsymbol{y}=(y_{2},\,\cdots,\,y_{d})\in\mathbb{R}^{d-1}$,
\begin{equation}
g_{\epsilon}^{\boldsymbol{\sigma}}(\boldsymbol{y})=\frac{J\delta}{\sqrt{\lambda_{1}}}\boldsymbol{v}_{1}^{\boldsymbol{\sigma}}+\sum_{k=2}^{d}\sqrt{\frac{\epsilon}{\lambda_{k}}}y_{k}\boldsymbol{v}_{k}^{\boldsymbol{\sigma}}\;,\label{ecc1}
\end{equation}
(recall $\boldsymbol{\sigma}=0$). Notice here that $\partial_{+}\mathcal{B}_{\boldsymbol{\sigma}}^{\epsilon}\subset g_{\epsilon}^{\boldsymbol{\sigma}}(\mathbb{R}^{d-1})$.
Write
\[
\mathcal{D}_{\boldsymbol{\sigma}}^{\epsilon}=(g_{\epsilon}^{\boldsymbol{\sigma}})^{-1}(\partial_{+}\mathcal{B}_{\boldsymbol{\sigma}}^{\epsilon})\subset\mathbb{R}^{d-1}\;.
\]
Then, by a change of variable $\boldsymbol{x}=g_{\epsilon}^{\boldsymbol{\sigma}}(\boldsymbol{y})$,
we can rewrite \eqref{a31} as
\begin{equation}
A_{3}=(1+o_{\epsilon}(1))\,\frac{1}{\nu_{\star}}\frac{\sqrt{\lambda_{1}^{\boldsymbol{\sigma}}}}{2\pi\sqrt{\prod_{k=2}^{d}\lambda_{k}^{\boldsymbol{\sigma}}}}\int_{\mathcal{D}_{\boldsymbol{\sigma}}^{\epsilon}}\frac{1}{(2\pi)^{(d-1)/2}}e^{-\frac{1}{2}|\boldsymbol{y}|^{2}}\psi_{\epsilon}(g^{\boldsymbol{\sigma}}(\boldsymbol{y}))\,d\boldsymbol{y}\;.\label{a33}
\end{equation}
Now we analyze $\mathcal{D}_{\boldsymbol{\sigma}}^{\epsilon}$. For
$\boldsymbol{y}\in\mathcal{D}_{\boldsymbol{\sigma}}^{\epsilon}$,
we note that $|g^{\boldsymbol{\sigma}}(\boldsymbol{y})-\boldsymbol{\sigma}|=O(\delta)$
and thus by the Taylor expansion,
\[
U(g^{\boldsymbol{\sigma}}(\boldsymbol{y}))=H-\frac{1}{2}J^{2}\delta^{2}+\frac{\epsilon}{2}\sum_{k=2}^{d}y_{k}^{2}+o(\delta^{2})\;.
\]
Denote by $\mathcal{Q}_{d-1}(r)$ the $(d-1)$-dimensional ball of
radius $r>0$,\textbf{ }centered at origin. Then, for $\boldsymbol{y}\in\mathcal{Q}_{d-1}(\frac{J}{2}\sqrt{\log\frac{1}{\epsilon}})$,
by the previous display we have that
\[
U(g^{\boldsymbol{\sigma}}(\boldsymbol{y}))\le H-\frac{1}{2}J^{2}\delta^{2}+\frac{1}{8}J^{2}\delta^{2}+o(\delta^{2})<H-\frac{1}{4}J^{2}\delta^{2}
\]
for all sufficiently small $\epsilon>0$. For such $\epsilon$, we
can conclude that $\boldsymbol{y}\in\partial_{+}\mathcal{B}_{\boldsymbol{\sigma}}^{\epsilon}\cap\mathcal{V}_{j}^{(3)}$
by definition of $\mathcal{V}_{j}^{(3)}$ and therefore by Proposition
\ref{p49}, we have that $\psi_{\epsilon}(g^{\boldsymbol{\sigma}}(\boldsymbol{y}))=\mathbf{q}_{\epsilon}(j)+o_{\epsilon}(1)\lambda_{\epsilon}$.
Consequently, we have
\[
\int_{\mathcal{Q}_{d-1}(\frac{J}{2}\sqrt{\log\frac{1}{\epsilon}})}\,\frac{1}{(2\pi)^{(d-1)/2}}e^{-\frac{1}{2}|\boldsymbol{y}|^{2}}\psi_{\epsilon}(g^{\boldsymbol{\sigma}}(\boldsymbol{y}))\,d\boldsymbol{y}=(1+o_{\epsilon}(1))\mathbf{q}_{\epsilon}(j)+o_{\epsilon}(1)\lambda_{\epsilon}\;,
\]
because the integral of the probability density function of the $(d-1)$-dimensional
standard normal distribution on $\mathcal{Q}_{d-1}(\frac{J}{2}\sqrt{\log\frac{1}{\epsilon}})$
is $1+o_{\epsilon}(1)$.

On the other hand, by \eqref{se6},
\begin{align*}
 & \left|\int_{\mathcal{D}_{\boldsymbol{\sigma}}^{\epsilon}\setminus\mathcal{Q}_{d-1}(\frac{J}{2}\sqrt{\log\frac{1}{\epsilon}})}\frac{1}{(2\pi)^{(d-1)/2}}e^{-\frac{1}{2}|\boldsymbol{y}|^{2}}\psi_{\epsilon}(g^{\boldsymbol{\sigma}}(\boldsymbol{y}))\,d\boldsymbol{y}\right|\\
 & \qquad\qquad\le\Vert\psi_{\epsilon}\Vert_{\infty}\int_{\mathcal{Q}_{d-1}(\frac{J}{2}\sqrt{\log\frac{1}{\epsilon}})^{c}}\frac{1}{(2\pi)^{(d-1)/2}}e^{-\frac{1}{2}|\boldsymbol{y}|^{2}}\,d\boldsymbol{y}=o_{\epsilon}(1)\lambda_{\epsilon}\;.
\end{align*}
By the two last centered displays and by the definition of $\omega_{\boldsymbol{\sigma}}$,
we can rewrite \eqref{a33} as
\[
A_{3}=\frac{\omega_{\boldsymbol{\sigma}}}{\nu_{\star}}\left[(1+o_{\epsilon}(1))\mathbf{q}_{\epsilon}(j)+o_{\epsilon}(1)\lambda_{\epsilon}\right]\;.
\]
The proof of \eqref{a3} is completed by recalling that the fact that
by \eqref{se6}
\[
|\mathbf{q}_{\epsilon}(i)|\le\left\Vert \psi_{\epsilon}\right\Vert _{\infty}\le(1+o_{\epsilon}(1))\lambda_{\epsilon}.
\]
By combining \eqref{e52-1}, \eqref{a1}, and \eqref{a3}, we complete
the proof.
\end{proof}
\begin{lem}
\label{lem510} Assume that $\mathbf{f}\neq0$. It then holds,
\begin{equation}
\theta_{\epsilon}\int_{\mathbb{R}^{d}}F_{\epsilon}(\boldsymbol{x})\,(\mathscr{L}_{\epsilon}\psi_{\epsilon})(\boldsymbol{x})\,\hat{\mu}_{\epsilon}(\boldsymbol{x})d\boldsymbol{x}=-\frac{1}{\nu_{\star}}D_{\mathbf{x}}(\widetilde{\mathbf{h}}_{\epsilon},\,\mathbf{q}_{\epsilon})+o_{\epsilon}(1)\,\lambda_{\epsilon}\,\Vert\mathbf{h}_{\epsilon}\Vert_{\infty}\;.\label{ee52}
\end{equation}
\end{lem}

\begin{proof}
By Lemma \ref{l510} and the definition \eqref{nos} (cf. \eqref{fq})
of $F_{\epsilon}$ we can rewrite the left-hand side as
\begin{equation}
-\theta_{\epsilon}\,\epsilon\,\sum_{1\le i<j<K}\left[(\widetilde{\mathbf{h}}_{\epsilon}(j)-\widetilde{\mathbf{h}}_{\epsilon}(i))\sum_{\boldsymbol{\sigma}\in\mathcal{W}_{i,\,j}}\int_{\mathcal{B}_{\boldsymbol{\sigma}}^{\epsilon}}(\nabla f_{\epsilon}^{\boldsymbol{\sigma}}\cdot\nabla\psi_{\epsilon})(\boldsymbol{x})\,\hat{\mu}_{\epsilon}(\boldsymbol{x})d\boldsymbol{x}\right]+o_{\epsilon}(1)\,\lambda_{\epsilon}^{1/2}\,\Vert\mathbf{h}_{\epsilon}\Vert_{\infty}\;.\label{e51}
\end{equation}
From this and Lemma \ref{l511}, we deduce that the left-hand side
of \eqref{ee52} equals to
\[
-\frac{1}{\nu_{\star}}D_{\mathbf{x}}(\widetilde{\mathbf{h}}_{\epsilon},\,\mathbf{q}_{\epsilon})+o_{\epsilon}(1)\,\lambda_{\epsilon}^{1/2}\,\Vert\mathbf{h}_{\epsilon}\Vert_{\infty}+o_{\epsilon}(1)\,\lambda_{\epsilon}\,\Vert\widetilde{\mathbf{h}}_{\epsilon}\Vert_{\infty}.
\]
Therefore, the proof is completed because by Maximum Principle $\Vert\widetilde{\mathbf{h}}_{\epsilon}\Vert_{\infty}=\Vert\mathbf{h}_{\epsilon}\Vert_{\infty}$,
and $\lambda_{\epsilon}$ is uniformly positive whenever $\mathbf{f}\neq0$
by Proposition \ref{p53}.
\end{proof}
Now we are ready to prove Proposition \ref{p59}.
\begin{proof}[Proof of Proposition \ref{p59}]
The proof of the Proposition is trivial when $\mathbf{f}=0$, because
we may choose $\psi_{\epsilon}=c_{\epsilon}=0.$ From now on, we assume
that $\mathbf{f}\neq0$. By \eqref{e59}, Proposition \ref{p53},
and Lemma \ref{lem510}, we have
\begin{equation}
D_{\mathbf{y}}(\mathbf{h}_{\epsilon},\,\mathbf{f})=\frac{1}{\nu_{\star}}D_{\mathbf{x}}(\widetilde{\mathbf{h}}_{\epsilon},\,\mathbf{q}_{\epsilon})+o_{\epsilon}(1)\,\lambda_{\epsilon}\,\Vert\mathbf{h}_{\epsilon}\Vert_{\infty}\;.\label{ecc4}
\end{equation}
Denote by $\mathbf{q}_{\epsilon}^{\star}\in\mathbb{R}^{S_{\star}}$
the restriction of $\mathbf{q}_{\epsilon}$ on $S_{\star}$, i.e.,
$\mathbf{q}_{\epsilon}^{\star}(i)=\mathbf{q}_{\epsilon}(i)$ for all
$i\in S_{\star}$, and denote by $\mathbf{\widetilde{q}}_{\epsilon}^{\star}\in\mathbb{R}^{S}$
the harmonic extension of $\mathbf{q}_{\epsilon}^{\star}$ to $S$.
Note that $\mathbf{\widetilde{q}}_{\epsilon}^{\star}$ and $\mathbf{q}_{\epsilon}$
are two different extensions of $\mathbf{q}_{\epsilon}^{\star}\in\mathbb{R}^{S_{\star}}$
to $S$. Thus, by Lemma \ref{lem54} we have
\begin{equation}
D_{\mathbf{x}}(\widetilde{\mathbf{h}}_{\epsilon},\,\mathbf{q}_{\epsilon})=D_{\mathbf{x}}(\widetilde{\mathbf{h}}_{\epsilon},\,\mathbf{\widetilde{q}}_{\epsilon}^{\star})=\nu_{\star}D_{\mathbf{y}}(\mathbf{h}_{\epsilon},\,\mathbf{q}_{\epsilon}^{\star})\;.\label{ecc3}
\end{equation}
Hence, by \eqref{ecc4} and \eqref{ecc3}, we obtain
\begin{equation}
D_{\mathbf{y}}(\mathbf{h}_{\epsilon},\,\mathbf{q}_{\epsilon}^{\star}-\mathbf{f})=o_{\epsilon}(1)\,\lambda_{\epsilon}\,\Vert\mathbf{h}_{\epsilon}\Vert_{\infty}\;.\label{ecc5}
\end{equation}

Finally, let us define the test function $\mathbf{h}_{\epsilon}\in\mathbb{R}^{S_{\star}}$
as
\begin{equation}
\mathbf{h}_{\epsilon}(i):=\mathbf{q}_{\epsilon}(i)-\mathbf{f}(i)-c_{\epsilon}\;\;\text{for all }i\in S_{\star}\;,\label{ecc6}
\end{equation}
where
\begin{equation}
c_{\epsilon}=\frac{1}{|S_{\star}|}\sum_{i\in S_{\star}}\left[\mathbf{q}_{\epsilon}(i)-\mathbf{f}(i)\right]\;.\label{ecc61}
\end{equation}
By inserting this test function $\mathbf{h}_{\epsilon}$ in \eqref{ecc5},
we obtain
\begin{equation}
D_{\mathbf{y}}(\mathbf{h}_{\epsilon},\,\mathbf{h}_{\epsilon})=o_{\epsilon}(1)\,\lambda_{\epsilon}\,\Vert\mathbf{h}_{\epsilon}\Vert_{\infty}\;.\label{ecc7}
\end{equation}
Write
\[
\beta_{\star}=\frac{1}{2\nu_{\star}}\min_{i\in S_{\star},\,j\in S_{\star},\,i\neq j}\beta_{i,\,j}>0\;.
\]
Then, we have
\begin{equation}
D_{\mathbf{y}}(\mathbf{h}_{\epsilon},\,\mathbf{h}_{\epsilon})\ge\beta_{\star}\sum_{i,\,j\in S_{\star}}(\mathbf{h}_{\epsilon}(i)-\mathbf{h}_{\epsilon}(j))^{2}=2\beta_{\star}\,|S_{\star}|\sum_{i\in S_{\star}}\mathbf{h}_{\epsilon}^{2}\ge2\beta_{\star}\,|S_{\star}|^{2}\Vert\mathbf{h}_{\epsilon}\Vert_{\infty}^{2}\;,\label{ecc8}
\end{equation}
where the identity follows from the fact that $\sum_{i\in S_{\star}}\mathbf{h}_{\epsilon}=0$
thanks to our selection \eqref{ecc6} and \eqref{ecc7} of $\mathbf{h}_{\epsilon}$.
By \eqref{ecc7} and \eqref{ecc8}, we obtain
\[
\Vert\mathbf{h}_{\epsilon}\Vert_{\infty}\le o_{\epsilon}(1)\,\lambda_{\epsilon}\;.
\]
This completes the proof since $\mathbf{h}_{\epsilon}(i)=\mathbf{q}_{\epsilon}(i)-\mathbf{f}(i)-c_{\epsilon}$
for $i\in S_{\star}$.
\end{proof}

\subsection{\label{s56}Proof of Theorem \ref{t51}}

Now we are ready to prove Theorem \ref{t51}.
\begin{proof}[Proof of Theorem \ref{t51}]
Define $\phi_{\epsilon}=\psi_{\epsilon}-c_{\epsilon}$ where $c_{\epsilon}$
is the constant appearing in the statement of Proposition \ref{p59}.
Then, by Propositions \ref{p58} and \ref{p59}, we obtain
\[
\left\Vert \phi_{\epsilon}-\mathbf{f}(i)\right\Vert _{L^{\infty}(\mathcal{V}_{i}^{(1)})}=o_{\epsilon}(1)\,\lambda_{\epsilon}\;\;\text{for all }i\in S_{\star}.
\]
Since it already has been shown that $\phi_{\epsilon}$ satisfies
\eqref{p321}, and $\phi_{\epsilon}\in W_{\textrm{loc}}^{2,\,p}(\mathbb{R}^{d})\text{ for all }p\ge1$,
it only remains to show that $\lambda_{\epsilon}$ is bounded above.
By Lemma \ref{lem57}, Proposition \ref{p58}, and \eqref{lame},
we have that
\[
\mathbf{q}_{\epsilon}(1)=-(1+o_{\epsilon}(1))\lambda_{\epsilon}\;\text{ and \;}\mathbf{q}_{\epsilon}(2)=(1+o_{\epsilon}(1))\lambda_{\epsilon}\;.
\]
By combining these results with Proposition \ref{p59}, we obtain
\[
\mathbf{f}(1)=-c_{\epsilon}-(1+o_{\epsilon}(1))\lambda_{\epsilon}\;\text{ and }\;\mathbf{f}(2)=-c_{\epsilon}+(1+o_{\epsilon}(1))\lambda_{\epsilon}\;.
\]
Therefore, we have
\begin{equation}
\lambda_{\epsilon}=\frac{1+o_{\epsilon}(1)}{2}(\mathbf{f}(2)-\mathbf{f}(1))\;.\label{bdl}
\end{equation}
This proves the boundedness of $\lambda_{\epsilon}.$
\end{proof}

\section{\label{sec4}Tightness}

The main result of the current section is the following theorem regarding
the tightness of the family of processes $\{\mathbf{y}_{\epsilon}(\cdot):\epsilon\in(0,\,1]\}$.
\begin{thm}
\label{p31}For all $i\in S_{\star}$ and for any sequence of Borel
probability measures $(\pi_{\epsilon})_{\epsilon>0}$ concentrated
on $\mathcal{V}_{i}$, the family $\{\mathbf{Q}_{\pi_{\epsilon}}^{\epsilon}:\epsilon\in(0,\,1]\}$
is tight on $D([0,\,\infty),\,S_{\star})$, and every limit point
$\mathbf{Q}^{*}$, as $\epsilon\rightarrow0$, of this sequence satisfies
\[
\mathbf{Q}^{*}(\mathbf{x}(0)=i)=1\text{\;\;and\;\;}\mathbf{Q}^{*}(\mathbf{x}(t)\neq\mathbf{x}(t-))=0\;\;\text{for all }t>0\;.
\]
\end{thm}

We first introduce in Subsection \ref{sec43} two main ingredients
of the proof of the tightness. These technical estimates are the tight
bound of the transition time from a valley to other valleys (Proposition
\ref{p33}), and the negligibility of the time spent by $\widehat{\boldsymbol{x}}_{\epsilon}(t)$
in $\Delta$ (Proposition \ref{p34}). These are common technical
steps in the proof of tightness in the metastable situation, and Beltran
and Landim \cite{BL1,BL2} developed a robust methodology to verify
these when the underlying dynamics are discrete Markov chain. In \cite{LS3},
the corresponding tightness when the underlying dynamics is a $1$-dimensional
diffusion is obtained. The common feature for these models which allows
to prove the tightness is the coupling of two trajectories starting
from different points in the same well. Since two diffusion processes
living in $\mathbb{R}^{d}$, $d\ge2$, cannot be exactly coupled,
we have to developed another machinery. We shall use Theorem \ref{t51}
to bound the inter-valleys transition times, and Freidlin-Wentzell
theory \cite{FW1} for the negligibility of the time spent outside
valleys. Then, the proof of Theorem \ref{p31} is given in Subsection
\ref{sec44}.

\subsection{\label{sec43}Two preliminary estimates}

For $\mathcal{A}\subset\mathbb{R}^{d}$, we denote by $H_{\mathcal{A}}$
the hitting time of the set $\mathcal{A}$. Then the hitting time
$H_{\mathcal{V}_{\star}\setminus\mathcal{V}_{i}}$ under the law $\mathbb{P}_{\boldsymbol{x}}^{\epsilon}$,
$\boldsymbol{x}\in\mathcal{V}_{i}$, can be regarded as the transition
time from valley $\mathcal{V}_{i}$ to other deepest valleys. We now
verify that this inter-valley transition time cannot be too small.
\begin{prop}
\label{p33}For all $i\in S_{\star}$, it holds that,
\begin{equation}
\lim_{a\rightarrow0}\limsup_{\epsilon\rightarrow0}\sup_{\boldsymbol{x}\in\mathcal{V}_{i}}\mathbb{P}_{\boldsymbol{x}}^{\epsilon}\left[H_{\mathcal{V}_{\star}\setminus\mathcal{V}_{i}}\le a\theta_{\epsilon}\right]=0\;.\label{ep471}
\end{equation}
\end{prop}

\begin{rem}
The result of Freidlin and Wentzell \cite{FW1} provides that, for
all $i\in S_{\star}$,
\begin{equation}
\limsup_{\epsilon\rightarrow0}\sup_{\boldsymbol{x}\in\mathcal{V}_{i}}\mathbb{P}_{\boldsymbol{x}}^{\epsilon}\left[H_{\mathcal{V}_{\star}\setminus\mathcal{V}_{i}}\le e^{-\eta/\epsilon}\theta_{\epsilon}\right]=0\;\;\,\text{for all }\eta>0\;.\label{erm48}
\end{equation}
This estimate is definitely weaker than \eqref{ep471}. On the other
hand, Bovier et. al. \cite{BEGK2} demonstrated that $\theta_{\epsilon}^{-1}H_{\mathcal{V}_{\star}\setminus\mathcal{V}_{i}}$
converges to an exponential random variable with constant mean, and
this result does implies \eqref{ep471}. However, in this paper, we
provide another proof without using this result. Two main advantages
of our proof of \eqref{ep471} is that it is short, and is has a good
chance to be applicable to the non-reversible case \eqref{e15}; our
proof of \eqref{ep471} relies only on our analysis on the elliptic
equations carried out in the previous section. The reader can readily
notice that this result is a direct consequence of Theorem \ref{t51}.
\end{rem}

\begin{proof}
We fix $i\in S_{\star}$ and $\boldsymbol{x}\in\mathcal{V}_{i}$.
Consider a function $\mathbf{b}_{i}:S_{\star}\rightarrow\mathbb{R}$
given by
\[
\mathbf{b}_{i}(j)=\begin{cases}
0 & \text{if }j=i\\
1 & \text{if }j\in S_{\star}\setminus\{i\}\;.
\end{cases}
\]
Denote by $\phi_{\epsilon}=\phi_{\epsilon}^{i}$ the test function
we obtain in Theorem \ref{t51} for $\mathbf{f}=\mathbf{b}_{i}$.
Then, by Ito's formula and part (2) of Theorem \ref{t51}, we get
\begin{align*}
 & \mathbb{E}_{\boldsymbol{x}}^{\epsilon}\left[\phi_{\epsilon}(\boldsymbol{x}_{\epsilon}(a\theta_{\epsilon}\wedge H_{\mathcal{V}_{\star}\setminus\mathcal{V}_{i}}))\right]\\
 & \quad=\phi_{\epsilon}(\boldsymbol{x})+\sum_{i\in S_{\star}}\mathbb{E}_{\boldsymbol{x}}^{\epsilon}\left[\int_{0}^{a\theta_{\epsilon}\wedge H_{\mathcal{V}_{\star}\setminus\mathcal{V}_{i}}}\theta_{\epsilon}^{-1}\mathbf{a}_{\epsilon}(i)\,(L_{\mathbf{y}}\mathbf{f})(i)\,\zeta^{i}(\boldsymbol{x}_{\epsilon}(s))ds\right]\;.
\end{align*}
Note that the last integral is bounded by $Ca$ for some constant
$C>0$. Hence, by part (3) of Theorem \ref{t51}, the right-hand side
is bounded by $Ca+o_{\epsilon}(1)$.

Now we turn to the left-hand side. Again by part (3) of Theorem \ref{t51},
we can add small constant $\alpha_{\epsilon}=o_{\epsilon}(1)$ so
that $\widetilde{\phi}_{\epsilon}=\phi_{\epsilon}+\alpha_{\epsilon}\ge0$
on $\mathcal{V}_{\star}$. Then, by the maximum principle, $\widetilde{\phi}_{\epsilon}\ge0$
on $\mathbb{R}^{d}$, and furthermore, $\widetilde{\phi}_{\epsilon}\ge1/2$
on $\mathcal{V}_{\star}\setminus\mathcal{V}_{i}$ provided that $\epsilon$
is sufficiently small. Hence,
\[
\mathbb{E}_{\boldsymbol{x}}^{\epsilon}\left[\phi_{\epsilon}(\boldsymbol{x}_{\epsilon}(a\theta_{\epsilon}\wedge H_{\mathcal{V}_{\star}\setminus\mathcal{V}_{i}}))\right]\ge-\alpha_{\epsilon}+\frac{1}{2}\mathbb{P}_{\boldsymbol{x}}^{\epsilon}\left[H_{\mathcal{V}_{\star}\setminus\mathcal{V}_{i}}<a\theta_{\epsilon}\right]\;.
\]
Summing up, there exists a constant $C>0$ such that
\[
\mathbb{P}_{\boldsymbol{x}}^{\epsilon}\left[H_{\mathcal{V}_{\star}\setminus\mathcal{V}_{i}}<a\theta_{\epsilon}\right]\le Ca+o_{\epsilon}(1)\;,
\]
as desired.
\end{proof}
Now we show that the process $\boldsymbol{x}_{\epsilon}(t)$ does
not spend too much time in $\Delta$ (cf. \eqref{twoset}). Define
the amount of time the rescaled process $\widehat{\boldsymbol{x}}_{\epsilon}(\cdot)$
spends in the set $\Delta$ up to time $t$ as
\[
\widehat{\Delta}(t)=\widehat{\Delta}_{\epsilon}(t)=\int_{0}^{t}\chi_{\Delta}(\widehat{\boldsymbol{x}}_{\epsilon}(s))\,ds\;.
\]

\begin{prop}
\label{p34}For any sequence of Borel probability measures $(\pi_{\epsilon})_{\epsilon>0}$
concentrated on $\mathcal{V}_{\star}$, it holds that
\[
\lim_{\epsilon\rightarrow0}\mathbb{E}_{\pi_{\epsilon}}^{\epsilon}\left[\widehat{\Delta}(t)\right]=0\;\;\text{for all }t\ge0.
\]
\end{prop}

The proof of this proposition can be deduced by combining several
classical results of Freidlin and Wentzell \cite{FW1} in a careful
manner. Since we have to introduce numerous new notations and have
to recall previous results that are not related to the other part
of the current article, we postpone the full proof of this proposition
to the appendix. Here, we only provide the proof of Proposition \ref{p34}
when $\pi_{\epsilon}$ has a density function with respect to the
equilibrium measure $\mu_{\epsilon}$ (cf. \eqref{mu}) for each $\epsilon>0$,
and this density function belongs to $L^{p}(\mu_{\epsilon})$ for
some $p>1$, with a uniform $L^{p}$ bound, i.e.,
\begin{equation}
\limsup_{\epsilon\rightarrow0}\int_{\mathcal{V}_{\star}}\left(\frac{d\pi_{\epsilon}}{d\mu_{\epsilon}}\right)^{p}d\mu_{\epsilon}<\infty\;.\label{clp}
\end{equation}
For this case, we can offer a simple proof.
\begin{proof}[Proof of Proposition \ref{p34} under the assumption \eqref{clp}]
We fix $t\ge0$. Write
\[
u_{\epsilon}(\boldsymbol{x})=\mathbb{E}_{\boldsymbol{x}}^{\epsilon}\left[\widehat{\Delta}(t)\right]\;.
\]
Then, by Fubini's theorem we get
\begin{equation}
\int_{\mathbb{R}^{d}}u_{\epsilon}\,d\mu_{\epsilon}=\mathbb{E}_{\mu_{\epsilon}}^{\epsilon}\left[\int_{0}^{t}\chi_{\Delta}(\widehat{\boldsymbol{x}}_{\epsilon}(s))\,ds\right]=\int_{0}^{t}\mathbb{P}_{\mu_{\epsilon}}^{\epsilon}\left[\widehat{\boldsymbol{x}}_{\epsilon}(s)\in\Delta\right]ds=t\mu_{\epsilon}(\Delta)\;.\label{vt1}
\end{equation}
Write $f_{\epsilon}=\frac{d\pi_{\epsilon}}{d\mu_{\epsilon}}$ so that
we can write
\[
\mathbb{E}_{\pi_{\epsilon}}^{\epsilon}\left[\widehat{\Delta}(t)\right]=\int_{\mathbb{R}^{d}}u_{\epsilon}f_{\epsilon}\,d\mu_{\epsilon}.
\]
Now we apply H\"older's inequality, the bound $u_{\epsilon}\le t$
and \eqref{vt1} to the right-hand side of the previous identity to
deduce
\[
\mathbb{E}_{\pi_{\epsilon}}^{\epsilon}\left[\widehat{\Delta}(t)\right]\le\left[\int_{\mathbb{R}^{d}}u_{\epsilon}\,d\mu_{\epsilon}\right]^{1/q}\left[\int_{\mathbb{R}^{d}}u_{\epsilon}f_{\epsilon}^{p}\,d\mu_{\epsilon}\right]^{1/p}\le t^{1/q}\mu_{\epsilon}(\Delta)^{1/q}\,\left[\int_{\mathbb{R}^{d}}f_{\epsilon}^{p}\,d\mu_{\epsilon}\right]^{1/p}\;,
\]
where $q$ is the conjugate exponent of $p$ satisfying $\frac{1}{p}+\frac{1}{q}=1$.
This, Proposition \ref{p21} and the condition \eqref{clp} complete
the proof.
\end{proof}

\subsection{\label{sec44}Proof of tightness}

For the completeness of the discussion, we start by summarizing well-known
properties related to the current situation. For the full discussion
of this material with the detailed proof, we refer to \cite[Section 7]{LS3}.
Denote by $\{\mathscr{F}_{t}^{0}:t\ge0\}$ the natural filtration
of $C([0,\,\infty),\,\mathbb{R}^{d})$ with respect to $\widehat{\boldsymbol{x}}_{\epsilon}(\cdot)$,
namely,
\[
\mathscr{F}_{t}^{0}=\sigma(\widehat{\boldsymbol{x}}_{\epsilon}(s):s\in[0,\,t])\;.
\]
and define $\{\mathscr{F}_{t}:t\ge0\}$ as the usual augmentation
of $\{\mathscr{F}_{t}^{0}:t\ge0\}$ with respect to $\mathbb{\widehat{P}}_{\pi_{\epsilon}}^{\epsilon}$
where $(\pi_{\epsilon})$ is a sequence of probability measures that
appeared in Theorem \ref{p31}. Define $\mathscr{G}_{t}=\mathscr{F}_{S^{\epsilon}(t)}$
for $t\ge0$, where $S^{\epsilon}$ was defined in \eqref{Se}.
\begin{lem}
\label{lem32} The following statements are true:
\begin{enumerate}
\item For each $u\ge0$, the random time $S^{\epsilon}(u)$ is a stopping
time with respect to the filtration $\{\mathscr{F}_{t}\}$.
\item Let $\tau$ be a stopping time with respect to the filtration $\{\mathscr{G}_{t}\}$.
Then, $S^{\epsilon}(\tau)$ is a stopping time with respect to the
filtration $\{\mathscr{F}_{t}\}$.
\item The process $\{\boldsymbol{y}_{\epsilon}(t):t\ge0\}$ defined in \eqref{e241}
is a continuous-time Markov chain on $\mathcal{V}_{\star}$ with respect
to the filtration $\{\mathscr{G}_{t}\}$.
\end{enumerate}
\end{lem}

\begin{proof}
See \cite[Lemma 7.2 and the paragraph below]{LS3}.
\end{proof}
For $M>0$, define $\mathcal{\mathscr{T}}_{M}$ as the collection
of stopping times with respect to the filtration $\{\mathscr{G}_{t}\}_{t\ge0}$
which is bounded by $M$. The following lemma is required to apply
the Aldous criterion to prove the tightness.
\begin{lem}
\label{p35}For any sequence of Borel probability measures $(\pi_{\epsilon})_{\epsilon>0}$
concentrated on $\mathcal{V}_{\star}$ and for all $M>0$, we have
\[
\lim_{a_{0}\rightarrow0}\limsup_{\epsilon\rightarrow0}\sup_{\tau\in\mathcal{\mathscr{T}}_{M}}\sup_{a\in(0,\,a_{0})}\mathbb{P}_{\pi_{\epsilon}}^{\epsilon}\left[S^{\epsilon}(\tau+a)-S^{\epsilon}(\tau)\ge2a_{0}\right]=0\;.
\]
\end{lem}

\begin{proof}
Since $S^{\epsilon}(\cdot)$ is a generalized inverse of $T^{\epsilon}(\cdot)$,
the set $\{S^{\epsilon}(\tau+a)-S^{\epsilon}(\tau)\ge2a_{0}\}$ is
a subset of
\begin{equation}
\{T^{\epsilon}(S^{\epsilon}(\tau)+2a_{0})-T^{\epsilon}(S^{\epsilon}(\tau))<a\}\;.\label{e350}
\end{equation}
Since $T^{\epsilon}(S^{\epsilon}(\tau)+2a_{0})-T^{\epsilon}(S^{\epsilon}(\tau))$
can be rewritten as
\[
\int_{S^{\epsilon}(\tau)}^{S^{\epsilon}(\tau)+2a_{0}}\chi_{\mathcal{V}_{\star}}(\widehat{\boldsymbol{x}}_{\epsilon}(s))ds=2a_{0}-\int_{S^{\epsilon}(\tau)}^{S^{\epsilon}(\tau)+2a_{0}}\chi_{\Delta}(\widehat{\boldsymbol{x}}_{\epsilon}(s))ds\;,
\]
the set \eqref{e350} is a subset of
\[
\left\{ \int_{S^{\epsilon}(\tau)}^{S^{\epsilon}(\tau)+2a_{0}}\chi_{\Delta}(\widehat{\boldsymbol{x}}_{\epsilon}(s))ds\ge2a_{0}-a\right\} \;.
\]
Therefore, we can replace the probability appeared in the statement
of the lemma with
\[
\mathbb{P}_{\pi_{\epsilon}}^{\epsilon}\left[\int_{S^{\epsilon}(\tau)}^{S^{\epsilon}(\tau)+2a_{0}}\chi_{\Delta}(\widehat{\boldsymbol{x}}_{\epsilon}(s))ds\ge2a_{0}-a\right]\;.
\]
This probability is bounded above by
\begin{equation}
\mathbb{P}_{\pi_{\epsilon}}^{\epsilon}\left[\widehat{\Delta}(2M+2a_{0})\ge2a_{0}-a\right]+\mathbb{P}_{\pi_{\epsilon}}^{\epsilon}\left[S^{\epsilon}(\tau)>2M\right]\;.\label{ke1}
\end{equation}
By Chebyshev's inequality the first term is bounded from above by
\[
\frac{\mathbb{E}_{\pi_{\epsilon}}^{\epsilon}\left[\widehat{\Delta}(2M+2a_{0})\right]}{2a_{0}-a}\;,
\]
and therefore by Proposition \ref{p34} we have
\begin{equation}
\lim_{a_{0}\rightarrow0}\limsup_{\epsilon\rightarrow0}\sup_{\tau\in\mathcal{\mathscr{T}}_{M}}\sup_{a\in(0,\,a_{0})}\mathbb{P}_{\pi_{\epsilon}}^{\epsilon}\left[\widehat{\Delta}(2M+2a_{0})\ge2a_{0}-a\right]=0\;.\label{ke2}
\end{equation}
For the second term of \eqref{ke1}, we observe that $S^{\epsilon}(\tau)>2M$
and $\tau\le M$ imply that $\widehat{\Delta}(2M)\ge M$. Hence again
by Chebyshev's inequality this probability is bounded by $M^{-1}\mathbb{E}_{\pi_{\epsilon}}^{\epsilon}[\widehat{\Delta}(2M)],$
and therefore by Proposition \ref{p34} we have
\begin{equation}
\lim_{a_{0}\rightarrow0}\limsup_{\epsilon\rightarrow0}\sup_{\tau\in\mathcal{\mathscr{T}}_{M}}\sup_{a\in(0,\,a_{0})}\mathbb{P}_{\pi_{\epsilon}}^{\epsilon}\left[S^{\epsilon}(\tau)>2M\right]=0\;.\label{ke3}
\end{equation}
This, \eqref{ke1}, and \eqref{ke2} complete the proof.
\end{proof}
Now we are ready to prove the main tightness result.
\begin{proof}[Proof of Theorem \ref{p31}]
By Aldous' criterion, it suffices to show that, for all $M>0$,
\begin{equation}
\lim_{a_{0}\rightarrow0}\limsup_{\epsilon\rightarrow0}\sup_{\tau\in\mathcal{\mathscr{T}}_{M}}\sup_{a\in(0,\,a_{0})}\mathbb{P}_{\pi_{\epsilon}}^{\epsilon}\left[\mathbf{y}_{\epsilon}(\tau+a)\neq\mathbf{y}_{\epsilon}(\tau)\right]=0\;.\label{e341}
\end{equation}
By Lemma \ref{p35}, it suffices to show that
\[
\lim_{a_{0}\rightarrow0}\limsup_{\epsilon\rightarrow0}\sup_{\tau\in\mathcal{\mathscr{T}}_{M}}\sup_{a\in(0,\,a_{0})}\mathbb{P}_{\pi_{\epsilon}}^{\epsilon}\left[\mathbf{y}_{\epsilon}(\tau+a)\neq\mathbf{y}_{\epsilon}(\tau),\,S^{\epsilon}(\tau+a)-S^{\epsilon}(\tau)\le2a_{0}\right]=0\;.
\]
Since $\mathbf{y}_{\epsilon}(t)=\Psi(\widehat{\boldsymbol{x}}_{\epsilon}(S^{\epsilon}(t)))$,
the last probability can be bounded above by
\[
\mathbb{P}_{\pi_{\epsilon}}^{\epsilon}\left[\Psi(\widehat{\boldsymbol{x}}_{\epsilon}(S^{\epsilon}(\tau)+t)\neq\Psi(\widehat{\boldsymbol{x}}_{\epsilon}(S^{\epsilon}(\tau)))\;\text{for some }t\in(0,\,2a_{0}]\right]\;.
\]
Since $S^{\epsilon}(\tau)$ is a stopping time with respect to the
filtration $\{\mathscr{F}_{t}\}$ by Lemma \ref{lem32}, and since
$\widehat{\boldsymbol{x}}_{\epsilon}(S^{\epsilon}(\tau))\in\mathcal{V}_{\star}$,
the last probability is bounded above by
\[
\sup_{\boldsymbol{y}\in\mathcal{V}_{\star}}\mathbb{P}_{\boldsymbol{y}}^{\epsilon}\left[\Psi(\widehat{\boldsymbol{x}}_{\epsilon}(t))\neq\Psi(\boldsymbol{y})\;\text{for some }t\in(0,\,2a_{0}]\right]=\sup_{i\in S_{\star}}\sup_{\boldsymbol{y}\in\mathcal{V}_{i}}\mathbb{P}_{\boldsymbol{y}}^{\epsilon}\left[H_{\mathcal{V}_{\star}\setminus\mathcal{V}_{i}}\le2a_{0}\theta_{\epsilon}\right]\;.
\]
Thus, the proof of \eqref{e341} is completed by Proposition \ref{p33}.

The assertion $\mathbf{Q}^{*}(\mathbf{x}(0)=i)=1$ is trivial. For
the last assertion of the proposition, it suffices to prove that
\[
\lim_{a_{0}\rightarrow0}\limsup_{\epsilon\rightarrow0}\mathbb{P}_{\pi_{\epsilon}}^{\epsilon}\left[\mathbf{y}_{\epsilon}(t-a)\neq\mathbf{y}_{\epsilon}(t)\text{ for some }a\in(0,\,a_{0})\right]=0\;.
\]
The proof of this estimate is almost identical to that of \eqref{e341}
and is omitted.
\end{proof}

\section{\label{sec6}Proof of Theorem \ref{main}}

We are now ready to prove the main convergence theorem. In view of
the tightness result obtained in Section \ref{sec4}, it is enough
to demonstrate the uniqueness of limit point. The main ingredient
is Theorem \ref{t51}.
\begin{proof}[Proof of Theorem \ref{main}]

Fix $\mathbf{f}\in\mathbb{R}^{S_{\star}}$, and let $\phi_{\epsilon}=\phi_{\epsilon}^{\mathbf{f}}$
be the function obtained in Theorem \ref{t51} for the function $\mathbf{f}$.
Note that the distribution of $\boldsymbol{x}_{\epsilon}(0)$ is concentrated
on a valley $\mathcal{V}_{i}$ for some $i\in S_{\star}$. We fix
$i$ in the proof.

We begin with the observation that
\[
M_{\epsilon}(t)=\phi_{\epsilon}(\widehat{\boldsymbol{x}}_{\epsilon}(t))-\theta_{\epsilon}\int_{0}^{t}(\mathscr{L}_{\epsilon}\phi_{\epsilon})(\widehat{\boldsymbol{x}}_{\epsilon}(s))ds
\]
is a martingale with respect to the filtration $\{\mathscr{F}_{t}\}$
defined in Section \ref{sec4}. Write
\[
M_{\epsilon}(t)=U_{\epsilon}(t)+N_{\epsilon}(t)
\]
where
\begin{align*}
U_{\epsilon}(t) & =\phi_{\epsilon}(\widehat{\boldsymbol{x}}_{\epsilon}(t))-\theta_{\epsilon}\int_{0}^{t}(\mathscr{L}_{\epsilon}\phi_{\epsilon})(\widehat{\boldsymbol{x}}_{\epsilon}(s))\,\mathbf{1}\{\widehat{\boldsymbol{x}}_{\epsilon}(s)\in\mathcal{V}_{\star}\}ds\;,\\
N_{\epsilon}(t) & =-\theta_{\epsilon}\int_{0}^{t}(\mathscr{L}_{\epsilon}\phi_{\epsilon})(\widehat{\boldsymbol{x}}_{\epsilon}(s))\,\mathbf{1}\{\widehat{\boldsymbol{x}}_{\epsilon}(s)\in\Delta\}ds\;.
\end{align*}
Since $\theta_{\epsilon}(\mathscr{L}_{\epsilon}\phi_{\epsilon})(\cdot)$
is bounded function by its construction (cf. \eqref{p321}) and hence
there exists $c>0$ such that
\begin{equation}
|N_{\epsilon}(t)|\le c\widehat{\Delta}(t)\;.\label{ek-1}
\end{equation}
By definition of $\boldsymbol{y}_{\epsilon}(t)$, we can write
\[
U_{\epsilon}(S_{\epsilon}(t))=\phi_{\epsilon}(\boldsymbol{y}_{\epsilon}(t))-\theta_{\epsilon}\int_{0}^{t}(\mathscr{L}_{\epsilon}\phi_{\epsilon})(\boldsymbol{y}_{\epsilon}(s))ds\;,
\]
and hence
\begin{equation}
\widetilde{M}_{\epsilon}(t)=M_{\epsilon}(S_{\epsilon}(t))=\phi_{\epsilon}(\boldsymbol{y}_{\epsilon}(t))-\theta_{\epsilon}\int_{0}^{t}(\mathscr{L}_{\epsilon}\phi_{\epsilon})(\boldsymbol{y}_{\epsilon}(s))ds+N_{\epsilon}(S_{\epsilon}(t))\;.\label{ek0}
\end{equation}
By Lemma \ref{lem32}, $S_{\epsilon}(t)$ is a stopping time with
respect to the filtration $\{\mathscr{F}_{t}\}$, and therefore $\widetilde{M}_{\epsilon}(t)$
is a martingale with respect to $\{\mathscr{G}_{t}\}$. We now investigate
each terms in the expression \eqref{ek0} separately. Recall $\Psi$
from \eqref{proj}. Then, by Theorem \ref{t51}, we can write $\phi_{\epsilon}=\mathbf{f}\circ\Psi+o_{\epsilon}(1)$
on $\mathcal{V}_{\star}$. Since the process $\boldsymbol{y}_{\epsilon}(t)$
takes values in $\mathcal{V}_{\star}$, and by definition $\mathbf{y}_{\epsilon}=\Psi(\boldsymbol{y}_{\epsilon})$,
we have
\begin{equation}
\phi_{\epsilon}(\boldsymbol{y}_{\epsilon}(t))=\mathbf{f}(\Psi(\boldsymbol{y}_{\epsilon}(t)))+o_{\epsilon}(1)=\mathbf{f}(\mathbf{y}_{\epsilon}(t))+o_{\epsilon}(1)\;.\label{ek1}
\end{equation}
Next we consider the second term at the right-hand side of \eqref{ek0}.
Since $\theta_{\epsilon}\mathscr{L}_{\epsilon}\phi_{\epsilon}=(L_{\mathbf{y}}\mathbf{f})\circ\Psi+o_{\epsilon}(1)$
on $\mathcal{V}_{\star}$ by Theorem \ref{t51} and \eqref{aei},
we can write
\begin{equation}
\theta_{\epsilon}\int_{0}^{t}(\mathscr{L}_{\epsilon}\phi_{\epsilon})(\boldsymbol{y}_{\epsilon}(s))ds=\int_{0}^{t}(L_{\mathbf{y}}\mathbf{f})(\mathbf{y}_{\epsilon}(s))ds+o_{\epsilon}(1)\;.\label{ek2}
\end{equation}
Hence, by \eqref{ek0}, \eqref{ek1}, and \eqref{ek2}, we can write
$\widetilde{M}_{\epsilon}(t)$ as
\begin{align}
\widetilde{M}_{\epsilon}(t) & =\mathbf{f}(\mathbf{y}_{\epsilon}(t))-\int_{0}^{t}(L_{\mathbf{y}}\mathbf{f})(\mathbf{y}_{\epsilon}(s))ds+N_{\epsilon}(S_{\epsilon}(t))+o_{\epsilon}(1)\;.\label{ek3}
\end{align}
Recall that $\mathbf{Q}_{\pi_{\epsilon}}^{\epsilon}$ represents the
law of the process $\mathbf{y}_{\epsilon}(\cdot)$ under $\mathbb{P}_{\pi_{\epsilon}}^{\epsilon}$
and let $\mathbf{Q}^{*}$ be a limit point of the family $\{\mathbf{Q}_{\pi_{\epsilon}}^{\epsilon}\}_{\epsilon\in(0,\,1]}$.
Then, by \eqref{ek-1}, \eqref{ek3}, and Proposition \ref{p34},
we can conclude that the process
\[
\widetilde{M}(t)=\mathbf{f}(\mathbf{x}(t))-\int_{0}^{t}(L_{\mathbf{y}}\mathbf{f})(\mathbf{x}(s))ds
\]
is a martingale under $\mathbf{Q}^{*}$. Furthermore, by Theorem \ref{p31},
we have that $\mathbf{Q}^{*}[\mathbf{x}(0)=i]=1$ and $\mathbf{Q}^{*}(\mathbf{x}(t)\neq\mathbf{x}(t-))=0$
for all $t>0$. The only probability measure on $D([0,\,\infty),\,\mathbb{R}^{d})$
satisfying these properties is $\mathbf{Q}_{i}$, and thus we can
conclude that $\mathbf{Q}^{*}=\mathbf{Q}_{i}$. This completes the
characterization of the limit point of the family $\{\mathbf{Q}_{\pi_{\epsilon}}^{\epsilon}\}_{\epsilon\in(0,\,1]}$.
\end{proof}

\appendix

\section{Negligibility of $\widehat{\Delta}$}

In this appendix, we prove Proposition \ref{p34}. The proof relies
solely on the Freidlin-Wentzell theory, and hence our result is not
restricted to the reversible process \eqref{e13}, but also holds
for the general dynamics \eqref{e12} as well. The verification of
this generality is immediate from a careful reading of our proof.

\begin{figure}
\includegraphics[scale=0.19]{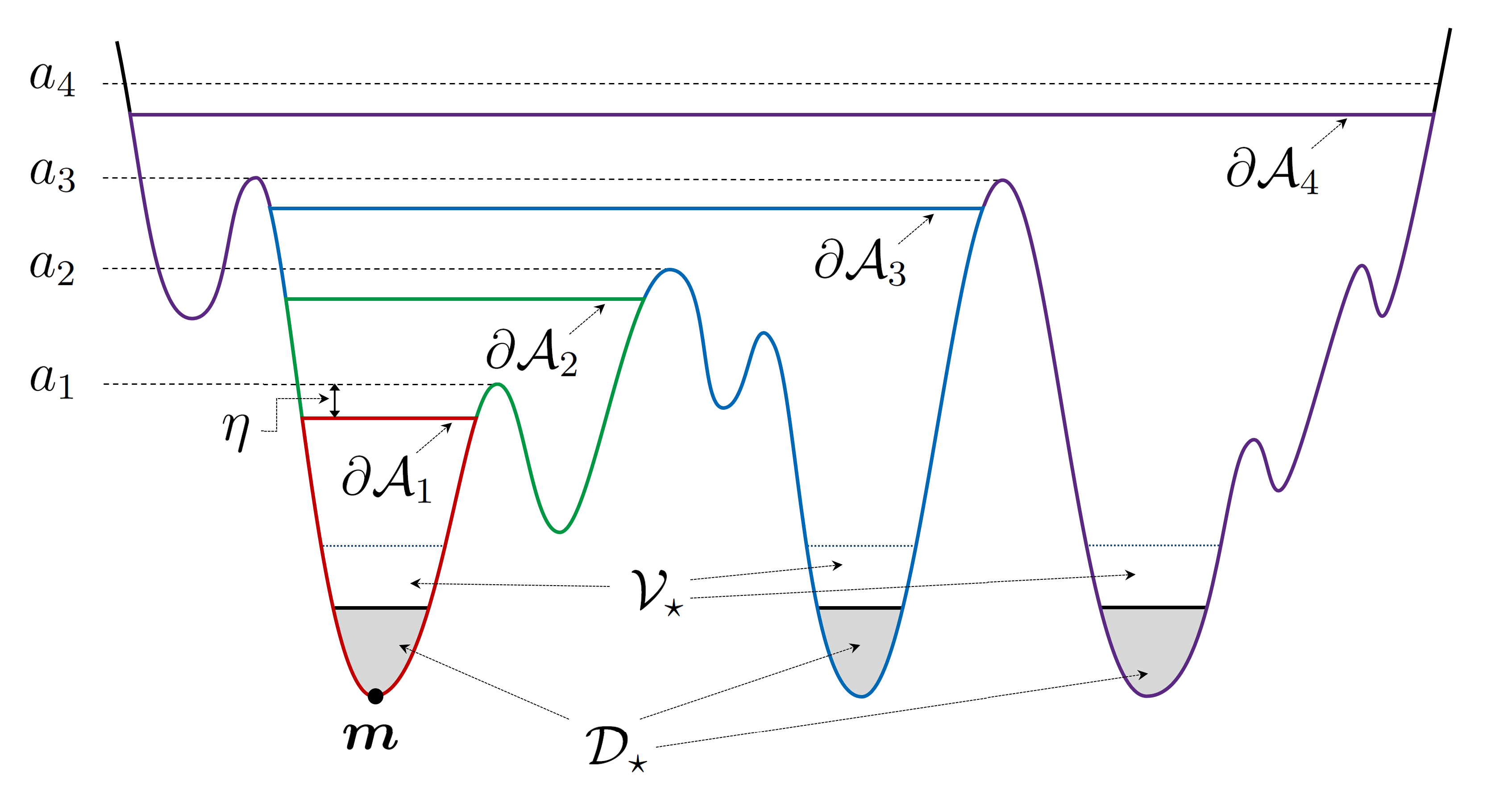}\caption{\label{fig4}Cycle structure associated to $\boldsymbol{m}$: in this
example, $l=3$ so that $a_{3}=H$.}
\end{figure}

\subsection{Notations and idea of proof}

We introduce some additional notations to those in Section \ref{s21}.
Denote by $\mathcal{C}$ the set of critical points of $U$. Let $\eta$
be any sufficiently small number such that
\begin{equation}
\eta<\frac{1}{5}\min\left\{ |U(\boldsymbol{c}')-U(\boldsymbol{c})|:\boldsymbol{c},\,\boldsymbol{c}'\in\mathcal{C}\text{ and }U(\boldsymbol{c}')\neq U(\boldsymbol{c})\right\} \;.\label{mesh}
\end{equation}
In particular, there is no local minima $\boldsymbol{m}$ of $U$
such that $U(\boldsymbol{m})\in(h,\,h+5\eta]$. We write the level
set of $U$ as
\begin{equation}
\mathcal{Q}_{a}:=\{\boldsymbol{x}:U(\boldsymbol{x})<a\}\;\;;\;a\in\mathbb{R}\;.\label{levs}
\end{equation}
For each $\boldsymbol{m}\in\mathcal{M}_{\star}$, define $\mathcal{D}_{\boldsymbol{m}}$
as a connected component of $\mathcal{Q}_{h+\eta}$ containing $\boldsymbol{m}$
and let
\[
\mathcal{D}_{\star}:=\bigcup_{\boldsymbol{m}\in\mathcal{M_{\star}}}\mathcal{D}_{\boldsymbol{m}}\;.
\]
We take $\eta>0$ small enough so that $\mathcal{D}_{\boldsymbol{m}}\subset\mathcal{B}(\boldsymbol{m},\,r_{0})$
(cf. \eqref{ev}) for all $\boldsymbol{m}\in\mathcal{M}_{\star}$.
This implies that $\mathcal{D}_{\star}\subset\mathcal{V}_{\star}$.
From now on we regard $\eta$ as a constant.

\subsubsection*{Strategy of proof}

Define the time spent in the set $\Delta$ (without time-rescaling)
as
\[
\Delta(t)=\Delta_{\epsilon}(t):=\int_{0}^{t}\chi_{\Delta}(\boldsymbol{x}_{\epsilon}(s))\,ds\;.
\]
Then, by a change of variable, we get
\begin{equation}
\widehat{\Delta}(t)=\theta_{\epsilon}^{-1}\Delta(\theta_{\epsilon}t)\;.\label{cov}
\end{equation}
Our main purpose is to estimate $\Delta(t)$ and verify that it is
negligible in the sense of Proposition \ref{p34}. To this end, define
two sequences $(\tau_{i})_{i\in\mathbb{N}}$, $(\sigma_{i})_{i\in\mathbb{N}}$
of hitting times recursively according to the following rules: set
$\tau_{0}=0$, and
\begin{align}
 & \sigma_{i}=\inf\left\{ s>\tau_{i-1}:\boldsymbol{x}_{\epsilon}(s)\in\partial\mathcal{V}_{\star}\right\} \;\;;\;i\ge1\;.\nonumber \\
 & \tau_{i}=\inf\left\{ s>\sigma_{i}:\boldsymbol{x}_{\epsilon}(s)\in\partial\mathcal{D}_{\star}\right\} \;\;;\;i\ge1\;,\label{ap0}
\end{align}
With these notations, we have the following bound on $\Delta(t)$:
\begin{equation}
\Delta(t)\le\sum_{i=1}^{\nu(t)}(\tau_{i}-\sigma_{i})\;,\label{ap1}
\end{equation}
where $\nu(t)=\sup\left\{ n\in\mathbb{N}:\tau_{n}\le t\right\} $.
Hence, for the negligibility of $\Delta(t)$, it suffices to estimate
the term $\tau_{i}-\sigma_{i}$, which measures the length of the
$i$th excursion from $\partial\mathcal{V}_{\star}$ to $\partial\mathcal{D}_{\star}$.
This length is typically short since the drift term $-\nabla U(\boldsymbol{x}_{\epsilon}(t))dt$
pushes the process toward the deeper part of the valley. However,
because of the small random noise, some of these excursions are extraordinarily
long, though such a long excursion is extremely rare. Therefore, in
order to control the right-hand side of \eqref{ap1}, one has to characterize
these long excursions and control both the length and the frequency
of them in a careful manner. This will be carried out in the remaining
part of the appendix.

\subsection{Cyclic structure and Freidlin-Wentzell type estimates}

We introduce a hierarchy structure of the landscape associated to
each global minimum of $U$. Let us fix $\boldsymbol{m}\in\mathcal{M}_{\star}$
throughout this subsection. The constructions below are illustrated
in Figure \ref{fig4}.

For each $a\in\mathbb{R}$, denote by $\mathcal{Q}_{a}(\boldsymbol{m})$
the connected component of the level set $\mathcal{Q}_{a}$ (cf. \eqref{levs})
containing $\boldsymbol{m}$. For $\mathcal{A}\subset\mathbb{R}^{d}$,
we denote by $\mathcal{M}(\mathcal{A})$ the set of local minima of
$U$ contained in $\mathcal{A}$. Then, define an increasing sequence
$(a_{i})_{i=0}^{l+1}$ recursively as follows: set $a_{0}=h+5\eta$
and
\begin{align*}
a_{k+1} & =\inf\left\{ a:\mathcal{M}\left(\mathcal{Q}_{a_{k}}(\boldsymbol{m})\right)\subsetneq\mathcal{M}\left(\mathcal{Q}_{a}(\boldsymbol{m})\right)\right\} \;;\;k\ge0\;.
\end{align*}
If $a_{l}=H$, we stop the recursion procedure and set $a_{l+1}=H+3\eta$.
Now we define
\[
\mathcal{A}_{k}=\mathcal{Q}_{a_{k}-\eta}(\boldsymbol{m})\;\;;\;k\in\llbracket1,\,l+1\rrbracket\;.
\]
By \eqref{mesh}, one can notice that $\mathcal{A}_{k}$ is a connected
set. The sequence of connected sets $\mathcal{A}_{0}\subset\mathcal{A}_{1}\subset\cdots\subset\mathcal{A}_{l+1}$
represents a growing landscape surrounding $\boldsymbol{m}.$ According
to the classical monograph \cite[Chapter 6.6]{FW1}, the set $\mathcal{A}_{k}$
(or $\mathcal{M}(\mathcal{A}_{k})$) corresponds to the rank-$k$
cycle containing $\boldsymbol{m}$. We shall classify each excursions
in \eqref{ap1} by the maximum $k$ such that the corresponding trajectory
hit $\partial\mathcal{A}_{k}$ before arriving at a point in $\partial\mathcal{D}_{\star}$.
Hitting $\partial\mathcal{A}_{k}$ for large $k$ means that we may
have a long excursion.

We define a sequence $(J_{k})_{k=0}^{l+1}$ as
\[
J_{k}=a_{k}-h-5\eta\;.
\]
With the notations introduced above, we are ready to recall several
classical results from \cite{FW1}.
\begin{thm}
\label{FW}There exists $\epsilon_{0}$ such that for all $\epsilon\in(0,\,\epsilon_{0})$
and the followings hold.
\begin{align}
 & \sup_{\boldsymbol{x}\in\mathcal{A}_{k}\setminus\overline{\mathcal{D}_{\star}}}\mathbb{E}_{\boldsymbol{x}}^{\epsilon}\left[H_{\partial\mathcal{A}_{k}\cup\partial\mathcal{D}_{\star}}\right]<\exp\frac{J_{k-1}+\eta}{\epsilon}\;\;\text{for all }k\in\llbracket1,\,l+1\rrbracket\;,\label{fw1}\\
 & \sup_{\boldsymbol{x}\in\mathcal{D}_{\star}}\mathbb{P}_{\boldsymbol{x}}^{\epsilon}\biggl[H_{\partial\mathcal{A}_{k}}<\exp\frac{J_{k}+3\eta}{\epsilon}\biggr]<\frac{1}{64}\;\;\text{for all }k\in\llbracket1,\,l+1\rrbracket\;,\text{ and}\label{fw2}\\
 & \sup_{\boldsymbol{x}\in\partial\mathcal{Q}_{H+\eta}}\mathbb{P}_{\boldsymbol{x}}^{\epsilon}\left[H_{\partial\mathcal{A}_{l+1}}<H_{\partial\mathcal{D}_{\star}}\right]<\frac{1}{8}\;.\label{fw3}
\end{align}
\end{thm}

\begin{rem}
Of course, we can replace constants $1/64$ and $1/8$ appeared in
the statement of theorem with any small positive number. From now
on, $\epsilon_{0}$ always denotes the constant that appeared in this
theorem.
\end{rem}

\begin{proof}
All of these results are consequence of well-known Freidlin-Wentzell
theory. The bound \eqref{fw1} follows from \cite[Theorem 5.3 in Chapter 6]{FW1}
since the deepest possible depth of a valley in $\mathcal{A}_{k+1}$,
which does not contain a global minimum of $\mathcal{M}$ is at most
$J_{k-1}$ by \eqref{mesh}. The bound \eqref{fw2} is a consequence
of \cite[Theorem 6.2 in Chapter 6]{FW1}, since the depth of $\mathcal{A}_{k}$
is $(a_{k}-\eta)-h=J_{k}+4\eta$. Finally, \eqref{fw3} can be deduced
from \cite[Theorem 5.1 in Chapter 6]{FW1}.
\end{proof}
We next present some exponential-type tail estimates that are consequences
of Theorem \ref{FW}. We acknowledge that these estimates are inspired
by \cite[Lemmas B.1 and B.2]{MS}. For the simplicity of notation
we write
\[
\rho_{k}=\exp\biggl(-\frac{J_{k}+2\eta}{\epsilon}\biggr)\;,\ \text{for }k\in\llbracket0,\,\ell\rrbracket\;.
\]

\begin{lem}
\label{pe1}There exists a constant $c_{0}>0$ such that for all $\epsilon\in(0,\,\epsilon_{0})$,
we have
\begin{align}
 & \sup_{\boldsymbol{x}\in\mathcal{A}_{k}\setminus\overline{\mathcal{D}_{\star}}}\mathbb{E}_{\boldsymbol{x}}^{\epsilon}\exp\left(c_{0}\rho_{k-1}H_{\partial\mathcal{A}_{k}\cup\partial\mathcal{D}_{\star}}\right)<2\;\;\text{\ensuremath{\forall}}k\in\llbracket1,\,l+1\rrbracket\text{ and}\label{fw4}\\
 & \sup_{\boldsymbol{x}\in\partial\mathcal{A}_{l+1}}\mathbb{E}_{\boldsymbol{x}}^{\epsilon}\exp\left(c_{0}\rho_{l}H_{\partial\mathcal{D}_{\star}}\right)<4\;.\label{fw5}
\end{align}
\end{lem}

\begin{proof}
For \eqref{fw4}, it suffices to prove that there exists $c>0$ such
that

\begin{equation}
\sup_{\boldsymbol{x}\in\mathcal{A}_{k}\setminus\overline{\mathcal{D}_{\star}}}\mathbb{P}_{\boldsymbol{x}}^{\epsilon}\left[\rho_{k-1}H_{\partial\mathcal{A}_{k}\cup\partial\mathcal{D}_{\star}}>t\right]<\exp\left(-\frac{c}{\epsilon}t\right)\label{pe11}
\end{equation}
for all $t>0$ and for all $\epsilon\in(0,\,\epsilon_{0})$. Write
the left-hand side of the previous inequality as $f(t)$. Then, by
the strong Markov property, Chebyshev's inequality, and \eqref{fw1},
one can deduce that, for $n\in\mathbb{N}$,
\[
f(n)\le f(1)^{n}\le\sup_{\boldsymbol{x}\in\mathcal{A}_{k}\setminus\overline{\mathcal{D}_{\star}}}\left(\rho_{k-1}\mathbb{E}_{\boldsymbol{x}}^{\epsilon}H_{\partial\mathcal{A}_{k}\cup\partial\mathcal{D}_{\star}}\right)^{n}\le\exp\left(-\frac{\eta}{\epsilon}n\right)
\]
provided that $\epsilon$ is sufficiently small. This completes the
proof of \eqref{fw4}.

For \eqref{fw5}, we first claim that there exists $c>0$ such that
\begin{equation}
\sup_{\boldsymbol{x}\in\partial\mathcal{A}_{l+1}}\mathbb{P}_{\boldsymbol{x}}^{\epsilon}\left[H_{\partial\mathcal{Q}_{H+\eta}}>t\right]<\exp\left(-\frac{ct}{\epsilon}\right)\;\;\text{for all }\epsilon\in(0,\,\epsilon_{0})\;.\label{ss1}
\end{equation}
The proof is identical to \cite[Proof of Lemma B.2]{MS} and we will
omit the detail. The main ingredient of the proof therein is the fact
that for any trajectory $\phi:[0,\,t]\rightarrow\mathbb{R}^{d}$ such
that $\phi(s)\in\mathcal{Q}_{H+\eta}^{c}$ for all $s\in[0,\,t]$
must satisfy
\begin{equation}
\int_{0}^{t}|\dot{\phi}(s)+\nabla U(\phi(s))|^{2}ds\ge ct\label{pee1}
\end{equation}
for some $c>0$. This follows mainly because there is no critical
point of $U$ in $\mathcal{Q}_{H+\eta}^{c}$. Then, \eqref{ss1} is
immediate from \eqref{pee1} through Schilder's classical large deviation
theorem.

Now we define two sequences of hitting times $(\pi_{i})_{i=0}^{\infty}$,
$(\zeta_{i})_{i=1}^{\infty}$ recursively as, $\pi_{0}=0$ and
\begin{align*}
 & \zeta_{i}=\inf\left\{ s>\pi_{i-1}:\boldsymbol{x}_{\epsilon}(s)\in\partial\mathcal{Q}_{H+\eta}\right\} \;\;;\;i\ge1\;,\\
 & \pi_{i}=\inf\left\{ s>\zeta_{i}:\boldsymbol{x}_{\epsilon}(s)\in\partial\mathcal{A}_{l+1}\text{ or }\partial\mathcal{D}_{\star}\right\} \;\;;\;i\ge1\;.
\end{align*}
Let $N=\inf\{n:\boldsymbol{x}_{\epsilon}(\pi_{n})\in\partial\mathcal{D}_{\star}\}$.
Then, we can write
\begin{equation}
H_{\partial\mathcal{D}_{\star}}=\pi_{N}=\sum_{i=1}^{N}(\zeta_{i}-\pi_{i-1})+\sum_{i=0}^{N}(\pi_{i}-\zeta_{i})\;.\label{pem2}
\end{equation}
Then, by H\"older's inequality,
\begin{equation}
\begin{aligned} & \mathbb{E}_{\boldsymbol{x}}^{\epsilon}\exp\left(c\rho_{l}H_{\partial\mathcal{D}_{\star}}\right)\\
 & =\sum_{n=1}^{\infty}\mathbb{E}_{\boldsymbol{x}}^{\epsilon}\biggl[\exp\biggl\{ c\rho_{l}\biggl(\sum_{i=1}^{N}(\zeta_{i}-\pi_{i-1})+\sum_{i=0}^{N}(\pi_{i}-\zeta_{i})\biggr)\biggr\}\mathbf{1}\{N=n\}\biggr]\\
 & \le\sum_{n=1}^{\infty}\left[\mathbb{E}_{\boldsymbol{x}}^{\epsilon}\exp\biggl\{3c\rho_{l}\sum_{i=0}^{n}(\pi_{i}-\zeta_{i})\biggr\}\right]^{\frac{1}{3}}\left[\mathbb{E}_{\boldsymbol{x}}^{\epsilon}\exp\biggl\{3c\rho_{l}\sum_{i=0}^{n}(\zeta_{i}-\pi_{i-1})\biggr\}\right]^{\frac{1}{3}}\mathbb{P}_{\boldsymbol{x}}^{\epsilon}\left[N=n\right]^{\frac{1}{3}}\;.
\end{aligned}
\label{epm0}
\end{equation}
Now we consider the terms appeared in the last line separately. By
the strong Markov property, \eqref{ss1} and the first part of the
current lemma with $k=l+1$, we get
\begin{align}
 & \mathbb{E}_{\boldsymbol{x}}^{\epsilon}\exp\Big(3c\rho_{l}\sum_{i=0}^{n}(\zeta_{i}-\pi_{i-1})\Big)\le\sup_{\boldsymbol{y}\in\partial\mathcal{A}_{l+1}}\biggl[\mathbb{E}_{\boldsymbol{x}}^{\epsilon}\exp\Big(3c\rho_{l}H_{\partial\mathcal{Q}_{H+\eta}}\Big)\biggr]^{n}<2^{n}\;\;\mbox{and}\label{pem3}\\
 & \mathbb{E}_{\boldsymbol{x}}^{\epsilon}\exp\Big(\frac{c\rho_{l}}{3}\sum_{i=0}^{n}(\pi_{i}-\zeta_{i})\Big)\le\sup_{\boldsymbol{y}\in\partial\mathcal{Q}_{H+\eta}}\biggl[\mathbb{E}_{\boldsymbol{x}}^{\epsilon}\exp\Big(\frac{c\rho_{l}}{3}H_{\partial\mathcal{A}_{l+1}\cup\partial\mathcal{D}_{\star}}\Big)\biggr]^{n}<2^{n}\;,\label{pem4}
\end{align}
for all small enough $c$ and $\epsilon\in(0,\,\epsilon_{0})$. On
the other hand, the strong Markov property and \eqref{fw3} implies
that
\begin{equation}
\sup_{\boldsymbol{x}\in\partial\mathcal{A}_{l+1}}\mathbb{P}_{\boldsymbol{x}}^{\epsilon}\left[N=n\right]<\frac{1}{8^{n-1}}\;\;;\;n\ge1\;.\label{pem1}
\end{equation}
Now applying \eqref{pem3}, \eqref{pem4}, and \eqref{pem1} to \eqref{epm0}
finally yields
\[
\sup_{\boldsymbol{x}\in\partial\mathcal{A}_{l+1}}\mathbb{E}_{\boldsymbol{x}}^{\epsilon}\exp\left(c\rho_{l}H_{\partial\mathcal{D}_{\star}}\right)<\sum_{n=1}^{\infty}4^{\frac{n}{3}}\frac{1}{8^{\frac{n-1}{3}}}\le4\;.
\]
\end{proof}

\subsection{Proof of Proposition \ref{p34}}

The main ingredient to prove Proposition \ref{p34} is the following
exponential tail estimate for $\Delta(t)$.
\begin{lem}
\label{kc}For any $\upsilon\in(0,\,1)$, there exist constants $C_{1}$,
$C_{2}$, $\epsilon_{1}(\upsilon)>0$ such that,
\begin{equation}
\sup_{\boldsymbol{x}\in\mathcal{V}_{\star}}\mathbb{P}_{\boldsymbol{x}}^{\epsilon}\left[\Delta(t)>\alpha t\right]\le C_{1}\exp\left\{ -C_{2}(\alpha-\upsilon)\rho_{l}t\right\} \;,\label{mest}
\end{equation}
for all $\alpha\in(\upsilon,\,1)$, $\epsilon\in(0,\,\epsilon_{1}(\upsilon))$,
and $t>0$.
\end{lem}

Before proving this proposition, we show how it implies Proposition
\ref{p34}.
\begin{proof}[Proof of Proposition \ref{p34}]
Fix $\upsilon\in(0,\,1)$. Then, by Lemma \ref{kc}, for all $\epsilon\in(0,\,\epsilon_{1}(\upsilon))$,
$\boldsymbol{x}\in\mathcal{V}_{\star}$ and $t>0$, we obtain
\[
\mathbb{E}_{\boldsymbol{x}}^{\epsilon}\left[\frac{\Delta(t)}{t}\right]=\int_{0}^{1}\mathbb{P}_{\boldsymbol{x}}^{\epsilon}\left[\frac{\Delta(t)}{t}>\alpha\right]d\alpha\le\upsilon+\int_{\upsilon}^{\infty}C_{1}\exp\left\{ -C_{2}(\alpha-\upsilon)\rho_{l}t\right\} d\alpha=\upsilon+\frac{C}{t\rho_{l}}\;.
\]
Therefore, by \eqref{cov}, we have
\[
\mathbb{E}_{\boldsymbol{x}}^{\epsilon}\left[\widehat{\Delta}(t)\right]\le\upsilon t+\frac{C}{\theta_{\epsilon}\rho_{l}}=\upsilon t+C\exp\left(-\frac{H-h-3\eta}{\epsilon}\right)\;.
\]
Hence,
\[
\limsup_{\epsilon\rightarrow0}\sup_{\boldsymbol{x}\in\mathcal{V}_{\star}}\mathbb{E}_{\boldsymbol{x}}^{\epsilon}\left[\widehat{\Delta}(t)\right]\le\upsilon t\;.
\]
The proof is now completed by letting $\upsilon\rightarrow0$.
\end{proof}
Now we turn to the proof of Lemma \ref{kc}. Let us fix $\boldsymbol{x}\in\mathcal{B}(\boldsymbol{m},\,r_{0})$
for some $\boldsymbol{m}\in\mathcal{M}_{\star}$, and recall the cycle
structure $\mathcal{A}_{0}\subset\cdots\subset\mathcal{A}_{l+1}$
associated to $\boldsymbol{m}$. Recall the sequences of hitting times
$(\sigma_{i})$ and $(\tau_{i})$ from \eqref{ap0}. For each $i$,
we define a sequence of hitting times $\sigma_{i}=\tau_{i}^{(0)}\le\tau_{i}^{(1)}\le\cdots\le\tau_{i}^{(l+2)}=\tau_{i}$
recursively as,
\begin{align*}
 & \tau_{i}^{(k)}=\inf\{s\ge\tau_{i}^{(k-1)}:\boldsymbol{x}_{\epsilon}(s)\in\partial\mathcal{D}_{\star}\cup\partial\mathcal{A}_{k}\}\;\;;\;k\in\llbracket1,\,l+1\rrbracket\;,\\
 & \tau_{i}^{(l+2)}=\inf\{s\ge\tau_{i}^{(l+1)}:\boldsymbol{x}_{\epsilon}(s)\in\partial\mathcal{D}_{\star}\}\;.
\end{align*}
Now we write
\begin{equation}
\Delta^{(k)}(t)=\sum_{i=1}^{\nu(t)}(\tau_{i}^{(k+1)}-\tau_{i}^{(k)})\;\;;\;k\in\llbracket0,\,l+1\rrbracket\;.\label{lamk}
\end{equation}
With these notations, it suffices to prove the following lemma. For
convenience, we set $\rho_{l+1}:=\rho_{l}$.
\begin{lem}
\label{kcs}For all $k\in\llbracket0,\,l+1\rrbracket$ and $\upsilon\in(0,\,1)$,
there exist constants $C_{1}$, $C_{2}$ and $\epsilon_{1}=\epsilon_{1}(\upsilon)$
such that,
\[
\mathbb{P}_{\boldsymbol{x}}^{\epsilon}\left[\Delta^{(k)}(t)>\alpha t\right]\le C_{1}\exp\left\{ -C_{2}(\alpha-\upsilon)\rho_{k}t\right\} \;,
\]
for all $\alpha\in(\upsilon,\,1)$, $\epsilon\in(0,\,\epsilon_{0})$,
and $t>0$.
\end{lem}

\begin{proof}
We fix $k\in\llbracket0,\,l+1\rrbracket$. Observe first that $\tau_{i}^{(k+1)}-\tau_{i}^{(k)}\neq0$
if and only if $\boldsymbol{x}_{\epsilon}(\tau_{i}^{(k)})\in\partial\mathcal{A}_{k}$.
Denote by $\{i_{1},\,i_{2},\,\cdots\}$ the (random) set of $i$ such
that $\boldsymbol{x}_{\epsilon}(\tau_{i}^{(k)})\in\partial\mathcal{A}_{k}$,
and write $\nu^{(k)}(t)=\sup\{i:\tau_{i}^{(k)}\le t\}$. With these
notations, we can write
\[
\Delta^{(k)}(t)=\sum_{m=1}^{\nu^{(k)}(t)}(\tau_{i_{m}}^{(k+1)}-\tau_{i_{m}}^{(k)})\;.
\]
Then, by Chebyshev's inequality and Cauchy-Schwarz's inequality, we
obtain
\begin{align*}
\mathbb{P}_{\boldsymbol{x}}^{\epsilon}\left[\Delta^{(k)}(t)>\alpha t\right] & \le e^{-\lambda\alpha\rho_{k}t}\,\sum_{n=0}^{\infty}\mathbb{E}_{\boldsymbol{x}}^{\epsilon}\biggl[\exp\biggl\{\lambda\rho_{k}\sum_{m=1}^{\nu^{(k)}(t)}(\tau_{i_{m}}^{(k+1)}-\tau_{i_{m}}^{(k)})\biggr\}\mathbf{1}\left\{ \nu^{(k)}(t)=n\right\} \biggr]\\
 & \le e^{-\lambda\alpha\rho_{k}t}\,\sum_{n=0}^{\infty}\mathbb{E}_{\boldsymbol{x}}^{\epsilon}\biggl[\exp\biggl\{2\lambda\rho_{k}\sum_{m=1}^{n}(\tau_{i_{m}}^{(k+1)}-\tau_{i_{m}}^{(k)})\biggr\}\biggr]^{\frac{1}{2}}\mathbb{P}_{\boldsymbol{x}}^{\epsilon}\left[\nu^{(k)}(t)=n\right]^{\frac{1}{2}}\;.
\end{align*}
Now let $\lambda=c_{0}/2$ be the half of the constant that appeared
in Lemma \ref{pe1}. By the strong Markov property and Lemma \ref{pe1}
(we use \eqref{fw4} for $k\le l$ and \eqref{fw5} for $k=l+1$),
\[
\mathbb{E}_{\boldsymbol{x}}^{\epsilon}\biggl[\exp\biggl\{2\lambda\rho_{k}\sum_{m=1}^{n}(\tau_{i_{m}}^{(k+1)}-\tau_{i_{m}}^{(k)})\biggr\}\biggr]^{\frac{1}{2}}\le\sup_{\boldsymbol{y}\in\partial\mathcal{A}_{k}}\mathbb{E}_{\boldsymbol{y}}^{\epsilon}\left[\exp\left\{ 2\lambda\rho_{k}H_{\partial\mathcal{A}_{k+1}\cup\partial\mathcal{D}_{\star}}\right\} \right]^{\frac{n}{2}}<2^{\frac{n}{2}}\;.
\]
Summing up, we get
\begin{equation}
\mathbb{P}_{\boldsymbol{x}}^{\epsilon}\left[\Delta^{(k)}(t)>\alpha t\right]\le e^{-\lambda\alpha\rho_{k}t}\,\sum_{n=0}^{\infty}2^{\frac{n}{2}}\mathbb{P}_{\boldsymbol{x}}^{\epsilon}\left[\nu^{(k)}(t)=n\right]^{\frac{1}{2}}\;.\label{conn0}
\end{equation}

Now we estimate the probability $\mathbb{P}_{\boldsymbol{x}}^{\epsilon}\left[\nu^{(k)}(t)=n\right]$.
Fix $\upsilon>0$ and suppose that $n>\upsilon\rho_{k}t$. Conditioned
on the event $\{\nu^{(k)}(t)=n\}$, consider $n-1$ disjoint sub-intervals
of $[0,\,t]$:
\begin{equation}
[\sigma_{i_{1}},\,\tau_{i_{1}}^{(k)}],\,[\sigma_{i_{2}},\,\tau_{i_{2}}^{(k)}],\cdots,\,[\sigma_{i_{n-1}},\,\tau_{i_{n-1}}^{(k)}]\;.\label{esj}
\end{equation}
Note that the last interval $[\sigma_{i_{n}},\,\tau_{i_{n}}^{(k)}]$
is excluded since it is possible that $\tau_{i_{n}}^{(k)}>t$. Then,
since $n>\upsilon\rho_{k}t$, we can find $(n-1)/2$ intervals among
\eqref{esj} that have length at most $2/(\upsilon\rho_{k})$. Hence,
by the strong Markov property and \eqref{fw2}, there exists $\epsilon_{1}(\upsilon)>0$
such that
\begin{align*}
\mathbb{P}_{\boldsymbol{x}}^{\epsilon}\left[\nu^{(k)}(t)=n\right] & \le\sum_{S\subset\{i_{1},\,i_{2},\,\cdots,\,i_{n}\},\;|S|=\frac{n-1}{2}}\mathbb{P}_{\boldsymbol{x}}^{\epsilon}\biggl[\tau_{i}^{(k)}-\sigma_{i}\le\frac{2}{\upsilon\rho_{k+1}}\;\forall i\in S\biggr]\\
 & \le{n \choose (n-1)/2}\sup_{\boldsymbol{y}\in\mathcal{\mathcal{D}}_{\star}}\mathbb{P}_{\boldsymbol{y}}^{\epsilon}\biggl[H_{\partial\mathcal{A}_{k}}\le\frac{2}{\upsilon\rho_{k+1}}\biggr]^{\frac{n-1}{2}}\le{n \choose (n-1)/2}\frac{1}{8^{n-1}}\le\frac{1}{4^{n-1}}
\end{align*}
for all $\epsilon\in(0,\,\epsilon_{1}(\upsilon))$ and $n>\upsilon\rho_{k+1}t$.
Combining this computation with \eqref{conn0}, we get
\[
\mathbb{P}_{\boldsymbol{x}}^{\epsilon}\left[\Delta^{(k)}(t)>\alpha t\right]\le e^{-\alpha\beta\rho_{k+1}t}\,\biggl[\sum_{n=0}^{\upsilon\rho_{k+1}t}2^{\frac{n}{2}}+\sum_{n=\upsilon\rho_{k+1}t+1}^{\infty}2^{\frac{n}{2}}\frac{1}{4^{n-1}}\biggr]\le Ce^{-\alpha\beta\rho_{k+1}t}e^{\upsilon\rho_{k+1}t}\;.
\]
This completes the proof.
\end{proof}
\begin{acknowledgement*}
We wish to thank two anonymous reviewers  for directing our attention to several references, and valuable comments which have improved our manuscript.
I. Seo was supported by the National Research Foundation of Korea
(NRF) grant funded by the Korea government(MSIT) (No. 2018R1C1B6006896)
and (No. 2017R1A5A1015626). F. Rezakhanlou was supported in part by
NSF grant DMS-1407723.
\end{acknowledgement*}


\begin{thebibliography}{10}
\bibitem{AGL}I. Armend\'{a}riz, S. Grosskinsky, M. Loulakis: Metastability
in a condensing zero-range process in the thermodynamic limit. Probab.
Theory Related Fields \textbf{169}, 105--175 (2017)

\bibitem{BL1}J. Beltr\'{a}n, C. Landim: Tunneling and metastability
of continuous time Markov chains. J. Stat. Phys. \textbf{140}, 1065--1114
(2010)

\bibitem{BL2}J. Beltr\'{a}n, C. Landim: Tunneling and metastability
of continuous time Markov chains II. J. Stat. Phys. \textbf{149},
598--618 (2012)

\bibitem{BL3}J. Beltr\'{a}n, C. Landim: Metastability of reversible
condensed zero range processes on a finite set. Probab. Theory Related
Fields \textbf{152}, 781--807 (2012)

\bibitem{newBerg} N. Berglund: Kramers' law: validity, derivations
and generalisations. Markov Process. Related Fields \textbf{19}, 459-490
(2013).

\bibitem{newBDW}N. Berglund, G. Di Ges\`{u}, H. Weber: An Eyring-Kramers
law for the stochastic Allen-Cahn equation in dimension two. Electron.
J. Probab. 22 (2017), Paper No. 41.

\bibitem{BDG}A. Bianchi, S. Dommers, and C. Giardinà: Metastability
in the reversible inclusion process. Electron. J. Probab. \textbf{22}
(2017)

\bibitem{BR} F. Bouchet, J. Reygner: Generalisation of the Eyring-Kramers
transition rate formula to irreversible diffusion processes. J. Ann.
Henri Poincar\'{e} \textbf{17}, 3499--3532 (2016)

\bibitem{BEGK1}A. Bovier, M. Eckhoff, V. Gayrard, M. Klein: Metastability
in stochastic dynamics of disordered mean-field models. Probab. Theory
Relat. Fields \textbf{119}, 99--161 (2001)

\bibitem{BEGK2}A. Bovier, M. Eckhoff, V. Gayrard, M. Klein: Metastability
in reversible diffusion process I. Sharp asymptotics for capacities
and exit times. J. Eur. Math. Soc. \textbf{6}, 399--424 (2004)

\bibitem{newBGK}A. Bovier, V. Gayrard, M. Klein: Metastability in
reversible diffusion processes II. Precise asymptotics for small eigenvalues.
J. Eur. Math. Soc. \textbf{7}, 69-99 (2005)

\bibitem{BdH}A. Bovier, F. den Hollander: \textsl{Metastability:
a potential-theoretic approach}. Grundlehren der mathematischen Wissenschaften
\textbf{351}, Springer, Berlin, 2015.

\bibitem{new DLLN review}G. Di Ges\`{u}, T. Leli\`{e}vre, D. Le
Peutrec and B. Nectoux, Jump Markov models and transition state theory:
the Quasi-Stationary Distribution approach, Faraday Discussion, 195,
469-495, (2016).

\bibitem{new DLLN2}G. Di Ges\`{u}, T. Leli\`{e}vre, D. Le Peutrec
and B. Nectoux, Sharp asymptotics of the first exit point density.
Ann. PDE 5, 5 (2019). https://doi.org/10.1007/s40818-019-0059-2.

\bibitem{ET}C. Evans, P. Tabrizian: Asymptotics for scaled Kramers-Smoluchoswski
equations. Siam J. Math. Anal. \textbf{48,} 2944--2961 (2016)

\bibitem{Ey}H. Eyring: The activated complex in chemical reactions.
J. Chem. Phys. \textbf{3}, 107--115 (1935)

\bibitem{FW1}M. I. Freidlin, A. D. Wentzell: On small random perturbation
of dynamical systems, Usp. Math. Nauk \textbf{25} (1970) {[}English
transl., Russ. Math. Surv. 25 (1970){]}

\bibitem{FW2}M. I. Freidlin, A. D. Wentzell: Some problems concerning
stability under small random perturbations, Theory Pro. Appl. \textbf{17}
(1972)

\bibitem{FW3}M. I. Freidlin, A. D. Wentzell: \textit{Random Perturbations.
In: Random Perturbations of Dynamical Systems.} Grundlehren der mathematischen
Wissenschaften \textbf{260}. Springer, New York, NY, 1998

\bibitem{GT}D. Gilbarg and N. Trudinger: \textit{Elliptic Partial
Differential Equations of Second Order}. 2nd ed, Springer, 1983.

\bibitem{GL}A. Gaudillière, C. Landim: A Dirichlet principle for
non reversible Markov chains and some recurrence theorems. Probab.
Theory Related Fields \textbf{158}, 55--89 (2014)

\bibitem{GRV2}S. Grosskinsky, F. Redig and K. Vafayi: Dynamics of
condensation in the symmetric inclusion process. Electron. J. Probab.
\textbf{18}, 1--23 (2013)

\bibitem{newHKN1}B. Helffer, M. Klein, F. Nier: Quantitative analysis
of metastability in reversible diffusion processes via a Witten complex
approach. Mat. Contemp. \textbf{26} 41-85 (2005).

\bibitem{newHKN2}B. Helffer, M. Klein, F. Nier: Hypoelliptic Estimates
and Spectral Theory for Fokker-Planck Operators and Witten Laplacians.
Lecture Notes in Math. 1862. Springer, Berlin (2005).

\bibitem{newHKN3}B. Helffer, M. Klein, F. Nier: Quantitative analysis
of metastability in reversible diffusion processes via a Witten complex
approach: The case with boundary. M\'{e}m. Soc. Math. Fr. (N.S.)
105 vi+89 (2006).

\bibitem{Kim-Seo1}S. Kim, I. Seo: Metastability of stochastic Ising
and Potts model on lattice withou external fields. Submitted (2021).

\bibitem{Kim-Seo2}S. Kim, I. Seo: Condensation and metastable behavior
of non-reversible inclusion processes. To appear in Commun. Math.
Phys. (2020).

\bibitem{Kra}H. A. Kramers: Brownian motion in a field of force and
the diusion model of chemical reactions. Physica \textbf{7}, 284--304
(1940)

\bibitem{Lan1}C. Landim: A topology for limits of Markov chains.
Stoch. Proc. Appl. \textbf{125}, 1058--1098 (2014)

\bibitem{Lan2}C. Landim: Metastability for a Non-reversible Dynamics:
The Evolution of the Condensate in Totally Asymmetric Zero Range Processes.
Commun. Math. Phys. \textbf{330}, 1--32 (2014)

\bibitem{Lan3}C. Landim, Personal communication.

\bibitem{LLM}C. Landim, M. Loulakis, M. Mourragui: Metastable Markov
chains. Electron. J. Probab. \textbf{23}, 1-34 (2018)

\bibitem{newLMS1}C. Landim, D. Marcondes, I. Seo: Metastable behavior
of reversible, critical zero-range processes. Submitted. (2020)

\bibitem{LMS}C. Landim, M. Mariani, I. Seo:. A Dirichlet and a Thomson
principle for non-selfadjoint elliptic operators, Metastability in
non-reversible diffusion processes. Arch. Rational Mech. Anal., forthcoming.
(2017)

\bibitem{LMT}C. Landim, R. Misturini, K. Tsunoda: Metastability of
reversible random walks in potential field. J. Stat. Phys. \textbf{160},
1449--1482 (2015)

\bibitem{LS1}C. Landim, I. Seo: Metastability of non-reversible random
walks in a potential field, the Eyring-Kramers transition rate formula.
Comm. Pure. Appl. Math. \textbf{71} 203--266 (2018)

\bibitem{LS2}C. Landim, I. Seo: Metastability of non-reversible mean-field
Potts model with three spins. J. Stat. Phys. \textbf{165}, 693--726
(2016)

\bibitem{LS3}C. Landim, I. Seo: Metastability of one-dimensional,
non-reversible diffusions with periodic boundary conditions. Ann.
Henri Poincar\'{e} (B) Probability and Statistics, forthcoming. (2017)

\bibitem{newLS1}J. Lee, I. Seo: Non-reversible metastable diffusions
with Gibbs invariant measure I: Eyring-Kramers formula. Submitted.

\bibitem{newLS2}J. Lee, I. Seo: Non-reversible metastable diffusions
with Gibbs invariant measure II: Markov chain convergence. Submitted.

\bibitem{newLLN}T. Lelievre, D. Le Peutrec, B. Nectoux: Exit event
from a metastable state and Eyring-Kramers law for the overdamped
Langevin dynamics. Stochastic dynamics out of equilibrium, 331-363,
Springer Proc. Math. Stat., \textbf{282}, Springer, Cham, 2019.

\bibitem{MS}F. Martinelli, E. Scoppola: Small random perturbation
of dynamical systems: exponential loss of memory of the initial condition.
Commun. Math. Phys. \textbf{120}, 25--69 (1988)

\bibitem{newMS}G. Menz, A. Schlichting: Poincar\'{e} and logarithmic
Sobolev inequalities by decomposition of the energy landscape. Ann.
Probab. \textbf{42}, no. 5, 1809-1884 (2014).

\bibitem{NZ}F. R. Nardi, A. Zocca: Tunneling behavior of Ising and
Potts models on grid graphs. arXiv:1708.09677

\bibitem{Su1} M. Sugiura: Metastable behaviors of diffusion processes
with small parameter. J. Math. Soc. Japan \textbf{47}, 755--788 (1995)

\bibitem{ST}I. Seo, P. Tabrizian: Asymptotics for scaled Kramers-Smoluchowski
equation in several dimensions with general potentials. Submitted.
arXiv:1808.09108 (2018)

\bibitem{S}I. Seo: Condensation of non-reversible zero-range processes.
Commun. Math. Phys., \textbf{366}, 781-839 (2019)
\end{thebibliography}
\end{document}